\documentclass[onefignum,onetabnum]{siamart220329}

\usepackage[showonlyrefs]{mathtools}
\usepackage{lipsum}
\usepackage{amsfonts}
\usepackage{graphicx}
\usepackage{epstopdf}
\usepackage{algorithmic}

\ifpdf
  \DeclareGraphicsExtensions{.eps,.pdf,.png,.jpg}
\else
  \DeclareGraphicsExtensions{.eps}
\fi

\newsiamremark{remark}{Remark}
\newsiamremark{hypothesis}{Hypothesis}
\crefname{hypothesis}{Hypothesis}{Hypotheses}
\newsiamthm{claim}{Claim}

\headers{Extending linear response with synthetic forcings}{R. Spacek and G. Stoltz}

\title{Extending the regime of linear response with synthetic forcings
}

\author{Renato Spacek 
		\thanks{CERMICS (\'Ecole des Ponts), Marne-la-Vall\'ee, France \& MATHERIALS team, Inria Paris, France (\email{renato.spacek@enpc.fr}, \email{gabriel.stoltz@enpc.fr})}
\and Gabriel Stoltz \footnotemark[2]}

\usepackage{amsopn}

\usepackage{autonum}
\usepackage{amssymb}
\usepackage{mathrsfs}
\usepackage{subcaption}
\usepackage{multirow}
\usepackage{booktabs}

\newcommand{\T}{\mathbb{T}}
\newcommand{\Z}{\mathbb{Z}}
\newcommand{\E}{\mathbb{E}}
\newcommand{\R}{\mathbb{R}}

\newcommand{\N}{\mathbb{N}}
\newcommand{\eps}{\varepsilon}
\renewcommand{\div}{\operatorname{div}}
\DeclareMathOperator*{\argmin}{arg\,min}
\DeclareMathOperator*{\argmax}{arg\,max}
\newcommand{\e}{\mathrm{e}}
\newcommand{\f}{\mathfrak{f}}
\newcommand{\ind}{\mathbf{1}}
\renewcommand{\L}{\mathcal{L}}
\newcommand{\Lx}{\widetilde{\L}_\mathrm{extra}}
\newcommand{\Lp}{\widetilde{\L}_\mathrm{phys}}
\newcommand{\Lt}{\widetilde{\L}}
\newcommand{\Lr}{\L_0}
\renewcommand{\S}{\mathscr{S}}
\newcommand{\bigO}{\mathrm{O}}
\newcommand\dpsi[2]{[\Psi^h_{#1}]_{#2}}
\newcommand\ddpsi[2]{[\partial\Psi^h_{#1}]_{#2}}
\newcommand\dfxn[2]{[#1]_{#2}}
\newcommand\sn[2]{{#1}\times 10^{#2}}
\renewcommand{\geq}{\geqslant}
\renewcommand{\leq}{\leqslant}

\usepackage{enumitem}
\setenumerate[1]{label={(\arabic*)}}
\setlist[enumerate]{itemsep=0.75em, topsep=0.75em}

\def\code#1{\texttt{#1}}

\newtheorem{assumption}{Assumption}
\newsiamremark{example}{Example}

\ifpdf
\hypersetup{
  pdftitle={},
  pdfauthor={R. Spacek, G. Stoltz}
}
\fi

\begin{document}

\maketitle

\begin{abstract}
Transport coefficients, such as the mobility, thermal conductivity and shear viscosity, are quantities of prime interest in statistical physics. At the macroscopic level, transport coefficients relate an external forcing of magnitude $\eta$, with $\eta \ll 1$, acting on the system to an average response expressed through some steady-state flux. In practice, steady-state averages involved in the linear response are computed as time averages over a realization of some stochastic differential equation. Variance reduction techniques are of paramount interest in this context, as the linear response is scaled by a factor of $1/\eta$, leading to large statistical error. One way to limit the increase in the variance is to allow for larger values of $\eta$ by increasing the range of values of the forcing for which the nonlinear part of the response is sufficiently small. In theory, one can add an extra forcing to the physical perturbation of the system, called synthetic forcing, as long as this extra forcing preserves the invariant measure of the reference system. The aim is to find synthetic perturbations allowing to reduce the nonlinear part of the response as much as possible. We present a mathematical framework for quantifying the quality of synthetic forcings, in the context of linear response theory, and discuss various possible choices for them. Our findings are illustrated with numerical results in low-dimensional systems.
\end{abstract}




\section{Introduction}
Statistical physics provides a formulation to study the macroscopic properties of interacting particle systems based on the behavior of their microscopic constituents. Its numerical realization, known as  ``molecular dynamics'', has played an essential role in science for the past 70 years; see \cite{ciccotti_md} for a historical perspective, and \cite{frenkel_smit, tuckerman, leimkuhler_matthews_book, allen_tildesley} for reference textbooks. One of the major aims of molecular dynamics is to compute macroscopic quantities or thermodynamic properties, typically given by averages of basic dynamical variables, which allows to obtain quantitative information on a system. Molecular dynamics can be thought of as a \emph{numerical microscope}, as it bridges the gap between theoretical and experimental work, playing an important part across (essentially) all of science.

One particular interest in molecular dynamics is the computation of transport coefficients, such as the mobility, thermal conductivity and shear viscosity. At the macroscopic level, transport coefficients relate an external forcing acting on the system (\emph{i.e.} a perturbation on the equilibrium dynamics) to an average response expressed through some steady-state flux (of energy, momentum, charged particles etc). From a mathematical perspective, this can be done by introducing a reference dynamics which models systems at equilibrium, and perturbing it by some external forcing of magnitude $\eta\in\R$ mimicking the driving exerted on the system to create some flux. This corresponds to the so-called nonequilibrium molecular dynamics (NEMD) \cite{ciccotti_nemd}.

To realize this in molecular dynamics simulations, we consider general perturbed dynamics of the form
\begin{equation}
    dX_t^\eta = \left(b(X_t^\eta) + \eta F(X_t^\eta)\right) dt + \sigma(X_t^\eta) \, dW_t,
\end{equation}
where $F$ is the external forcing and $\eta\in\R$ its magnitude, and $W_t$ a standard Brownian motion. In general, it is observed that the response $\E_\eta(R)$ of the system, given by the steady-state average of a physical observable $R$ of interest which has average 0 for unperturbed dynamics ($\eta=0$), is proportional to the magnitude of the forcing for small values of the forcing. This corresponds to the linear response regime; see Figure \ref{fig:synthetic_mock}. When this is the case, the linear response $\rho_1$ is computed as
\begin{equation}
    \rho_1 = \lim_{\eta \to 0} \frac{\E_\eta(R)}{\eta}.
    \label{linRes_intro}
\end{equation}
By definition, transport coefficients are the proportionality constants $\rho_1$ relating the response $\E_\eta(R)$ to the forcing for $|\eta|$ small. One typical example is the case where the force $F$ is constant and the configuration space is a torus. Such a forcing is considered to compute the mobility of particles in the system, which we discuss in more detail in Section \ref{subsubsec:mobility}. Another example is the case where $F$ has components in one direction only, whose magnitude depends on the position in another direction. Such forcings can be used to estimate the shear viscosity through the so-called sinusoidal transverse force method~\cite{gosling1975}.

In practice, steady-state averages involved in the linear response are computed as ergodic averages over a very long trajectory of the system, obtained as a realization of the stochastic differential equation (SDE) \eqref{noneq_SDE}. Although there are several such ways of computing these steady-state averages, it is typically done in one of two ways: either based on (i) equilibrium techniques based on Green--Kubo formulae, which are integrated autocorrelation functions \cite{green,kubo}; or (ii) nonequilibrium steady-state methods where the limit \eqref{linRes_intro} is numerically estimated; see \cite{ciccotti_nemd}. Both numerical methods have advantages and drawbacks, and are constantly undergoing algorithmic advances; see \cite{gk1,plechac_nemd}, and \cite{mcqmc_stoltz} for a comparison and discussion between the methods. In this work, however, we focus exclusively on the nonequilibrium approach.

In the NEMD approach, a standard estimator to compute the linear response~\eqref{linRes_intro} is
\begin{equation}
    \widehat{\Phi}_{\eta, t}=\frac{1}{\eta t} \int_0^t R(X_s^{\eta}) \, ds,
    \label{estimator_intro}
\end{equation}
where the fixed value of $\eta\ne 0$ should be small enough in absolute value in order for the nonlinear part of the response to be negligible. Although there are various sources of error associated with the estimator \eqref{estimator_intro}, as made precise in Section \ref{subsec:analysis}, the variance associated with \eqref{estimator_intro} is the main issue. In particular, in the context of computing transport coefficients, the large signal-to-noise ratio leads to a very large statistical error, as we are averaging very small linear responses \cite{control_variates}. The estimator~\eqref{estimator_intro}, for instance, has asymptotic variance of order $\bigO(t^{-1}\eta^{-2})$, much larger than the usual asymptotic variance of order $\bigO(t^{-1})$ associated with its equilibrium averages counterpart, as the linear regime is only valid for $|\eta|\ll 1$. This motivates the interest of variance reduction techniques, which are used to decrease the statistical error in the estimated averages. Examples of standard variance reduction techniques for Monte Carlo simulations, for which a review is provided in \cite{caflisch_1998}, are antithetic variables, stratification, control variate methods and importance sampling.

One of the main challenges with computing nonequilibrium steady-state averages, however, is that traditional equilibrium variance reduction techniques cannot be used the same way (see~\cite[Section 5.4]{pde_mono} for a detailed discussion on obstructions to this end). Although many practitioners of molecular dynamics realize that the computation of transport coefficients is a difficult numerical issue, there were only a handful of attempts to develop dedicated variance reduction techniques. Many practitioners still use direct, brute force numerical methods based on a time integration of the dynamics.

In order to reduce the statistical error of order $1/(\eta^2 t)$, one idea is to extend the regime of linear response, which consequently allows to use larger values of $\eta$. One approach for doing so is to use synthetic forcings. Such forcings, as presented in this work, were introduced by Evans and Morriss~\cite{evans} (``synthetic fields'', in their terminology) with the purpose to produce a mechanical analog of a thermal transport process. This notion was originally used by Gillan and Dixon \cite{gillan_dixon}, then abstracted and extended by Evans and Morriss in \cite{evans}. The name \emph{synthetic} is used to denote that the external fields under consideration do not exist in nature.

The key idea behind synthetic forcings is that there are infinitely many forcings which lead to the same transport coefficient. This flexibility should be used to develop better numerical methods for the computation of transport coefficients. In theory, one can add an extra forcing to the physical perturbation of the system, as long as this extra forcing preserves the invariant measure of the reference system, which in turn preserves the linear response. We call the resulting perturbation a synthetic forcing, as it has no physical representation; it is a mathematical device used to simplify the problem at hand. 

The aim is to find synthetic perturbations allowing to reduce the nonlinear part of the response as much as possible (see Figure \ref{fig:synthetic_mock}) in order to consider larger values of~$|\eta|$. In this article, we present a mathematical framework for choosing and quantifying the quality of synthetic forcings, in the context of linear response theory, and discuss various possible choices for them. We illustrate the analysis with numerical results in low dimensional systems. 

\begin{figure}
	\centering
	\includegraphics[width=0.75\textwidth]{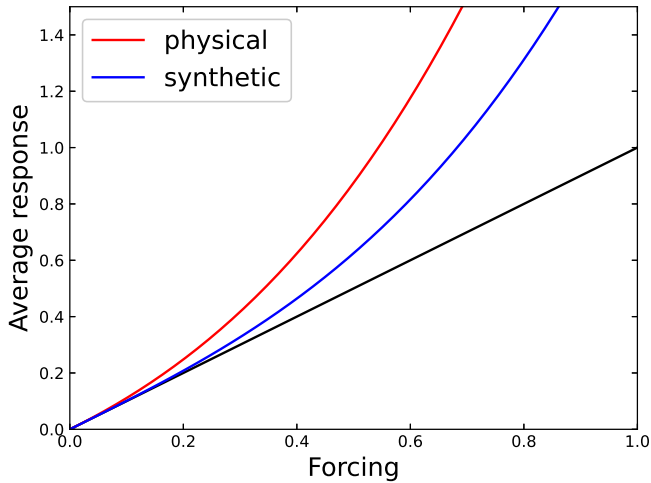}
	\caption{Illustration of the effect of synthetic forcings. Ideally, the nonlinear part of the average response $\E_\eta(R)$ is smaller, so that $\E_\eta(R) \approx \rho_1\eta$ for larger values of $|\eta|$ than with physical forcings.}
	\label{fig:synthetic_mock}
\end{figure}

This work is organized as follows. We present in Section \ref{sec:review} a review of linear response theory and associated computational techniques, and state technical results which make precise some error estimates, in particular bounds on the asymptotic variance of time averages such as \eqref{estimator_intro}. We then introduce the notion of synthetic forcings in Section \ref{sec:extending}, where we also give examples and discuss them in more detail. We next demonstrate the possibly dramatic benefits in terms of statistical error with some numerical results in low dimensions in Section \ref{sec:numerical}, namely one and two-dimensional overdamped Langevin dynamics, and one-dimensional Langevin dynamics. We finally discuss in Section \ref{sec:perspectives} the extensions and perspectives of this approach.

\section{Linear response and associated computational techniques}
\label{sec:review}
In this section, we review linear response theory and the computational techniques allowing to compute transport coefficients. The framework we consider is that of stochastic dynamics, ergodic for the Boltzmann--Gibbs measure, which are perturbed by external forcings.

We start in Section~\ref{subsec:ref_dyn} by setting up the framework for general time-homogeneous SDEs, and the specific dynamics we consider, namely (non)equilibrium overdamped Langevin and Langevin dynamics. We then review in Section~\ref{subsec:response_review} linear response theory, the definition of transport coefficients and discuss the associated standard numerical techniques to estimate these coefficients. We finally present the numerical analysis of nonequilibrium molecular dynamics in Section \ref{subsec:analysis}.

\subsection{Reference dynamics and perturbation}
\label{subsec:ref_dyn}
We start by describing in Section~\ref{subsubsec:gen_set} the general setting for linear response theory for a general stochastic differential equation (SDE). We introduce in particular some assumptions on the dynamics which will be used in the analysis throughout this paper. We next describe the dynamics and their nonequilibrium perturbations for overdamped Langevin and Langevin dynamics in Sections \ref{subsubsec:ovd_review} and \ref{subsubsec:lang_review}, respectively.

\subsubsection{General setting}
\label{subsubsec:gen_set}
\paragraph{Reference dynamics} Consider a general time-homogeneous SDE defined on the state-space $\mathcal{X}$, where $\mathcal{X}$ is typically the full space $\R^d$ or a bounded domain with periodic boundary conditions $\T^d$ (with $\T = \R/\Z$ the one-dimensional torus):
\begin{equation}
    dX_t = b(X_t) \, dt + \sigma(X_t) \, dW_t,
    \label{general_SDE}
\end{equation}
for a given initial condition $x_0 \in \mathcal{X}$, a standard Brownian motion $W_t\in \R^m$, and where~$b \colon \mathcal{X} \to \R^d$ and $\sigma \colon \mathcal{X} \to \R^{d\times m}$ are assumed to be such that there exists a unique solution to \eqref{general_SDE}. A simple setting is to assume that $b$ and $\sigma$ are $C^\infty$, hence locally Lipschitz, which is used to prove existence and uniqueness of the solution~\cite{luc,mao,khasminskii}. The SDE \eqref{general_SDE} is associated with the following infinitesimal generator
\begin{equation}
    \Lr = b^T\nabla + \frac{1}{2}\sigma\sigma^{T} \colon \nabla^2,
    \label{gen_generator}
\end{equation}
where $:$ denotes the Frobenius inner product, and $\nabla^2$ is the Hessian operator. More explicitly, for some given $C^\infty$ test function $\varphi \colon \R^d \to \R$, the operator $\Lr$ acts as
\begin{equation}
    \Lr\varphi = \sum_{i=1}^d b_i\partial_{x_i}\varphi + \frac{1}{2}\sum_{i=1}^d\sum_{j=1}^d\sum_{k=1}^m \sigma_{i,k}\sigma_{j,k}\partial^2_{x_i,x_j}\varphi .
\end{equation}
We assume that the dynamics \eqref{general_SDE} has a unique invariant probability measure with a density with respect to the Lebesgue measure denoted by $\psi_0(x)$. This density satisfies the stationary Fokker--Planck equation
\begin{equation}
    \Lr^\dagger \psi_0 = 0,
    \label{general_FP}
\end{equation}
where $\Lr^\dagger$ denotes the $L^2$-adjoint of the operator $\Lr$, acting as
\begin{equation}
    \Lr^\dagger\varphi = -\div(b\varphi) + \frac{1}{2}\nabla^2 : \left(\sigma\sigma^T\varphi\right).
\end{equation}
\paragraph{Perturbation of the reference dynamics} We next consider a perturbation of the reference dynamics \eqref{general_SDE}, obtained by adding to the drift field $b$ some smooth nongradient forcing $F \colon \mathcal{X} \to \R^d$ of magnitude $\eta \in \R$:
\begin{equation}
    dX_t^\eta = \left(b(X_t^\eta) + \eta F(X_t^\eta)\right) dt + \sigma(X_t^\eta) \, dW_t.
    \label{noneq_SDE}
\end{equation}
The generator of \eqref{noneq_SDE} is denoted by $\L_\eta = \L_0 + \eta\Lp$, where $\Lp$ is the generator associated with the physical perturbation: 
\begin{equation}
    \Lp = F^T\nabla.   
\end{equation}
We assume that $b$, $\sigma$ and $F$ are such that the following assumption holds.
\begin{assumption}[{\bf Uniqueness of the invariant measure}]
\label{assumption1}
 The dynamics \eqref{noneq_SDE} admits a unique invariant probability measure for any $\eta \in \mathbb{R}$, with a smooth density~$\psi_\eta$ with respect to the Lebesgue measure. Moreover, trajectorial ergodicity holds: for any observable $R \in$ $L^{1}\left(\psi_\eta\right)$, and any initial condition $X_0^\eta$,
\begin{equation}
    \mathbb{E}_\eta(R) := \int_{\mathcal{M}} R(x) \psi_\eta(x) \, dx = \lim _{T \rightarrow \infty} \frac{1}{T} \int_0^T R(X_t^\eta) \, dt \quad \text { a.s. }
\end{equation}
\end{assumption}
Note that $\psi_\eta$ solves the Fokker--Planck equation
\begin{equation}
	\left(\Lr + \eta\Lp\right)^\dagger\psi_\eta = 0.
\end{equation}
Sufficient conditions for Assumption \ref{assumption1} to hold are discussed after Assumption \ref{assumption2} below. Although~$\psi_\eta$ can be shown to uniquely exist, its analytical expression is generally not known. Note that in the case~$\eta = 0$, the dynamics \eqref{noneq_SDE} reduces to the reference dynamics \eqref{general_SDE}. 

 We need to make precise some estimates on the evolution semigroup associated with \eqref{noneq_SDE} for the analysis presented in Section \ref{subsec:analysis}. We consider to this end weighted spaces of bounded functions. To define them, we first introduce a family of Lyapunov functions denoted by $(\mathcal{K}_n)_{n\in\mathbb{N}}$, with $\mathcal{K}_n \colon \mathcal{X} \to [1,+\infty)$ such that
\begin{equation}
    \forall n \geq 0, \qquad \mathcal{K}_n \leq \mathcal{K}_{n+1}.
\end{equation}
The associated weighted $B^\infty$ spaces are
\begin{equation}
    B_n^\infty = \left\{\varphi \text{ measurable} \, \middle| \, \|\varphi\|_{B_n^\infty} := \sup _{x \in \mathcal{X}}\left| \frac{\varphi(x)}{\mathcal{K}_n(x)} \right| < +\infty\right\}.
\end{equation}
In particular, for any $\varphi \in B^\infty_m$, it holds that $|\varphi| \leq \|\varphi\|_{B^\infty_m} \mathcal{K}_m$. We can then introduce the space $\S$, which consists of all functions $\varphi \in C^\infty(\mathcal{X})$ which grow at most like $\mathcal{K}_n$ for some $n$, and whose derivates also grow at most like $\mathcal{K}_m$, possibly with $m > n$ and $m$ depending on the order of derivation. Denoting by $\partial^k = \partial_{x_1}^{k_1}\dotsc\partial_{x_d}^{k_d}$ for~$k = (k_1,\dotsc,k_d) \in \N^d$,
\begin{equation}
    \S = \left\{\varphi \in C^\infty\left(\mathcal{X}\right) \, \middle| \, \forall k \in \mathbb{N}^d, \; \exists n \in \mathbb{N}, \; \partial^k \varphi \in B_n^\infty\right\}.
\end{equation}
When $\mathcal{X}$ is bounded, it is possible to choose $\mathcal{K}_n = 1$ for all $n\geq 0$, in which case~$\S = C^\infty(\mathcal{X})$. For unbounded spaces, a typical choice is $\mathcal{K}_n = 1+|x|^n$, in which case the functions $\varphi \in \S$ and their derivatives grow at most polynomially. We also consider the subspace $\S_0=\Pi_0 \S$ of functions with average 0 with respect to $\psi_0$, where $\Pi_\eta$ for~$\eta \in \R$ denotes the projection operator
\begin{equation}
    \Pi_\eta \varphi = \varphi - \int_\mathcal{X} \varphi \, \psi_\eta.
\end{equation}
We denote by $\|\cdot\|_{\mathcal{B}(E)}$ the operator norm on the Banach space $\mathcal{B}(E)$ of bounded linear operators on a Banach space $E$, defined by
\begin{equation}
    \|\mathcal{A}\|_{\mathcal{B}(E)} = \sup_{g\in E\setminus\{0\}} \frac{\|\mathcal{A}g\|_E}{\|g\|_E}.
\end{equation}
 Before we state the next assumption, we let $\mathcal{A}^*$ denote the adjoint of $\mathcal{A}$ on the space~$L^2(\psi_0)$. More explicitly, for any test functions $\varphi, \phi \in C^\infty$ with compact support,
\begin{equation}
    \int_\mathcal{X} (\mathcal{A}\varphi)\phi \, \psi_0 = \int_\mathcal{X} \varphi(\mathcal{A}^*\phi) \, \psi_0.
    \label{L2star_adj}
\end{equation}
The action of these operators can be found via integration by parts. In particular, for $\psi_0(x)$ a probability density proportional to $\e^{-\beta U(x)}$, an integration by parts shows that $\partial_{x_i}^* = -\partial_{x_i} + \beta\partial_{x_i} U$. In addition, note that~$\nabla_x^*\nabla_x = \sum_{i=1}^d \partial_{x_i}^* \partial_{x_i}$.

Note that the Fokker--Planck equation \eqref{general_FP}, written in terms of the $L^2$-adjoint, can also be written with the $L^2(\psi_0)$-adjoint as $\Lr^*\ind = 0$. We make the following assumption on the Lyapunov functions.
\begin{assumption}[{\bf Lyapunov estimates}]
\label{assumption2}
The space $\S$ is dense in $L^2(\psi_\eta)$ for any $\eta \in \R$. For any $n \in \mathbb{N}$, the $L^{2}(\psi_\eta)$ norms of the Lyapunov functions are uniformly bounded on compact sets of $\eta$: for any $\eta_*>0$, there exists a constant $C_{n, \eta_{*}} < +\infty$ such that
\begin{equation}
    \forall|\eta| \leqslant \eta_*, \quad\left\|\mathcal{K}_n\right\|_{L^2\left(\psi_\eta\right)} \leqslant C_{n, \eta_*}
    \label{est1}.
\end{equation}
Moreover $-\mathcal{L}_{\eta}$ is invertible on $\Pi_{\eta} B_n^\infty$, and the inverse generator is bounded uniformly on compact sets of $\eta$:
\begin{equation}
    \forall|\eta| \leqslant \eta_*, \quad\left\|-\mathcal{L}_\eta^{-1}\right\|_{\mathcal{B}\left(\Pi_\eta B_n^\infty\right)} \leqslant K_{n, \eta_*}.
    \label{est3}
\end{equation}
Additionally,
\begin{equation}
    \label{est2}
    \forall |\eta| \leq \eta_*, \quad \forall n \geqslant 1, \qquad \sup _{t \in \mathbb{R}_+} \left\|\e^{t\L_\eta}\mathcal{K}_n\right\|_{B^\infty_n} \leqslant M_{n, \eta_{*}} < +\infty
\end{equation}
We finally assume that for any $n,n' \in \mathbb{N}$, there exists $m\in\mathbb{N}$ such that $\mathcal{K}_n\mathcal{K}_{n'} \in B^\infty_m$.
\end{assumption}
The estimates \eqref{est1} through \eqref{est2} are usually obtained from Lyapunov conditions of the form
\begin{equation}
    \L_\eta\mathcal{K}_n \leq -a_{n,\eta}\mathcal{K}_n + b_{n,\eta},
    \label{est_lyapunov}
\end{equation}
for some $a_{n,\eta} > 0$ and $b_{n,\eta} \in \R$. After integration against $\psi_\eta$, and making use of the invariance of~$\psi_\eta$ by the dynamics,
\begin{equation}
    0 = \int_\mathcal{X}\L_\eta \mathcal{K}_n \, \psi_\eta \leq -a_{n,\eta}\int_\mathcal{X}\mathcal{K}_n \, \psi_\eta + b_{n,\eta},
\end{equation}
so that
\begin{equation}
    1 \leq \int_\mathcal{X}\mathcal{K}_n \, \psi_\eta \leq \frac{b_{n,\eta}}{a_{n,\eta}}.
\end{equation}
The condition \eqref{est1} then follows from the last statement of Assumption \ref{assumption2} since, for any~$n\geq 1$, there exists $C_n>0$ and $m\geq1$ such that $1\leq \mathcal{K}_n^2 \leq C_n\mathcal{K}_m$. Condition~\eqref{est_lyapunov} also implies the existence of an invariant probability measure for any~$\eta\in\R$ when a minorization condition holds \cite{minorization}. The minorization condition typically follows from a controllability argument and (hypo)ellipticity conditions, see for instance~\cite{mattingly,luc}. As for Assumption \ref{assumption1}, trajectorial ergodicity holds when the generator~$\L_\eta$ is either elliptic or hypoelliptic, and there exists an invariant probability measure with positive density with respect to the Lebesgue measure \cite{kliemann}.

 We need an assumption in our analysis on the generator $\Lr$ of the reference dynamics, which will be useful when stating some technical results later on.
\begin{assumption}[{\bf Stability of smooth functions by inverse operators}]
\label{assumption3}
The space $\S$ is stable by the generator $\L_0$, and $\mathcal{L}_{0}$ and $\mathcal{L}_{0}^{*}$ are invertible on $\S_{0}$. This means that, for any $\varphi \in \S_0$, there exists a unique solution $\Psi \in \S_0$ to the Poisson equation~$-\mathcal{L}_0\Psi = \varphi$.
\end{assumption}

The generator of the perturbation should also satisfy the next assumption.
\begin{assumption}[{\bf Stability of smooth functions by the perturbation operator}]
\label{assumption4}
The generator $\Lp$ of the perturbation is such that $\S$ is stable by $\Lp$, and $\Lp^*\S \subset~\S_0$.
\end{assumption}
Assumption \ref{assumption4} is easily seen to be satisfied for perturbations $\Lp$ of the form $F^T\nabla$ when $F$ has components in $\S$, which is assumed here and for the remainder of this work. A simple computation based on integrations by parts shows that $\Lp^*\S \subset \S_0$ when $\Lp\ind = 0$. Indeed, for $\varphi \in \S$, using the definition of the $L^2(\psi_0)$-adjoint,
\begin{equation}
	\int_\mathcal{X} \Lp^*\varphi \, \psi_0 = \int_\mathcal{X} \varphi \left(\Lp \ind\right) \, \psi_0 = 0.
	\label{ass4_int}
\end{equation}
One function of particular interest in the range of $\Lp^*$ is the conjugate response function
\begin{equation}
	S = \Lp^*\ind,
	\label{conjugate_response}  
\end{equation}
which we introduce here as it will be useful later. Note that the expression for $S$ comes from the generator $\Lp$ of the perturbation, and not the observable $R$. Let us emphasize that beyond the case of physical forcings $F^T\nabla$, Assumption \ref{assumption4} will be needed for more general operators in Section \ref{sec:extending}.

 Assumptions \ref{assumption1} through \ref{assumption4} are natural for overdamped Langevin and Langevin dynamics, which we present in the next two sections.

\subsubsection{Overdamped Langevin dynamics}
\label{subsubsec:ovd_review}
One typical dynamics used in molecular dynamics is overdamped Langevin dynamics, which evolves only the positions~$q$ of the system.

\paragraph{Reference dynamics} Mathematically, the overdamped Langevin dynamics corresponds to the following SDE with nondegenerate noise:
\begin{equation}
    dq_t = -\nabla V(q_t) \, dt + \sqrt{\frac{2}{\beta}} \, dW_t,
    \label{eq_ovd}
\end{equation}
where $\beta >0$ is proportional to the inverse temperature, and $V$ is a smooth potential. The generator associated with \eqref{eq_ovd} is given by
\begin{equation}
    \Lr = -\nabla V^T \nabla + \frac{1}{\beta}\Delta,
    \label{ovd_Lr}
\end{equation}
with $L^2(\mathcal{X})$-adjoint
\begin{align}
    \Lr^\dagger &= \div(\nabla V\cdot) + \frac{1}{\beta}\Delta.
\end{align}
The dynamics \eqref{eq_ovd} admits the Gibbs measure with density 
\begin{equation}
    \psi_0(q) = \frac{1}{Z}\e^{-\beta V(q)}, \qquad Z = \int_\mathcal{X}\e^{-\beta V(q)} dq < +\infty
\end{equation}
as its unique invariant probability measure. It can indeed be checked that $\Lr^\dagger \psi_0 = 0$, so that $\psi_0$ is a stationary solution to the Fokker--Planck equation. Here and in the remainder of this work, we assume that $\e^{-\beta V} \in L^1(\mathcal{X})$. Note that the generator~\eqref{ovd_Lr} is self-adjoint on the Hilbert space $L^2(\psi_0)$, as $\Lr = -\beta^{-1}\nabla^*\nabla$.

\paragraph{Nonequilibrium perturbation} The perturbed overdamped Langevin dynamics corresponding to \eqref{noneq_SDE} reads
\begin{equation}
    dq_t^\eta = \left(-\nabla V(q_t^\eta) + \eta F(q_t^\eta)\right) dt + \sqrt{\frac{2}{\beta}} \, dW_t.
    \label{neld_ovd}
\end{equation}
It has generator $\L_\eta = \Lr + \eta\Lp$, where $\Lp = F^T\nabla$. Recall that $F$ is assumed to be nongradient, so that in general there is no explicit expression for the invariant probability measure of \eqref{neld_ovd} when such a measure exists.

\begin{remark}
One of the only cases where a closed form for $\psi_\eta$ is known is the one-dimensional case for overdamped Langevin dynamics, where the associated Fokker--Planck equation is directly solvable, as discussed for instance in \cite{fathi2018}, which generalizes the computations of \cite[Section 2.5]{Reimann_2002}.
\end{remark}

\paragraph{Conditions for assumptions to be satisfied}
We now discuss the conditions under which Assumptions \ref{assumption1} through \ref{assumption4} are satisfied for the perturbed overdamped Langevin dynamics \eqref{neld_ovd}, which correspond to standard results in the literature.

\begin{itemize}[itemsep=0.75em, topsep=0.75em]
	\item If the space $\mathcal{X}$ is compact, then it is trivial to satisfy the Lyapunov condition~\eqref{est_lyapunov} by choosing $\mathcal{K}_n = \ind$. In the case of unbounded spaces, a typical choice is~$\mathcal{K}_n = 1 + |q|^n$, with the condition that there exist $A>0$ and $B\in\R$ such that
	\begin{equation}
  		q^T\nabla V(q) \geq A|q|^2 - B.
  		\label{lyap_ovd}
	\end{equation}
	Condition \eqref{lyap_ovd} is satisfied for potentials $V(q)$ that behave as $|q|^k$ for~$k\geq 2$ at infinity. This ensures that the dynamics returns to some compact region around the origin.
	 
	\item Assumption \ref{assumption1} can be shown to hold using Lyapunov techniques and a minorization condition \cite{mattingly,luc, kopec_ovd}. It can thus be shown that the dynamics~\eqref{neld_ovd} has a unique invariant probability measure with positive density with respect to the Lebesgue measure, and therefore that trajectorial ergodicity holds.
	
	\item Assumptions \ref{assumption2} and \ref{assumption3} hold under some conditions on the potential $V$, which include \eqref{lyap_ovd}; see \cite{kopec_ovd} for a precise discussion.
	 
	\item Lastly, Assumption \ref{assumption4} is trivially satisfied, as $\Lp = F^T\nabla_q$ leaves $\S$ stable when $F \in \S$, and $\Lp^*\S \subset \S_0$ in view of \eqref{ass4_int}.
\end{itemize}

\subsubsection{Langevin dynamics}
\label{subsubsec:lang_review}
Another dynamics of interest used in molecular dynamics is Langevin dynamics, which can be seen as a Hamiltonian dynamics perturbed by an Ornstein--Uhlenbeck process on the momenta. Mathematically, it corresponds to an SDE with degenerate noise, as the noise acts on the momenta only (\emph{i.e.} the diffusion matrix does not have full rank).

\paragraph{Reference dynamics} At equilibrium, Langevin dynamics evolves positions $q$ and momenta $p$ according to the SDE
\begin{align}
\begin{split}
\begin{cases}
    dq_t = M^{-1} p_t \, dt, \\
    dp_t = -\nabla V(q_t) \, dt - \gamma M^{-1} p_t \, dt + \sqrt{\dfrac{2\gamma}{\beta}} \, dW_t,
    \label{lang_ref}
\end{cases}
\end{split}
\end{align}
where $\gamma>0$ is the friction coefficient and $M \in \R^{d\times d}$ is a positive definite matrix, called the mass matrix. The state-space $\mathcal{X}$ is either $\T^d \times \R^d$ or the full space $\R^d \times \R^d$.

The infinitesimal generator associated with \eqref{lang_ref} is the following degenerate elliptic operator
\begin{equation}
    \Lr = \L_\text{ham} + \gamma \L_\text{FD},
\end{equation}
where $\L_\text{ham}$ is the generator of the Hamiltonian part of the dynamics, and $\L_\text{FD}$ is the generator of the fluctuation-dissipation part, \emph{i.e.} the Gaussian process on the momenta, also known as the Ornstein--Uhlenbeck process:
\begin{equation}
	\L_\text{ham} = p^T M^{-1} \nabla_q - \nabla V^T\nabla_p, \qquad \L_\text{FD} = -p^T M^{-1} \nabla_p + \beta^{-1} \Delta_p.
	\label{lang_gen_split}
\end{equation}
The $L^2(\mathcal{X})$-adjoint of $\Lr$ acts as
\begin{align}
    \Lr^\dagger\psi = -\L_\mathrm{ham}\psi + \gamma\div_p\left(M^{-1} p \psi + \beta^{-1} \nabla_p\psi\right).
    \label{lang_FP_L2}
\end{align}
A simple computation shows that the dynamics \eqref{lang_ref} admits the Boltzmann--Gibbs distribution as an invariant probability measure, which is in fact unique. The density of this measure satisfies the stationary Fokker--Planck equation $\Lr^\dagger\psi_0 = 0$, where
\begin{equation}
    \psi_0(q) = \frac{1}{Z}\e^{-\beta H(q,p)}, \qquad Z = \int_\mathcal{X}\e^{-\beta H(q,p)} dq \, dp.
    \label{gibbs_lang}
\end{equation}
Overdamped Langevin dynamics can be obtained from Langevin dynamics in two limiting cases: in the high friction limit $\gamma \to +\infty$, upon rescaling time as $\gamma t$; or in the small mass limit $m\to 0$, as discussed in \cite{nelson} and \cite[Section 2.2.4]{freeEnergy}.

\paragraph{Nonequilibrium perturbation} We also consider the case where the dynamics \eqref{lang_ref} is perturbed by a nongradient force $F \colon \mathcal{X} \to \R^d$ of magnitude $\eta\in \R$. The resulting nonequilibrium Langevin dynamics reads
\begin{align}
\begin{split}
\begin{cases}
    dq_t^\eta = M^{-1} p_t^\eta \, dt, \\
    dp_t^\eta = \left(-\nabla V(q_t^\eta) + \eta F(q_t^\eta)\right) dt - \gamma M^{-1} p_t^\eta \, dt + \sqrt{\dfrac{2\gamma}{\beta}} \, dW_t.
    \label{lang_noneq}
\end{cases}
\end{split}
\end{align}
It has generator $\L_\eta = \L_0 + \eta\Lp$ with $\Lp = F^T\nabla_p$. 

\paragraph{Conditions for assumptions to be satisfied}
We now discuss the conditions under which Assumptions \ref{assumption1} through \ref{assumption4} are satisfied for the perturbed Langevin dynamics~\eqref{lang_noneq}, which correspond here as well to standard results in the literature.

\begin{itemize}[itemsep=0.75em, topsep=0.75em]
	\item For Langevin dynamics, the momentum space is always unbounded, so the choice for the Lyapunov function to satisfy the condition \eqref{est_lyapunov} depends on the position space. For compact position spaces, it suffices to choose~$\mathcal{K}_n(p) = 1 + |p|^n$. For unbounded position spaces, there are various possible choices (see \cite{mattingly,wu2002,talay}). One possibility is to consider 
	\begin{equation}
  		\mathcal{K}_n(q, p) = \left(1 + H(q, p) - V_- + \frac{\gamma}{2} p^T M^{-1}q + \frac{\gamma^2}{4}q^TM^{-1}q\right)^n,
	\end{equation}
	under the following conditions: the potential energy function $V$ is bounded from below by~$V_->-\infty$, and there exist $A,B>0$ and $C\in\R$ such that
	\begin{equation}
		q^TM^{-1}\nabla V(q) \geq AV(q) + Bq^TM^{-1}q + C.
		\label{lyap_lang}
	\end{equation}
	\item The existence of an invariant probability measure in Assumption \ref{assumption1} can be shown using Lyapunov techniques when a minorization condition holds \cite{mattingly,luc}. It can be shown that the dynamics \eqref{lang_noneq} has a unique invariant measure, as $\psi_0$ has positive density with respect to the Lebesgue measure. Thus, since the generator $\L_\eta$ is hypoelliptic \cite{pde_mono}, it follows that trajectorial ergodicity holds \cite{kliemann}. 
	\item Assumption \ref{assumption2} holds as discussed in \cite{control_variates}.
	\item Assumption \ref{assumption3} holds under some conditions on the potential $V$, which include~\eqref{lyap_lang}; see \cite{kopec_lang} for a precise discussion.
	\item Lastly, Assumption \ref{assumption4} is trivially satisfied, as $\Lp = F^T\nabla_p$ is stable by $\S$, and $\Lp\ind = 0$; recall the discussion around \eqref{ass4_int}.
\end{itemize}

\subsection{Transport coefficients and their numerical approximations}
\label{subsec:response_review}
We discuss in this section how to compute transport coefficients such as the mobility, thermal conductivity and shear viscosity, in the context of linear response theory. We outline in Section \ref{subsubsec:linRes_theory} the theoretical framework of linear response theory, discuss regularity and well-posedness conditions, and state some technical results. We then discuss standard numerical techniques for computing transport coefficients in Section \ref{subsubsec:num_techniques}. We finally discuss the paradigmatic example of mobility in Section \ref{subsubsec:mobility}, which is the running example we use to theoretically illustrate our framework.

\subsubsection{Linear response theory}
\label{subsubsec:linRes_theory}
The linear response $\rho_1$ of a given observable~$R$ is defined as the proportionality constant between the average response $\E_\eta(R)$ and the magnitude of the perturbation $\eta$ in the limit $\eta\to 0$, provided this limit makes sense (see Lemma \ref{lemma1} below for a functional framework ensuring that this is the case):
\begin{equation}
    \rho_1 = \lim _{\eta \to 0} \frac{1}\eta\left(\int_\mathcal{X} R \, \psi_\eta - \int_\mathcal{X} R  \, \psi_0\right) = \lim _{\eta \to 0} \frac{\E_\eta(R) - \E_0(R)}{\eta}.
        \label{linRes}
\end{equation}
We typically consider observables for which $\E_0(R) = 0$. Linear response characterizes, in powers of $\eta$, the modification of $\psi_\eta$ with respect to the canonical measure. It is expected that $\psi_\eta$ is a modification of order $\eta$ of $\psi_0$, with dominant order $\eta$ for $\eta$ small. Rigorously proving this statement requires some regularity results. Let us first motivate the form of $\psi_\eta$ for $\eta$ small. We rewrite this function as $\psi_\eta = f_\eta \psi_0$, with $f_\eta$ a perturbation of the constant function $\ind$:
\begin{equation}
    f_\eta = \ind + \eta\f_1 + \eta^2 \f_2 + \cdots.
    \label{feta}
\end{equation}
In this context, the Fokker--Planck equation $(\L_0 + \eta\Lp)^\dagger \psi_\eta = 0$ can be reformulated, by the definition of the $L^2(\psi_0)$ adjoint, as
\begin{equation}
    \left(\L_0 + \eta\Lp\right)^*f_\eta = 0.
    \label{FPstar}
\end{equation}
By identifying terms with the same powers of $\eta$ in \eqref{FPstar}, we obtain, for terms of order~$\eta$, the following Poisson equation:
\begin{equation}
    \L_0^*\f_1 = -\Lp^*\ind.
    \label{poissonf1}
\end{equation}
The solution $\f_1$ to \eqref{poissonf1} is well-defined in view of Assumptions \ref{assumption3} and \ref{assumption4}. Using the conjugate response function $S$ defined in \eqref{conjugate_response}, we can then write 
\begin{equation}
	\f_1 = -(\Lr^{-1})^*S.
	\label{f1}
\end{equation}
Similarly, higher order terms are obtained by identifying terms of order $\eta^k$ in \eqref{FPstar}, leading to the recursive definition
\begin{equation}
    \forall k \geq 1, \qquad \f_{k+1} = (-\L_0^{-1})^* \Lp^*\f_k.
    \label{eq6}
\end{equation}
It can be shown inductively that the expression \eqref{eq6} is also well-defined, as $\Lp^*$ and $(\Lr^{-1})^*$ stabilize~$\S$. 

In particular, the expansion \eqref{feta} allows us to write the linear response \eqref{linRes} in terms of the first-order perturbation of the invariant measure of the reference dynamics when~$\E_0(R) = 0$:
\begin{equation}
 	\rho_1 = \lim_{\eta\to0} \frac{\E_\eta(R)}{\eta} = \int_\mathcal{X} R \f_1 \, \psi_0.
 	\label{linResf1}
\end{equation}
In order to rigorously prove that $\psi_\eta$ is of the form $f_\eta\psi_0$ with $f_\eta$ as above, we consider Lemma \ref{lemma1}, which generalizes \cite[Remark 5.5]{pde_mono} to expansions of arbitrary order $k$.
\begin{lemma}
\label{lemma1}
Suppose that Assumptions \ref{assumption1} through \ref{assumption4} hold true. Fix $\eta_{*}>0$ and~$\varphi \in \S$. For any $k\geq 2$, there exists $M \in \R_+$ (which depends on $\eta_{*}$, $k$ and $\varphi$) such that
\begin{equation}
    \int_{\mathcal{X}} \varphi \, \psi_\eta = \int_\mathcal{X} \varphi\left(1+\eta \f_1 + \cdots + \eta^{k-1}\f_{k-1}\right) \psi_0 + \eta^k \mathscr{R}_{\eta, \varphi,k},
    \label{lemma1_statement}
\end{equation}
with $\left|\mathscr{R}_{\eta, \varphi,k}\right| \leqslant M$ for all $\eta \in\left[-\eta_{*}, \eta_{*}\right]$.
\end{lemma}

\begin{remark}
	This expansion can be given a meaning as a converging infinite expansion at the level of operators when $\Lp$ is $\Lr$-bounded, and $\Lr$ and its inverse are restricted to $\Pi_0L^2(\psi_0)$ (see \cite[Theorem 5.2]{pde_mono}). However, the perturbation we consider in Section \ref{sec:extending} will be quite general and it cannot be assumed that the perturbation operator is $\Lr$-bounded.
\end{remark}

\begin{proof}
It is sufficient to prove the result for $\varphi\in\S_0$, as we can replace $\varphi$ by $\varphi+C$ for some constant $C$, since $\f_k$ has average 0 with respect to $\psi_0$ for any $k$. Using the definition of $\f_k$, a straightforward computation gives, for $\phi \in \S$,
\begin{equation}
    \int_\mathcal{X} (\mathcal{L}_\eta\phi)\left(1 + \eta \mathfrak{f}_1 + \cdots + \eta^{k-1}\f_{k-1}\right) \psi_0 = \eta^k \int_\mathcal{X}\left(\Lp \phi\right) \mathfrak{f}_{k-1} \psi_0.
\end{equation}
Note that all functions which appear in the above integrals are in $\S$, so that their integrals with respect to $\psi_0$ are well defined. One would like at this stage to replace~$\phi$ by~$\L_\eta^{-1}\Pi_\eta\varphi$, which would already give the result. However, this would require controlling the integrability of derivatives of~$\L_\eta^{-1}\Pi_\eta\varphi$, which does not follow from our assumptions. We therefore consider an operator $Q_{\eta,k}$, defined on~$\S_0$, which approximates $\L_\eta^{-1}$ in some sense. Replacing $\phi$ by $Q_{\eta,k}\varphi$ leads to
\begin{equation}
\begin{aligned}
    \int_\mathcal{X} \mathcal{L}_\eta Q_{\eta,k} \varphi \, \psi_\eta = 0 =& \int_\mathcal{X} \mathcal{L}_\eta Q_{\eta,k} \varphi\left(1 + \eta \mathfrak{f}_1 + \cdots + \eta^{k-1}\f_{k-1}\right) \, \psi_0 \\
   & \ \ \ - \eta^k \int_\mathcal{X}\left(\Lp Q_{\eta,k}\varphi\right) \mathfrak{f}_{k-1} \, \psi_{0}.
    \label{lemma1_mid}
\end{aligned}
\end{equation}
In order to construct $Q_{\eta,k}$, we start from the following \emph{formal} identity:
\begin{align}
	\left(\Lr + \eta\Lp\right)^{-1} &= \Lr^{-1}\left(1+\eta \Lp\Lr^{-1}\right)^{-1} \\
	&= \Lr^{-1}\left(1-\eta\Lp\Lr^{-1} + \cdots + (-\eta)^k (\Lp\Lr^{-1})^k + \cdots \right).
\end{align}
The previous expansion suggests to introduce an approximate inverse operator, obtained by truncating the formal infinite expansion at order $\bigO(\eta^k)$. It is moreover sufficient to construct a pseudo-inverse on $\S_0$, which amounts to restricting all the operators using $\Pi_0$ on the left and on the right. We therefore introduce
\begin{equation}
	Q_{\eta,k} = \Pi_0\Lr^{-1}\Pi_0 + \sum_{n=1}^{k-1} (-\eta)^n \Pi_0\Lr^{-1}\Pi_0\left(\Lp\Pi_0\Lr^{-1}\Pi_0\right)^n.
	\label{pseudo_inverse}
\end{equation}
This operator is well-defined by Assumptions \ref{assumption3} and \ref{assumption4}, as it consists of finite compositions of operators leaving $\S$ invariant, so that $Q_{\eta,k} \colon \S\to \S_0$. Note that, by construction,
\begin{equation}
	\L_\eta Q_{\eta,k} = \Pi_0 + (-1)^{k-1}\eta^k\left(\Lp\Pi_0\Lr^{-1}\Pi_0\right)^k.
\end{equation}
Equation \eqref{lemma1_mid} then becomes
\begin{equation}
    \int_\mathcal{X} (\Pi_0 \varphi) \, \psi_\eta = \int_\mathcal{X} \Pi_0 \varphi\left(1 + \eta \mathfrak{f}_1 + \cdots + \eta^{k-1}\f_{k-1}\right) \psi_0 + \eta^k \mathscr{R}_{\eta, \varphi, k},
    \label{lemma1_midResult}
\end{equation}
where the remainder is given by
\begin{equation}
\begin{aligned}
    \mathscr{R}_{\eta, \varphi,k} &= (-1)^{k}\int_\mathcal{X}\left(\Lp \Pi_0 \mathcal{L}_0^{-1} \Pi_0\right)^k \varphi \, \psi_\eta \\ 
    & +\int_\mathcal{X}\left((-1)^{k-1}\left[\left(\Lp \Pi_0 \mathcal{L}_0^{-1} \Pi_0\right)^{k} \varphi\right](1 + \cdots + \eta^{k-1}\f_{k-1})-\left(\Lp Q_{\eta,k} \varphi\right) \mathfrak{f}_{k-1}\right) \psi_0.
\end{aligned}
\label{remainder}
\end{equation}
Equation \eqref{lemma1_midResult} is the desired result \eqref{lemma1_statement}, since $\Pi_0\varphi = \varphi$ for $\varphi \in \S_0$. 

It remains at this stage to show that the remainder term $\mathscr{R}_{\eta,\varphi,k}$ is uniformly bounded. We introduce, for notational convenience, the operator $\mathcal{A} = \Lp \Pi_{0} \mathcal{L}_{0}^{-1} \Pi_{0}$. Since $\varphi\in \S$, it holds that $\mathcal{A}^k\varphi \in \S$ by Assumptions \ref{assumption3} and \ref{assumption4}. By Assumption \ref{assumption2}, there exists $m_k \in \N$ such that
\begin{equation}
\begin{aligned}
    \forall|\eta| \leq \eta_*, \qquad \left|\int_\mathcal{X}\mathcal{A}^{k-1} \varphi \, \psi_\eta\right|
    &\leq \left\|\mathcal{A}^{k-1}\varphi\right\|_{B^\infty_{m_k}}\int_\mathcal{X} \mathcal{K}_{m_k} \, \psi_\eta \\
    &\leq \left\|\mathcal{A}^{k-1}\varphi\right\|_{B^\infty_{m_k}}\|\mathcal{K}_{m_k}\|_{L^2(\psi_\eta)}.
    \label{lemma1_A_bound}
\end{aligned}
\end{equation}
The latter quantity is uniformly bounded in $\eta$ in view of \eqref{est1}, with a bound depending only on $\eta_*, \varphi$ and $k$.

 A similar uniform bound can be found for the second integral in \eqref{remainder}, since all the functions which appear in the integral belong to some space $B^\infty_{m_k}$, upon possibly increasing $m_k$. Indeed, since~$\f_i \in \S$ for any $1\leq i\leq k-1$, there exist $m_{k'}\geq 1$ such that
\begin{equation}
    |\eta\f_1 + \cdots + \eta^{k-1}\f_{k-1}| \leq \left(|\eta|\|\f_1\|_{B^\infty_{m_{k'}}} + \cdots |\eta|^{k-1}\|\f_{k-1}\|_{B^\infty_{m_{k'}}}\right)\mathcal{K}_{m_{k'}},
\end{equation}
where the prefactor of $\mathcal{K}_{m_{k'}}$ is uniformly bounded for $|\eta|\leq \eta_*$, so the second term in~$\eqref{remainder}$ is uniformly bounded in $\eta$ by some constant depending only on $\eta_*$, $\varphi$ and $k$. Lastly, since $Q_{\eta,k} \in \S$, then $\Lp Q_{\eta,k}\varphi \in \S$, so there exists $m_{k''}\geq 1$ such that
\begin{equation}
    \left|\int_\mathcal{X} (\Lp Q_{\eta,k}\varphi)\f_1 \, \psi_0\right| \leq \left\|\Lp Q_{\eta,k}\varphi\right\|_{B^\infty_{m_{k''}}}\|\f_1\|_{B^\infty_{m_{k''}}} \int_\mathcal{X} \mathcal{K}_{m_{k''}}^2 \, \psi_\eta,
    \label{lemma1_bound}
\end{equation}
where the three terms on the right-hand side of \eqref{lemma1_bound} are uniformly bounded for~$|\eta|\leq \eta_*$ by Assumptions \ref{assumption2} and~\ref{assumption3}. This allows us to obtain the desired result.
\end{proof}

\paragraph{Nonlinear response} Linear response is only valid up to certain values of $\eta$, after which the nonlinear part of the response becomes too large. To study the crossover, one needs to consider higher order terms of the response. As for the invariant measure~$\psi_\eta$ of the perturbed dynamics, the response can be expanded as a polynomial in~$\eta$ in view of \eqref{lemma1_statement}:
\begin{equation}
	\E_\eta(R) = \int_\mathcal{X} R \, \psi_\eta = \eta\rho_1 + \eta^2\rho_2 + \eta^3\rho_3 +  \cdots.
	\label{rhoeta}
\end{equation}
This allows us to define the $k$th-order response, denoted by $\rho_k$, characterized inductively for $k\geq 2$ as
\begin{equation}
    \rho_k = \lim_{\eta\to 0} \frac{|\E_\eta(R) - (\eta\rho_1 + \eta^2\rho_2 + \cdots + \eta^{k-1}\rho_{k-1})|}{\eta^k} = \int_\mathcal{X} R \f_k \, \psi_0.
    \label{rhon_response}
\end{equation}
The purpose and usefulness of writing higher order terms of the response will be made clear in Section \ref{sec:extending}, when we attempt at reducing the contributions $\rho_k$ for $k\geq 2$.

\subsubsection{Numerical techniques to estimate transport coefficients}
\label{subsubsec:num_techniques}
In this section, we discuss standard numerical techniques for computing transport coefficients. We first reformulate the linear response presented in Section \ref{subsubsec:linRes_theory} as an integrated correlation, through the celebrated Green--Kubo formula, and then outline the numerical difficulties associated with the estimation of transport coefficients, in particular the large statistical error of the associated estimators. 

\paragraph{Reformulating the linear response as an integrated correlation} A useful corollary of \eqref{linResf1} and \eqref{f1} is that we can reformulate the definition of the linear response~\eqref{linRes} as an integrated correlation function, called the Green--Kubo formula. To define it, we first consider the following operator identity
\begin{equation}
	\Lr^{-1} = -\int_0^{+\infty} \e^{t\Lr} \, dt
	\label{operator_identity}
\end{equation}
on the Hilbert space
\begin{equation}
	L_0^2(\psi_0) = \Pi_0L^2(\psi_0) = \left\{\varphi\in L^2(\psi_0) \, \middle| \, \int_\mathcal{X} \varphi \, \psi_0 = 0\right\}.
\end{equation}
The identity \eqref{operator_identity} holds for underdamped and overdamped Langevin under certain conditions discussed in \cite[Section 2]{pde_mono}. In view of this identity, as well as \eqref{linRes}, \eqref{linResf1} and \eqref{f1}, and the definition \eqref{conjugate_response} of the conjugated response $S$, we can write
\begin{align}
	\rho_1 &= \lim_{\eta\to0}\frac{\E_\eta(R)}{\eta} = \int_\mathcal{X} R\f_1 \, \psi_0 = -\int_\mathcal{X} (\Lr^{-1} R)(\Lp^*\ind) \, \psi_0 \\
	&= \int_0^{+\infty} \E_0(R(X_t) S(X_0)) \, dt,
	\label{gk}
\end{align}
where the expectation $\E_0$ is over all initial conditions $X_0$ distributed according to the invariant measure $\psi_0$, and over all realizations of the reference dynamics \eqref{general_SDE} at hand. For Langevin dynamics \eqref{lang_noneq}, the conjugate response function reads
\begin{equation}
	S(q,p) = \beta F(q)^TM^{-1}p.
	\label{conj_res_lang}
\end{equation}
Similarly, for overdamped Langevin dynamics \eqref{neld_ovd}, it reads
\begin{equation}
  S(q) = \beta F(q)^T\nabla V(q).
  \label{conj_res_ovd}
\end{equation}
In both cases $S \in L_0^2(\psi_0)$. We emphasize that, as already discussed in Section \ref{subsec:ref_dyn}, the expression for $S$ comes from the generator $\Lp$ of the perturbation, and not the observable $R$.

 The Green--Kubo formula shows that a nonequilibrium property (the transport coefficient $\rho_1$ in this case) can be computed using simulations at equilibrium, \emph{i.e.} for~$\eta=0$.

\paragraph{Standard numerical techniques} Transport coefficients are often estimated in one of two ways:

\begin{enumerate}
    \item \emph{Equilibrium techniques based on the Green--Kubo formula} \eqref{gk}. To numerically estimate the quantity \eqref{gk}, one needs to discretize the continuous dynamics in time, using a fixed timestep $\Delta t>0$. This leads to the presence of some timestep bias of order $\bigO(\Delta t^\theta)$, where~$\theta$ depends on the numerical method at hand \cite{averages,pde_mono}. Additionally, the time integral must be truncated to some finite integration time $T$. This leads to some truncature bias, which is small due to the exponential convergence of $\e^{t\L_0}$. 
    	
    	Last but not least, the expectation is computed using empirical averages over~$K$ realizations. This naturally suggests the following estimator for $\rho_1$:
    	\begin{equation}
  			\widehat{\rho}_{K,T} = \frac{1}{K} \sum_{k=1}^K \int_0^T R(X_t^k)S(X_0^k) \, dt.
  			\label{GKestimator}
		\end{equation}
		Although there are clear advantages to using equilibrium techniques (for instance, correlation functions for different conjugate responses can be computed simultaneously), there is also one major challenge in using \eqref{gk}. The integral is a correlation term, which is a small quantity for large~$t$, plagued by a large statistical error \cite{gk1}. The statistical error is therefore the main source of error, as the variance is expected to scale linearly with the integration time~$T$. More precisely, it is shown in \cite[Section 5]{sensitivity} that the statistical error of~$\widehat{\rho}_{K,T}$ is of order~$T/K$.
		
		Overall, a tradeoff has to be considered for the choice of $T$ as the bias is smaller for larger $T$ while the variance increases with $T$.
    \item \emph{Nonequilibrium steady-state techniques}. This method works by first approximating the limit in \eqref{linRes} by the finite difference $\E_{\eta}(R)/\eta$, with $\eta$ sufficiently small to limit the bias; and next estimating the expectation with time averages as
    \begin{equation}
        \widehat{\Phi}_{\eta, t}=\frac{1}{\eta t} \int_0^t R(X_s^{\eta}) \, ds,
        \label{estimator}
    \end{equation}
    for response functions with average 0 with respect to $\psi_0$. When computing such steady-state averages over long trajectories, the asymptotic variance of the trajectory average computed using the discretized dynamics coincides at dominant order in $\eta$ with the asymptotic variance of the trajectory average computed using the corresponding continuous dynamics (see Proposition \ref{proposition1} below).
    
     One source of error is the systematic error due to three different biases. As discussed in Proposition~\ref{proposition1}, the finiteness of the integration time leads to some bias of order $1/(\eta t)$, which is typically smaller than the statistical error. Additionally, the fact that we consider $\eta\ne 0$ leads to some bias of order $\eta$, as a consequence of Lemma \ref{lemma1}. Lastly, the time discretization of the continuous dynamics also leads to some time step bias; see~\cite{averages}.
    
     As discussed in Proposition \ref{proposition1} below, the statistical error is dictated by a central limit theorem, so the variance of the estimator $\widehat{\Phi}_{\eta, t}$ scales as $1/(\eta^2 t)$. The simulation time required to estimate $\rho_1$ with a sufficient statistical accuracy therefore scales as $t \sim \eta^{-2}$, leading to very long integration times $t$. Such long simulation times are often prohibitive in practical cases of interest.
     
     These results, and the tradeoffs to be considered, are discussed in more detail in Section \ref{subsec:analysis}.
\end{enumerate}
\subsubsection{Application to mobility}
\label{subsubsec:mobility}
The aim of this section is to illustrate the various previous results in the paradigmatic case of mobility. We consider the case where $F \in \R^d$ is a constant force, and the state-space is $\mathcal{X} = \T^d$. We define the mobility for both overdamped Langevin and Langevin dynamics, and discuss how it is related to the self-diffusion by Einstein's relation. 

From a physical point of view, it is expected that a nonzero constant force in some given direction induces a response from the system. At steady-state, this response is represented by some nonzero flux, due to the fact that the forcing $F$ is not the gradient of a periodic function. This nonzero flux depends both on the perturbation and on the observable $R$ in question.
 
 For the Langevin dynamics \eqref{lang_noneq}, the perturbation $\eta F$ is expected to induce a nonzero velocity in the direction~$F$. The mobility is the proportionality constant between the externally applied force $F$ and the observed average velocity in the direction~$F$. Therefore, it is natural to consider the observable
\begin{equation}
    R(p) = F^TM^{-1}p.
    \label{R_lang}
\end{equation}
This gives us the following expression for the mobility, in view of \eqref{linResf1}:
\begin{equation}
	\rho_1 = \int_{\T^d \times \R^d} (F^TM^{-1}p)\f_1 \, \psi_0.
	\label{mobility} 
\end{equation}
It can be rewritten using the Green--Kubo formula \eqref{gk} and the expression \eqref{conj_res_lang} of the conjugate response as
\begin{equation}
 	\rho_1 = \beta \int_0^{+\infty} \E_0\left[\left(F^TM^{-1}p_t\right)\left(F^TM^{-1}p_0\right)\right] dt.
\end{equation}
From this expression, it is easy to see that the mobility is related to the self-diffusion coefficient $D_F$ as (see for instance \cite{averages})
\begin{equation}
	\rho_1 = \beta D_F,
	\label{einstein_eq}
\end{equation}
where
\begin{equation}
	D_F = \lim_{t\to+\infty} \frac{\E\left[(F^T(Q_t-Q_0))^2\right]}{2t},
\end{equation}
with
\begin{equation}
	Q_t = Q_0 + \int_0^t M^{-1}p_s \, ds
\end{equation}
the unperiodized displacement. The formula \eqref{einstein_eq} for $D_F$ is known as Einstein's relation \cite{einstein}.

For overdamped Langevin dynamics \eqref{eq_ovd}, there is no notion of velocities. The system is however expected to drift in the direction $F$. This can be quantified by how much the gradient part of the force changes in the direction $F$. Thus, it is natural to consider the following observable
\begin{equation}
    R(q) = F^T\nabla V(q).
    \label{R_ovd}
\end{equation}
The mobility is then defined as the average projected force in the direction of the perturbation
\begin{equation}
	\rho_1 = \int_{\T^d} \left(F^T\nabla V\right) \f_1 \, \psi_0.
\end{equation}
For overdamped Langevin, the mobility is related to the self-diffusion $D_F$ through the following equality (see for instance \cite{fathi2015}):
\begin{equation}
    \beta D_F = |F|^2 - \rho_1 = |F|^2 - \beta\int_0^{+\infty} \E_0\left[(F^T\nabla V(q_t))(F^T\nabla V(q_0))\right] dt,
    \label{eq77}
\end{equation}
where the second expression involves the Green-Kubo formula for the linear response of $F^T\nabla V$.

\subsection{Numerical Analysis for NEMD}
\label{subsec:analysis}
In this section, we perform error analysis on the estimator \eqref{estimator}. We obtain bounds on the variance, then on the finite-time integration bias. Without loss of generality, we consider response functions of the form~$R = \Pi_0 R$, \emph{i.e.} functions with zero average with respect to the invariant measure~$\psi_0$ of the reference dynamics. The estimator $\widehat{\Phi}_{\eta, t}$ defined in \eqref{estimator} converges almost surely, as~$t\to\infty$ to
\begin{equation}
	\widehat{\rho}_{1,\eta} = \frac{1}\eta \int_\mathcal{X} R \, \psi_\eta = \rho_1 + \bigO(\eta),
	\label{rhoHat_bias}
\end{equation}
where the last equality comes from Lemma \ref{lemma1}. However, $\widehat{\Phi}_{\eta, t}$ suffers both from a large asymptotic variance, of order $\sigma_{R, 0}^2 / \eta^2$ (with $\sigma_{R,0}^2$ the asymptotic variance for time averages of $R$ computed with the reference dynamics), and a large finite-time sampling bias, of order $1/(\eta t)$. The aim of this section is to make precise the latter two statements.

\paragraph{Bounds on the statistical error} The scaling of the statistical error is quantified in the following result.

\begin{proposition}
\label{proposition1}
Suppose that Assumptions \ref{assumption1} through \ref{assumption4} hold true. Fix $R \in \S_0$ and $\eta \in \R$. Assume that $X_0^\eta \sim \mu_{\rm{init}}$ for some initial probability measure $\mu_{\rm{init}}(dx)$ such that $\mu_{\mathrm{init}}\left(\mathcal{K}_n\right) < +\infty$ for any $n \geqslant 1$. Then the estimator $\widehat{\Phi}_{\eta, t}$ converges almost surely to $\widehat{\rho}_{1,\eta}$ as $t \to +\infty$, and the following central limit theorem holds:
\begin{equation}
	\sqrt{t} \left(\widehat{\Phi}_{\eta,t} - \widehat{\rho}_{1,\eta}\right) \xrightarrow[t\to+\infty]{\rm{law}} \mathcal{N}\left(0,\frac{\sigma^2_{R,\eta}}{\eta^2}\right).
    \label{CLT_prop1}
\end{equation}
Moreoever, there exists $\widetilde{\sigma}_{R, 0}$ such that for any $\eta_* \in (0,+\infty)$, there is $C \in \mathbb{R}_+$ (which depends on $\eta_*$ and $R$) for which
\begin{equation}
    \forall |\eta|\leq \eta_*, \qquad\left|\sigma_{R, \eta}^2 - \sigma_{R, 0}^2 - \eta\widetilde{\sigma}_{R, 0}^2\right| \leqslant C\eta^2.
    \label{eq9_sticky}
\end{equation}
\end{proposition} 
This result shows that simulation times of order $t \sim \eta^{-2}$ should be considered in order for the variance of the naive estimator \eqref{estimator} to be of order 1 , and also for its bias to be of order $\eta$, \emph{i.e.} of the same order of magnitude as the bias $\rho_1 - \widehat{\rho}_{1,\eta}$ arising from choosing $\eta\ne 0$. For completeness, the proof of \eqref{eq9_sticky} is done at second-order, in order to determine the expression of the term $\widetilde{\sigma}_{R, 0}^2$ characterizing the first-order variation of $\sigma_{R,\eta}$ with respect to $\eta$. 
\begin{proof}
The central limit theorem \eqref{CLT_prop1} holds by the results of \cite{CLT}, since the Poisson equation $-\mathcal{L}_\eta \widehat{R}_\eta = \Pi_\eta R$ has a unique solution in $\Pi_\eta B^\infty_n \subset L^{2}\left(\psi_\eta\right)$ for some integer~$n\geq 1$ in view of \eqref{est1} and \eqref{est3}. Note that $\Pi_\eta\widehat{R}_\eta = \widehat{R}_\eta$. To prove \eqref{eq9_sticky}, we first write the asymptotic variance as (see for instance~\cite[Section 3]{pde_mono})

\begin{equation}
    \sigma_{R, \eta}^2 = 2 \int_\mathcal{X} R \left(\Pi_\eta \widehat{R}_\eta\right) \psi_\eta .
    \label{variance}
\end{equation}
In view of Lemma \ref{lemma2} below, we introduce $\widetilde{R} = -\mathcal{L}_0^{-1} \Pi_0 \Lp \widehat{R}_0 \in \S$, so that \eqref{variance} can be expanded as
\begin{equation}
    \int_\mathcal{X} R \Pi_\eta\widehat{R}_\eta \, \psi_\eta = \int_\mathcal{X} R \Pi_\eta\widehat{R}_0 \, \psi_\eta + \eta \int_\mathcal{X} R \Pi_\eta\widetilde{R} \, \psi_\eta + \eta^2 \mathcal{R}_\eta,
    \label{lemma2_expansion}
\end{equation}
with the remainder term
\begin{equation}
	\mathcal{R}_\eta = \int_\mathcal{X} R \frac{\Pi_\eta\left(\widehat{R}_\eta - \widehat{R}_0 - \eta \widetilde{R}\right)}{\eta^2} \, \psi_\eta.
\end{equation}
By Lemma \ref{lemma2} and Assumption \ref{assumption2}, the remainder $\mathcal{R}_\eta$ is uniformly bounded for~$\eta \in\left[-\eta_*, \eta_*\right]$. Note that, since $\Pi_0\widehat{R}_0 = \widehat{R}_0$ and $\Pi_0 R = R$, we can write
\begin{align}
	\int_\mathcal{X} R \Pi_\eta\widehat{R}_0 \, \psi_0 = \int_\mathcal{X} R \widehat{R}_0 \, \psi_0 - \left(\int_\mathcal{X} \widehat{R}_0 \, \psi_\eta\right)\left(\int_\mathcal{X} R \, \psi_0\right) = \int_\mathcal{X} R \Pi_0\widehat{R}_0 \, \psi_0.
\end{align}
The same argument is valid for $\widetilde{R}$. We next use Lemma \ref{lemma1} to write the two integrals on the right hand side of \eqref{lemma2_expansion} as
\begin{equation}
    \int_\mathcal{X} R \Pi_\eta\widehat{R}_\eta \, \psi_\eta = \int_\mathcal{X} R \Pi_0\widehat{R}_0 \, \psi_0 + \eta \int_\mathcal{X} R\left(\widehat{R}_0 \mathfrak{f}_1 + \Pi_0\widetilde{R}\right) \psi_{0} + \eta^2 \widetilde{\mathcal{R}}_\eta,
    \label{question1}
\end{equation}
where
\begin{equation}
\begin{aligned}
    \widetilde{\mathcal{R}}_\eta = \mathcal{R}_\eta &+ \mathscr{R}_{\eta, R \widehat{R}_0,2} + \eta \mathscr{R}_{\eta, R \Pi_0\widetilde{R},1} + \int_\mathcal{X} R\widetilde{R}\f_1 \, \psi_0 \\ 
    &- \frac{1}{\eta^2}\left(\int_\mathcal{X} \widehat{R}_0 \, \psi_\eta\right)\left(\int_\mathcal{X} R \, \psi_\eta\right) - \frac{1}{\eta}\left(\int_\mathcal{X} \widetilde{R} \, \psi_\eta\right)\left(\int_\mathcal{X} R \, \psi_\eta\right).
    \label{Rtilde_eta}
\end{aligned}
\end{equation}
Equation \eqref{question1} leads to $\sigma_{R, \eta}^2 - \sigma_{R, 0}^2 - \eta\widetilde{\sigma}_{R, 0}^2 = \eta^2 \widetilde{\mathcal{R}}_\eta$, with
\begin{equation}
	\widetilde{\sigma}_{R, 0}^2 = 2\int_\mathcal{X} R\left(\widehat{R}_0 \mathfrak{f}_1 + \Pi_0\widetilde{R}\right) \psi_{0}.
\end{equation}
In view of Lemma \ref{lemma1} and Assumption \ref{assumption2}, the remainder \eqref{Rtilde_eta} is uniformly bounded for $\eta \in\left[-\eta_{*}, \eta_{*}\right]$. This proves that \eqref{eq9_sticky} holds.
\end{proof}

We conclude this section with a technical result used in the proof of Proposition~\ref{proposition1}.

\begin{lemma}
\label{lemma2}
Suppose that Assumptions \ref{assumption1} through \ref{assumption4} hold true. Fix $\eta_*>0$ and~$\varphi \in \S$. Denote by $m\geq 1$ an integer such that $\varphi\in B^\infty_m$. Consider for any $\eta \in \mathbb{R}$ the unique solution $\phi_{\eta} \in \Pi_\eta B^\infty_m$ of the Poisson equation $-\mathcal{L}_\eta\phi_\eta = \Pi_\eta \varphi$, and define~$\widetilde{\phi} = -\mathcal{L}_0^{-1} \Pi_0 \Lp\phi_0 \in \S$. Then, there exists $n \geq 1$ and $K \in \mathbb{R}_+$ such that
\begin{equation}
    \forall|\eta|\leq \eta_*, \qquad \left\|\Pi_\eta\left(\phi_\eta - \phi_0 - \eta \widetilde{\phi}\right)\right\|_{B_n^\infty} \leqslant K \eta^{2} .
\end{equation}
\end{lemma}

\begin{proof}
Since $\L_\eta = \L_0 + \eta\Lp$, a simple computation shows that
\begin{equation}
    -\mathcal{L}_\eta\left(\phi_\eta - \phi_0 - \eta\widetilde{\phi}\right) = \left(\Pi_\eta - \Pi_0\right) \varphi + \eta\left(1 - \Pi_0\right) \Lp \phi_0 + \eta^2 \Lp \widetilde{\phi}.
\end{equation}
In view of Lemma \ref{lemma1}, and since $\f_1$ has average 0 with respect to $\psi_0$,
\begin{align}
    \left(\Pi_\eta - \Pi_0\right) \varphi &= -\eta \int_\mathcal{X} \varphi \f_1 \, \psi_0 - \eta^2 \mathscr{R}_{\eta, \varphi,2} = -\eta \int_\mathcal{X}\left(\Pi_0 \varphi\right) \mathfrak{f}_1 \, \psi_0 - \eta^2 \mathscr{R}_{\eta, \varphi,2} \\
    &= \eta \int_\mathcal{X} \Lp \mathcal{L}_0^{-1} \Pi_0 \varphi \, \psi_0 - \eta^2 \mathscr{R}_{\eta, \varphi,2} = -\eta\int_\mathcal{X}\left(\Lp\phi_0\right) \, \psi_0 - \eta^2 \mathscr{R}_{\eta, \varphi,2} \\
    &= -\eta\left(1 - \Pi_0\right) \Lp \phi_0 - \eta^2 \mathscr{R}_{\eta, \varphi,2}.
    \label{last_line}
\end{align}
Therefore, $-\mathcal{L}_\eta\left(\phi_\eta - \phi_0 - \eta\widetilde{\phi}\right) = \eta^2 \Lp \widetilde{\phi} - \eta^2 \mathscr{R}_{\eta, \varphi,2} = \eta^2 \Pi_\eta\Lp \widetilde{\phi}$, because the right-hand side is $\eta^2 \Lp \widetilde{\phi}$ up to a constant term, and has to be in the image of $\L_\eta$. It is clear that $\Lp \widetilde{\phi}\in\S$, as $\widetilde{\phi}\in\S$ and $\Lp$ stabilizes $\S$ by Assumption \ref{assumption4}. Thus, there exists $n\geq1$ such that $\Lp \widetilde{\phi}\in B^\infty_n$. Using the definition of the operator norm,
\begin{align}
  \left\|\Pi_\eta\left(\phi_\eta - \phi_0 - \eta\widetilde{\phi}\right)\right\|_{B^\infty_n} &= \eta^2\left\|\L_\eta^{-1}\left(\Pi_\eta\Lp\widetilde{\phi}\right)\right\|_{B^\infty_n} \\
  &\leq \eta^2\|\L_\eta^{-1}\|_{\mathcal{B}(\Pi_\eta B^\infty_n)}\left\|\Pi_\eta(\Lp\widetilde{\phi})\right\|_{B^\infty_n},
\end{align}
where $\|\L_\eta^{-1}\|_{\mathcal{B}(\Pi_\eta B^\infty_n)}$ is uniformly bounded in view of \eqref{est3}, as is $\|\Pi_\eta(\Lp\widetilde{\phi})\|_{B^\infty_n}$ by \eqref{est1}. This gives the desired result.
\end{proof}

\paragraph{Bounds on the finite-time bias} For completeness, we also state a result on the finite-time bias of the estimator, which essentially says that this bias is of order~$1/(\eta t)$. For technical reasons, this estimate however has to be formulated in a more cumbersome way. Nevertheless, bounds on the statistical error given in Proposition \ref{proposition1} are more important in practice, which is why we did not try to improve the bounds below.

\begin{lemma}
\label{finite_time_bias_lemma}
Consider the same setting as Proposition \ref{proposition1}, and assume that $\Sigma = \sigma\sigma^T \in \S$. Then, for any $k\geq 1$ and any $\eta_*>0$, there exist $\mathcal{C}_k$ and $\mathcal{M}_k$ (which depend and~$R$) such that
\begin{equation}
    \forall |\eta|\leq \eta_*, \quad \forall t>0, \qquad \left|\mathbb{E}\left(\widehat{\Phi}_{\eta, t}\right)-\widehat{\rho}_{1,\eta}\right| \leqslant \frac{\mathcal{C}_k}{\eta t} + \mathcal{M}_k\eta^k.
    \label{eq10_sticky}
\end{equation}
\end{lemma}

\begin{proof}
	This result would be easy to prove if $\widehat{R}_{\eta} \in \S$, where $\widehat{R}_{\eta}$ is the unique solution to the Poisson equation $-\L_\eta\widehat{R}_{\eta} = \Pi_\eta R$ discussed in the proof of Proposition~\ref{proposition1}. However, there is no result that ensures that this property holds, so we turn to an alternative proof where we approximate $\widehat{R}_{\eta}$ with high precision by $\widehat{R}_{\eta,k} \in \S$. More precisely, we introduce~$\widehat{R}_{\eta,k} = Q_{\eta,k+1}\Pi_\eta R$, where $Q_{\eta,k}$ is the pseudo-inverse operator defined in \eqref{pseudo_inverse}, so that $\widehat{R}_{\eta,k}\in\S$. Since~$\widehat{R}_{\eta,k} \in C^\infty(\mathcal{X})$, we use It\^o's formula to write
\begin{equation}
    d\widehat{R}_{\eta,k}(X^\eta_t) = \L_\eta \widehat{R}_{\eta,k}(X^\eta_t) \, dt + \nabla \widehat{R}_{\eta,k}(X^\eta_t)^T \Sigma(X^\eta_t) \, dW_t.
    \label{lemma2_ito}
\end{equation}
The martingale term in \eqref{lemma2_ito} is square integrable since, in view of \eqref{est2} and the fact that~$\Sigma, \widehat{R}_{\eta,k} \in \S$ (so that $|\Sigma\nabla \widehat{R}_{\eta,k}|^2 \in \S$), there exists $\ell\in\N$ such that
\begin{align}
	\int_0^t \E\left[\left|\Sigma(X^\eta_s)\nabla \widehat{R}_{\eta,k}(X^\eta_s)\right|^2\right] ds &= \int_\mathcal{X}\int_0^t \e^{s\L_\eta} \left(\left|\Sigma\nabla \widehat{R}_{\eta,k}\right|^2\right) ds \, d\mu_\mathrm{init} \\
	&\leq \left\|\left|\Sigma\nabla\widehat{R}_{\eta,k}\right|^2\right\|_{B^\infty_\ell} \int_\mathcal{X}\int_0^t \e^{s\L_\eta} \mathcal{K}_\ell \, ds \, d\mu_\mathrm{init}
	\label{martingale_sqr_int}
\end{align}
is uniformly bounded for $|\eta| \leq \eta_*$. Next, we write
\begin{equation}
	\widehat{\Phi}_{\eta, t} - \widehat{\rho}_{1,\eta} = -\frac{1}{\eta t}\int_0^t \L_\eta \widehat{R}_{\eta,k} (X_s^\eta) \, ds + \eta^{k-1}\mathscr{R}_{\eta,k},
	\label{est_diff}
\end{equation}
with remainder term
\begin{equation}
	\mathscr{R}_{\eta,k} = \frac{1}{\eta^k t}\int_0^t \L_\eta\left(\widehat{R}_{\eta,k} - \widehat{R}_{\eta}\right)\left(X_s^\eta\right) ds.
\end{equation}
%
%
In view of \eqref{lemma2_ito}, we write \eqref{est_diff} as
\begin{align}
        \widehat{\Phi}_{\eta, t} - \widehat{\rho}_{1,\eta} = \frac{\widehat{R}_{\eta,k}\left(X_{0}^\eta\right) - \widehat{R}_{\eta,k}\left(X_t^\eta\right)}{\eta t}+\frac{1}{\eta t} \int_0^t \nabla \widehat{R}_{\eta,k}\left(X_s^\eta\right)^T\Sigma(X^\eta_s) \, dW_s + \eta^{k-1} \mathscr{R}_{\eta,k}.
    \label{question2}
\end{align}
In view of~\eqref{est2} and \eqref{lemma1_A_bound}, $\E[|\mathscr{R}_{\eta,k}|]$ is uniformly bounded for $|\eta| \leq \eta_*$ by $\mathcal{M}_k\in\R_+$. By taking expectations, \eqref{question2} then leads to
\begin{equation}
    \left|\mathbb{E}\left(\widehat{\Phi}_{\eta, t}\right) - \widehat{\rho}_{1,\eta}\right| \leq \frac{1}{\eta t} \left|\mathbb{E}\left[\widehat{R}_{\eta,k}\left(X_0^\eta\right) - \widehat{R}_{\eta,k}\left(X_t^\eta\right)\right]\right| + \eta^k \mathcal{M}_k.
\end{equation}
Since $\widehat{R}_{\eta,k}\in \S$, there exists $n\in\N$ such that $\widehat{R}_{\eta,k}\in B^\infty_n$. By Assumptions \ref{assumption3} and \ref{assumption4}, there exists $n'\in\N$ and $K'_{n',\eta_*}\in\R_+$ (depending on $n, n', \eta_*$ and $k$) such that
\begin{equation}
	\left\|\widehat{R}_{\eta,k}\right\|_{B^\infty_{n}} \leq K'_{n',\eta_*}\|\Pi_\eta R\|_{B^\infty_{n'}}.
	\label{Rhat_bound}
\end{equation}
In view of \eqref{Rhat_bound} and \eqref{est2},
\begin{align}
	\left|\E\left(\widehat{R}_{\eta,k}(X_t^\eta)\right)\right| &\leq \E(\mathcal{K}_n(X_t^\eta)) \left\|\widehat{R}_{\eta,k}\right\|_{B^\infty_n} \\
	&\leq K'_{n',\eta_*}\sup_{t\in\R_+} \left(\int_\mathcal{X} \e^{t\L_\eta}\mathcal{K}_n \, d\mu_\text{init}\right) \left\|\Pi_\eta R\right\|_{B^\infty_{n'}} \\
	&\leq K'_{n',\eta_*}M_{n,\eta_*} \left\|\Pi_\eta R\right\|_{B^\infty_{n'}}\int_\mathcal{X} \mathcal{K}_n \, \mu_\mathrm{init},
\end{align}
This leads to \eqref{eq10_sticky} with
\begin{equation}
    \mathcal{C}_k = 2K'_{n',\eta_*}M_{n,\eta_*}\int_\mathcal{X}\mathcal{K}_n \, \mu_\mathrm{init} < +\infty,
\end{equation}
thus concluding the proof. 
\end{proof}

\section{Extending the range of linear response with synthetic forcings}
\label{sec:extending}
We discuss in this section the notion of synthetic forcings, and how they can be used to extend the regime of linear response. We start by describing the notion of synthetic forcings in Section \ref{subsec:notion}, and give examples in Section \ref{subsec:examples} for both overdamped and underdamped Langevin dynamics. Then, we provide a methodology for choosing the magnitude of the forcings in Section \ref{subsec:choosing_alpha}. Finally, we briefly discuss how to linearly combine multiple extra forcings in Section \ref{subsec:lin_comb}.

\subsection{Notion of synthetic forcings}
\label{subsec:notion}
As discussed in Section \ref{subsubsec:gen_set}, a system with a nonequilibrium perturbation has a generator of the form
\begin{equation}
	\L_\eta = \L_0 + \eta\widetilde{\L},
	\label{gen_noneq}
\end{equation}
where $\widetilde{\L}$ is the generator of some perturbation to the reference dynamics with generator $\Lr$. We considered in Section \ref{sec:review} only the case where $\widetilde{\L}$ is $\Lp$ -- in other words, the perturbation corresponds to some physical perturbation on the system, which is the typical scenario in the context of statistical physics. We now consider, in addition to the physical perturbation, some possibly nonphysical extra perturbation, which we denote by $\Lx$, so that $\widetilde{\L}$ in \eqref{gen_noneq} is replaced by
\begin{equation}
	\widetilde{\L} = \Lp + \alpha\Lx,
\end{equation}
for $\alpha\in\R$. We call the resulting perturbation a \emph{synthetic forcing}.

The key requirement of synthetic forcings is that the the addition of the extra forcing should preserve the invariant measure of the reference dynamics, thus preserving the linear response. In other words, the dynamics with generator
\begin{equation}
	\L_{\eta,\alpha} = \Lr + \eta\left(\Lp + \alpha \Lx\right)
\end{equation}
has the same linear response as the dynamics associated with $\L_{\eta,0} = \Lr + \eta\Lp$ when
\begin{equation}
	\Lx^*\ind = 0.
	\label{synthetic_condition}
\end{equation}
We denote by $\psi_{\eta,\alpha}$ the invariant measure for the dynamics associated with the generator $\L_{\eta,\alpha}$, \emph{i.e.} the solution to the stationary Fokker--Planck equation
\begin{equation}
	\L_{\eta,\alpha}^\dagger\psi_{\eta,\alpha} = 0.
	\label{FP_eta_alpha}
\end{equation} 
When there is no indication of the value of $\alpha$, \emph{e.g.} $\psi_\eta$, it means that $\alpha = 0$.

From the definition \eqref{linResf1} of the linear response $\rho_1$, it is indeed easy to see why the linear response is preserved with the addition of $\Lx$: the conjugate response $S$ defined in \eqref{conjugate_response} is preserved, and so is $\f_1$ in view of~\eqref{f1}, which allows us to conclude by Lemma \ref{lemma1}. Let us emphasize that \eqref{synthetic_condition} is the key condition to be satisfied for extra forcings to be admissible. We provide various examples of admissible extra forcings in Section \ref{subsec:examples}. For technical reasons, the extra forcings we consider should satisfy the same conditions as $\Lp$ in Assumption \ref{assumption4}.

One practical interest is to optimize the extra perturbation in order to increase the regime of linear response. As made precise in Proposition \ref{proposition1}, the variance of the estimator \eqref{estimator} is of order~$\bigO(\eta^{-2})$. A larger linear regime therefore means that larger values of $\eta$ can be considered without introducing too much bias on $\rho_1$. This leads in turn to a smaller statistical error, and hence shorter simulation times to reach the same accuracy. One idea in particular is to look for $\widetilde{\mathcal{L}}_\mathrm{extra}$ which minimizes $\rho_2$, the leading order~$\eta^2$ of the nonlinear response, as a proxy for minimizing $|\mathbb{E}_\eta(R)-\rho_1\eta|$, \emph{i.e.} the nonlinear portion of the response. This is discussed in detail in Section \ref{subsec:choosing_alpha}. This also naturally suggests that one could further combine~$k$ forcings in order to cancel the first $k+1$ orders of the response, as discussed in Section \ref{subsec:lin_comb}.

Another approach to optimizing the perturbation is to increase the range of $\eta$ for which the nonlinear response is within some desired distance from the linear regime, in relative error, also discussed Section \ref{subsec:choosing_alpha}.

\subsection{Examples of synthetic forcings}
\label{subsec:examples}
To make synthetic forcings more concrete, we now go over some examples of extra forcings. We first outline in Section~\ref{subsubsec:gen_extra_forcings} the general classes of operators we consider. We then discuss more precisely examples for overdamped Langevin dynamics in Section \ref{subsubsec:ovd_examples} and for underdamped Langevin dynamics in Section \ref{subsubsec:lang_examples}.

\subsubsection{General classes of extra forcings}
\label{subsubsec:gen_extra_forcings}
When considering possible extra forcings, we restrict ourselves to differential operators of at most second-order in order to realize them in Monte--Carlo simulations. We consider the following classes of differential operators:
\begin{enumerate}
    \item \emph{First-order differential operators} $\Lx = G^T \nabla_x$, with $G\colon\mathcal{X} \to \R^d$ such that $\div(G\psi_0) = 0$. The latter condition ensures that \eqref{synthetic_condition} is satisfied since~$\Lx^* = -G^T \nabla_x$;
    \label{item1}
    \item \emph{Second-order differential operators} of the form $\Lx = -\partial_{x_i}^*\partial_{x_i}$ or more generally $-\partial_{x_j}^*D_{ij}\partial_{x_i}$ for some (nonnegative) function $D_{ij}\colon\mathcal{X}\to\R_+$. In the case $\Lx = -\partial_{x_i}^*\partial_{x_i}$, the operator is self-adjoint, \emph{i.e.} $\Lx = \Lx^*$, so that~\eqref{synthetic_condition} is easily seen to hold. For $\Lx = -\partial_{x_j}^*D_{ij}\partial_{x_i}$, it holds that~$\Lx^* = -\partial_{x_i}^*D_{ij}\partial_{x_j}$, which also satisfies \eqref{synthetic_condition}.
    \label{item2}
    \item \emph{First-order differential operators with nontrivial zero order parts}, such as $\Lx = \partial_{x_i}^* = \partial_{x_i}U - \partial_{x_i}$ for $\psi_0(x) = \e^{-U(x)}$. Through some abuse of notation, we use $\partial_{x_i}U$ to denote the multiplication operator by the function~$\partial_{x_i}U$. These operators satisfy $\Lx^* = \partial_{x_i}$, so that \eqref{synthetic_condition} holds.
    \label{item3}
\end{enumerate}
%

The class of extra forcings outlined in items \ref{item1} and \ref{item2} can be easily implemented and realized in Monte--Carlo simulations. Implementing forcings of the form described in item \ref{item3}, however, requires some extra work to take care of the multiplication operator, as we now discuss.

We denote the class of extra forcings described in item \ref{item3} as \emph{Feynman--Kac forcings}, as sampling the dynamics requires the use of the Feynman--Kac formula due to the nontrivial zero order term. Consider the general dynamics~\eqref{general_SDE} with generator~\eqref{gen_generator}, as presented in Section \ref{subsubsec:gen_set}. Suppose that the dynamics has a unique invariant probability measure with density $\psi_0(x) = \e^{-U(x)}$. The perturbed dynamics with generator $\L_{\eta,\alpha} = \Lr + \eta\left(\Lp + \alpha\Lx\right)$, with $\Lx = \xi^T\nabla^*$ for some $\xi\in\R^d$, can be sampled by evolving the SDE
\begin{equation}
	dX_t = \left(b(X_t) + \eta F(X_t) - \eta\alpha\xi\right) dt + \sigma(X_t) \, dW_t,
\end{equation}
and weighting trajectories with the Feynman--Kac weight
    \begin{equation}
  		\omega_t = \exp\left(\eta\alpha\int_0^t \xi^T\nabla U(X_s) \, ds\right).
  		\label{FK_gen_weight}
	\end{equation}
    In practice, this is done by evolving multiple replicas of the system with independent Brownian motions, and using resampling strategies to prevent the weights from degenerating (see for instance~\cite{delmoral} and~\cite[Chapter 6]{freeEnergy}).
    
    From an analytical point of view, items \ref{item1} and \ref{item2} fit in the framework of Section~\ref{sec:review}. Item \ref{item3}, however, does not directly fit in the framework of Section \ref{sec:review}. This is due to the fact that the steady-state measure of the dynamics with generator~$\L_{\eta,\alpha} = \Lr + \eta(\Lp + \alpha\Lx)$ is not a probability measure in general. This is the case only if the weights \eqref{FK_gen_weight} are renormalized. At the level of generators, this amounts to shifting the spectrum of $\L_{\eta,\alpha}$ by $\lambda_{\eta,\alpha}$, where $\lambda_{\eta,\alpha}$ is the nonzero real principal eigenvalue of the operator $\L_{\eta,\alpha}$ (see \cite{ferre_stoltz_2019,ferre2021} and references therein). The magnitude of $\lambda_{\eta,\alpha}$ can be made precise in terms of $\alpha$ and $\eta$ for $\eta$ small, and turns out to be of order $\bigO(\eta^2)$ for~$\eta$ small, as formally derived in Appendix \ref{app:FK_eig}.

Of course, by linearity, one can consider linear combinations of these extra forcings. In the next two sections, we give specific examples of each of the three classes of extra forcings discussed here for both overdamped and underdamped Langevin dynamics.

\subsubsection{Overdamped Langevin dynamics}
\label{subsubsec:ovd_examples}
We start by outlining some examples of synthetic forcings for the overdamped Langevin dynamics introduced in Section \ref{subsubsec:ovd_review}. We consider successively the distinct classes of extra forcings following the general presentation of Section \ref{subsubsec:gen_extra_forcings}.

\begin{example}[{\bf Divergence-free vector field}] We consider a first-order differential operator
    \begin{equation}
        \Lx = G^T\nabla,
        \label{divFree}
    \end{equation}
    where $G \colon \mathcal{X} \to \R^d$ is such that
    \begin{equation}
        \div\left(G(q)\e^{-\beta V(q)}\right) = 0.
        \label{divFree_condition}
    \end{equation}
    The resulting dynamics with the addition of \eqref{divFree} is
    \begin{equation}
        dq_t = \left[-\nabla V(q) + \eta F(q_t) + \alpha\eta G(q_t)\right] dt + \sqrt{\frac{2}{\beta}} \, dW_t.
    \end{equation}
    Condition \eqref{divFree_condition} can be rewritten as $\div(G) = \beta G\nabla V$. This equality is satisfied when~$G$ is both divergence-free and orthogonal to $\nabla V$, which is the situation considered in \cite{luc2015}. For instance, in any dimension $d\geq 2$, with $A$ an arbitrary anti-symmetric matrix, one possible choice for $G$ is
    \begin{equation}
    	G = A\nabla V.
    \label{A1_2D}
    \end{equation}
    As in \cite{general_divfree}, this can be generalized to $G = \Phi'(V)A\nabla V$, where $\Phi \colon \R\to\R$ is some smooth, compactly supported function. More generally, any divergence-free vector field in dimension $d$ can be written as $G = \nabla U_1 \times \cdots \times \nabla U_{d-1}$, where $U_i$ are scalar functions \cite{barbarosie}, a form which was used in \cite{luc2015}. Thus, more generally, $G$ satisfies \eqref{divFree_condition} if and only if it is of the form
    \begin{equation}
  		G = \left(\nabla U_1 \times \cdots \times \nabla U_{d-1}\right)\e^{\beta V}.
  		\label{gen_divFree}
	\end{equation}
	For the one-dimensional dynamics, the only divergence-free vector field is $G(q) = \e^{\beta V(q)}$. Of course, drifts such as \eqref{gen_divFree} may not be used as such in the dynamics when the position space is unbounded, as it is not clear whether the dynamics is well-posed because of the factor $\e^{\beta V}$ and when it is, whether it admits a unique invariant probability measure. Moreover, Assumption \ref{assumption4} may not hold.
\end{example}

\begin{example}[{\bf Modifying the fluctuation-dissipation relation}]
\label{mfdr}
    One possible choice for extra perturbations involving second-order derivatives is
    \begin{equation}
        \Lx = -\beta^{-1}\nabla_q^*\nabla_q = \beta^{-1}\Delta_q - \nabla V^T\nabla_q.
        \label{mfdr_ovd}
    \end{equation}
    Note that $\Lx = \L_0$, so the resulting generator of the perturbed dynamics can be written as
    \begin{equation}
  		\L_{\eta,\alpha} = \Lr + \eta\left(\Lp+\alpha\Lx\right) = (1+\alpha\eta)\Lr + \eta\Lp.
	\end{equation}
    The dynamics associated with $\L_{\eta,\alpha}$ reads
    \begin{equation}
    	dq_t = \left[-(1+\alpha \eta)\nabla V(q_t) + \eta F(q_t)\right] dt + \sqrt{\frac{2(1 + \alpha\eta)}{\beta}} \, dW_t.
    \end{equation}
    This amounts to increasing the magnitude of the terms involved in the fluctuation-dissipation as $\eta$ increases when $\alpha>0$. Note that, in order for this perturbation to be admissible, we require that $1+\alpha\eta>0$. 
    
    More generally, one could consider extra forcings of the form $-\nabla^* D(q)\nabla$ for some~$D\colon \mathcal{X} \to \R^{d\times d}$ with values in the space of symmetric matrix, possibly with $D$ constant.
\end{example}

\begin{example}[{\bf Feynman--Kac forcing}]
\label{ex:FK_ovd}
This choice, although it showcases great promise and potential in extending the linear regime, as we will demonstrate in Sections \ref{subsec:ovd_numerical_1d} and \ref{subsec:ovd_numerical_2d}, is not practical to be simulated by a single long realization of the dynamics as the weights degenerate, rendering the simulation inefficient. Its general form is
    \begin{equation}
        \Lx = \xi^T\nabla^* = \sum_{i=1}^d \xi_i \partial^*_{q_i},
        \label{fk_ovd}
    \end{equation}
    for some vector $\xi \in \R^d$. The generator $\Lr + \eta(\Lp + \alpha\Lx)$ is the sum of first and second-order differential operators and a weight $\xi^T\nabla V$. Its stochastic representation corresponds to evolving the SDE
    \begin{equation}
    	dq_t = \left(-\nabla V(q_t) + \eta F(q_t) - \alpha\eta\xi \right) dt + \sqrt{\frac{2}{\beta}} \, dW_t,
    \label{dyn1}
    \end{equation}
    and using the Feynman--Kac formula \eqref{FK_gen_weight} for the weight involving $\xi^T\nabla V$.
\end{example}

\subsubsection{Langevin dynamics}
\label{subsubsec:lang_examples}
We now outline some examples for Langevin dynamics, presented in Section \ref{subsubsec:lang_review}. Since this dynamics evolves two variables, namely positions $q$ and momenta $p$, we can in theory use both differential operators~$\nabla_q$ and~$\nabla_p$ to construct our synthetic forcings. This leads to additional options for each class of extra forcings.

\begin{example}[{\bf Divergence-free vector field}]
\label{ex:div_free_lang}
    The extra perturbation can be chosen as a first-order differential operator of the form
    \begin{equation}
  		\Lx = G^T\nabla = G_1^T\nabla_q + G_2^T\nabla_p,
  		\label{divFree_lang}
	\end{equation}
    where $G_1,G_2$ are such that
    \begin{equation}
  		\div_q(G_1 \psi_0) + \div_p(G_2 \psi_0) = 0.
	\end{equation}
	A perturbation of similar form to \eqref{divFree_lang} has been used and studied in \cite{nuesken_pavliotis}. For instance, in any dimension $d\geq 2$, a natural choice is $G = A\nabla H$, or more generally~$G = \Phi'(H)A\nabla H$, where $A$ is an antisymmetric matrix. Typical choices include the symplectic matrix
	\begin{equation}
		A = \begin{bmatrix}
			0 & I \\ 
			-I & 0
		\end{bmatrix},
		\label{symp_mat}
	\end{equation}
	or, more generally, linear combinations of matrices of the form
	\begin{equation}
		A = \begin{bmatrix}
			0 & B \\ 
			-B^T & 0
		\end{bmatrix} \qquad \text{or} \qquad 
		\begin{bmatrix}
			A_1 & 0 \\
			0 & A_2
		\end{bmatrix},
	\end{equation}
	with $A_1,A_2$ antisymmetric. In all these expressions, $A_1,A_2$ and $B$ can be functions of $(q,p)$. The choice \eqref{divFree_lang} leads to the generator
	\begin{equation}
  		\left(\Lx\varphi\right)(q,p) = G_1(q,p)^T\nabla_q + G_2(q,p)^T \nabla_p,
  		\label{example34_lang}
	\end{equation}
	and the dynamics
	\begin{align}
	\begin{split}
	\begin{cases}
	    dq_t = M^{-1} p_t \, dt + \alpha\eta G_1(q_t,p_t) \, dt, \\
	    dp_t = -\nabla V(q_t) \, dt + \eta(F(q_t) + \alpha G_2(q_t,p_t)) \, dt - \gamma M^{-1} p_t \, dt + \sqrt{\dfrac{2\gamma}{\beta}} \, dW_t.
	\end{cases}
	\end{split}
	\end{align}
	When $G=J\nabla H$ with $J$ the symplectic matrix, the extra perturbation corresponds to rescaling the Hamiltonian part of the dynamics, which is equivalent, up to a time rescaling, to changing the strength of the fluctuation-dissipation.
\end{example}

\begin{example}[{\bf Modified fluctuation-dissipation}] 
\label{ex:mod_fd_lang}
We first consider an operator that is second-order in $p$. A simple choice for the extra forcing is $\Lx = -\beta^{-1}\nabla^*_p\nabla_p$, so that
    \begin{equation}
        \L_{\eta,\alpha} = \L_\text{ham} + (\gamma + \alpha\eta)\L_\text{FD} + \eta\Lp.
    \end{equation}
    For this forcing to be admissible, we require that $1+\alpha\eta > 0$. The dynamics associated with $\L_{\eta,\alpha}$ reads
\begin{align}
\begin{split}
\begin{cases}
    dq_t = M^{-1} p_t \, dt, \\
    dp_t = -\nabla V(q_t) \, dt + \eta F(q_t) \, dt - \left(\gamma + \alpha\eta\right)M^{-1} p_t \, dt + \sqrt{\dfrac{2(\gamma + \alpha\eta)}{\beta}} \, dW_t.
\end{cases}
\end{split}
\end{align}
The effect of the extra forcing is to rescale the strength of the fluctuation-dissipation, either increasing it when~$\alpha\eta>0$, or reducing it when $\alpha\eta<0$. More generally, one can scale this forcing by a function of $q$, \emph{i.e.} consider $\Lx = -a(q)\nabla_p^*\nabla_p$ with $a(q)\colon \mathcal{X}\to\R$. One can also extend this form to matrix-valued diffusions of the form $\nabla^*_p D(q,p) \nabla_p$.

We can similarly consider a second-order forcing in $q$. One possible choice for the generator is then $\Lx = -\beta^{-1}\nabla^*_q\nabla_q$. Here, we require that $\alpha\eta>0$. The associated dynamics is
\begin{align}
\begin{split}
\begin{cases}
    dq_t = M^{-1} p_t \, dt - \alpha\eta\nabla V(q_t) \, dt + \sqrt{\dfrac{2\alpha\eta}{\beta}} \, dB_t, \\
    dp_t = \left(-\nabla V(q_t) \, dt + \eta F(q_t)\right) dt - \gamma M^{-1} p_t \, dt + \sqrt{\dfrac{2\gamma}{\beta}} \, dW_t,
\end{cases}
\end{split}
\end{align}
where $B_t$ is a standard $d$-dimensional Brownian motion independent of $W_t$. One can also generalize this choice by scaling it by a function of $p$, leading to the more general choice $-b(p)\nabla_q^*\nabla_q$, or introduce a matrix-valued diffusion and consider $\nabla^*_q D(q,p) \nabla_q$.

Additionally, one could consider a mixed-term forcing, such as $-\beta^{-1}\nabla^*_q\nabla_p$ or $-\beta^{-1}\nabla^*_p\nabla_q$. However, since the diffusion matrix must be symmetric and positive-definite, this prevents us from considering these mixed forcings without adding also a contribution~$-\beta^{-1}\nabla_q^*\nabla_q$. Upon writing $\L_{\eta,\alpha}$ as the sum of a first-order differential operator and $D_\alpha \colon \nabla^2$, we require that $D_\alpha$ be symmetric positive, so that one can simulate the associated SDE upon taking the square root of $D_\alpha$, with
\begin{equation}
	D_\alpha = D_0 + \alpha\widetilde{D}, \qquad D_0 = \gamma\begin{bmatrix}
		0 & 0 \\ 0 & \mathrm{Id}
	\end{bmatrix},
\end{equation}
where $\widetilde{D}$ denotes the diffusion associated with the extra forcings.
\end{example}

\begin{example}[{\bf Feynman-Kac forcing}]  
\label{ex:FK_lang}
    One possible form of the general forcing in item (3) of Section \ref{subsubsec:gen_extra_forcings} is the following, for two functions $\xi_1 \colon \R^d \to \R$ and $\xi_2 \colon \mathcal{X} \to \R$:
    \begin{equation}
  		\Lx = \xi_1(p)^T\nabla_q^* + \xi_2(q)^T\nabla_p^*.
	\end{equation}
	The associated dynamics reads
	\begin{align}
	\begin{split}
	\begin{cases}
	    dq_t = M^{-1} p_t \, dt - \alpha\eta\xi_1(p_t) \, dt, \\
	    dp_t = -\nabla V(q_t) \, dt + \eta(F(q_t) - \alpha\xi_2(q_t)) \, dt - \gamma M^{-1} p_t \, dt + \sqrt{\dfrac{2\gamma}{\beta}} \, dW_t,
	\end{cases}
	\end{split}
	\end{align}
	whose trajectories are reweighted using the Feynman--Kac weight involving $\xi_1(p_s)^T\nabla V(q_s) + \xi_2(q_s)^TM^{-1}p_s$ in the integral \eqref{FK_gen_weight}. 
\end{example}

\subsection{Choosing the magnitude of the forcing}
\label{subsec:choosing_alpha}
The synthetic forcing is comprised of a physical forcing and an extra one, which translates into a perturbation~$\Lp + \alpha\Lx$ at the level of generators. We discuss here how to choose the magnitude~$\alpha$ of the extra forcing in order to optimize it in terms of the linear response.

Recall from Section \ref{subsubsec:linRes_theory} that, for small values of $\eta$, we can think of the response as a polynomial in $\eta$:
\begin{equation}
	r_\alpha(\eta) = \E_{\eta,\alpha}(R) = \eta\rho_1(\alpha) + \eta^2\rho_2(\alpha) + \cdots + \eta^{n-1}\rho_{n-1}(\alpha) +\bigO(\eta^n).
	\label{nonlin_res}
\end{equation}
In practice, we want to choose $\alpha$ such that the contribution from nonlinear terms is minimized. There are two main approaches to choosing such an optimal $\alpha$, by (i) canceling the second-order response, or (ii) bounding the relative error with respect to the linear regime. We next discuss both options.

\paragraph{Canceling the second-order response} A first approach to reduce the nonlinear response is to cancel the second-order term $\rho_2$, as this is the dominant part of the nonlinear bias when $\eta$ is small. The second-order response is characterized by~$\f_2$ in~\eqref{eq6}, which reads here
\begin{align}
	\f_2(\alpha) &= (\L_0^{-1})^*(\Lp + \alpha\Lx)^* (\L_0^{-1})^*S = \f_{2,\text{phys}} + \alpha \f_{2,\text{extra}},
	\label{f2_alpha}
\end{align}
where we decomposed the second-order perturbation of the invariant measure into its physical and synthetic parts. Since $\f_2$ is linear in $\alpha$, the value of $\alpha$ for which $\rho_2$ is cancelled is easily obtained from the definition of the second-order response. Indeed,
\begin{equation}
    \rho_2(\alpha) = \int_\mathcal{X} R \f_{2,\text{phys}} \, \psi_0 + \alpha \int_\mathcal{X} R \f_{2,\text{extra}} \, \psi_0 = \rho_2 +\alpha\rho_{2,\mathrm{extra}}.
    \label{rho2_alpha}
\end{equation}
Therefore, $\rho_2(\alpha^\star) = 0$ for
\begin{equation}
    \alpha^\star = -\frac{\displaystyle\int_\mathcal{X} R\f_{2,\text{phys}} \, \psi_0}{\displaystyle\int_\mathcal{X} R\f_{2,\text{extra}} \, \psi_0},
    \label{aOpt1}
\end{equation}
provided that $\int_\mathcal{X} R\f_{2,\text{extra}} \, \psi_0 \ne 0$. The latter condition ensures that $\Lx$ has a nontrivial contribution to the second-order response, and can therefore be used to cancel the second-order response. An important remark is that although~$\alpha^\star$ cancels~$\rho_2$, it might significantly increase $\rho_3$ and higher order terms.

Note that for Feynman--Kac forcings, the computation of $\alpha^\star$ must be reformulated. This is due to the fact that \eqref{eq6} comes from the Fokker--Planck equation, which admits a nontrivial principle eigenvalue for Feynman--Kac forcings, as discussed in Section \ref{subsubsec:gen_extra_forcings} and made precise in Appendix \ref{app:FK_eig}. 

In general, the optimal value $\alpha^*$ cannot be determined a priori. It requires, in principle, two sets of simulations with $\alpha_1 \ne \alpha_2$, from which $\alpha^*$ can be extrapolated due to the linearity of $\rho_2(\alpha)$ in $\alpha$. We discuss in Section \ref{sec:perspectives} how to implement this approach for actual systems of interest.

\begin{remark}
For certain situations, it is possible to determine the impact of the extra perturbations from the response curve for the physical forcing (\emph{i.e.} $\alpha=0$). Consider for instance the setting of Example \ref{mfdr}. We define~$\E_{\eta,\alpha}$ as the steady-state average for $\psi_{\eta,\alpha}$, the solution to the Fokker--Planck equation $\L_{\eta,\alpha}^\dagger\psi_{\eta,\alpha} = 0$. Since~$\L_{\eta,\alpha}$ is proportional to $\L_{\eta/(1+\alpha\eta),0}$ (recall that we require $1+\alpha\eta \ne 0$), a simple computation shows that $\psi_{\eta,\alpha} = \psi_{\eta/(1+\alpha\eta),0}$, so 
\begin{align}
	\forall\eta\in\R, \qquad r_\alpha(\eta) = r_0\left(\frac{\eta}{1+\alpha\eta}\right).
\end{align}
Then,
\begin{align}
	r_\alpha(\eta) &= r_0'(0)\frac{\eta}{1+\alpha\eta} + \frac{1}{2}r_0''(0)\frac{\eta^2}{(1+\alpha\eta)^2} + \bigO(\eta^3) \\
    &= r_0'(0)\eta +\left[\frac{r_0''(0)}{2} - \alpha r_0'(0)\right]\eta^2 + \bigO(\eta^3).
\end{align}
This allows us to find the optimal $\alpha$ which cancels $\rho_2$, given by
\begin{equation}
    \alpha = \frac{r_0''(0)}{2r_0'(0)} = \frac{\rho_2}{\rho_1}.
\end{equation}
This result suggests that, for the example presented here, the magnitude of the extra forcing should be chosen so that $\rho_{2,\mathrm{extra}} = -\rho_1$.
\end{remark}

\paragraph{Bounding the relative error} From a practical viewpoint, and following the usual bias/variance tradeoff, it might be more advantageous to stay close enough to the linear response for larger $\eta$, even if that means decreasing the true linear regime. In order to make these statements quantitative, we consider the relative error $\delta$ in the response relative to the linear response, that is
\begin{equation}
    \delta_\alpha(\eta) = \left|\frac{r_\alpha(\eta) - \rho_1\eta}{\rho_1\eta}\right|.
    \label{rel_err}
\end{equation}
For an extra perturbation $\Lx$ and $\alpha \in \R$ fixed, the response stays in the neighborhood of the linear response until a certain value of $\eta$, at which point it deviates too far from the linear regime.  Since $\delta_\alpha(0) = 0$, we look, for some small fixed value $\eps>0$, for the smallest value of $|\eta|$ such that~$\delta_\alpha(\eta)=\eps$, which we denote by $\eta_{\alpha}(\eps)$. The question of optimizing $\alpha$ can be reformulated as finding $\alpha$ such that $\eta_{\alpha}(\eps)$ is maximized:
\begin{equation}
    \alpha_\star(\eps) = \argmax_{\alpha\in\R} \eta_{\alpha}(\eps), \qquad \eta_{\alpha}(\eps) = \argmin_{\eta\in\R} \left\{|\eta| \colon \delta_\alpha(\eta)\geq \eps\right\}.
    \label{aOpt2}
\end{equation}
As we tighten the bound $\eps$, the value of $\alpha_\star(\eps)$ gets closer to the value $\alpha^\star$ for which $\rho_2(\alpha^\star)=0$. This comes from the fact that $\eta_{\alpha}(\eps)$ is of order $\bigO(\eps)$ for $\alpha\ne\alpha^\star$, and of order $\bigO(\sqrt{\eps})$ for $\alpha=\alpha^\star$. Indeed, \eqref{rel_err} can be written as
\begin{equation}
	 \delta_\alpha(\eta) = \left|\rho_{2,\mathrm{extra}}(\alpha - \alpha^\star)\eta + \bigO(\eta^2)\right|.
\end{equation}
Intuitively, canceling the leading order term (namely~$\rho_2$) in the small $\eta$ regime is equivalent to minimizing the deviation from the linear regime, the latter being implied by the limit $\eps\to 0$. Figure \ref{fig:optAlpha} illustrates how~$\alpha^\star$ and~$\alpha_\star(\eps)$ can yield drastically different response curves. The functions $r_\alpha(\eta)$ are fourth-order polynomials in $\eta$ whose coefficients have been hand-picked in order to demonstrate the possibly very different behaviors of $r_{\alpha^\star}(\eta)$ and $r_{\alpha_\star(\eps)}(\eta)$. Figure \ref{subfig:mock1_res} shows the full response curves $r_\alpha(\eta)$, and Figure \ref{subfig:mock2_res} shows the corresponding relative error curves relative to the linear response \eqref{rel_err}. The curve obtained by choosing~$\alpha^\star$ minimizes the deviation from the linear regime for small $\eta$ as the second-order response~$\rho_2$ is canceled. On the other hand, the curve for $\alpha_\star(\eps)$ with $\eps=0.05$ slightly departs from the linear regime earlier than the one associated with $\alpha^\star$. It stays, however, in the vicinity of the linear response for much larger $\eta$. This behavior is more clearly observed in Figure \ref{subfig:mock2_res}.

\begin{figure}
\centering
\begin{subfigure}{0.49\textwidth}
    \includegraphics[width=\textwidth]{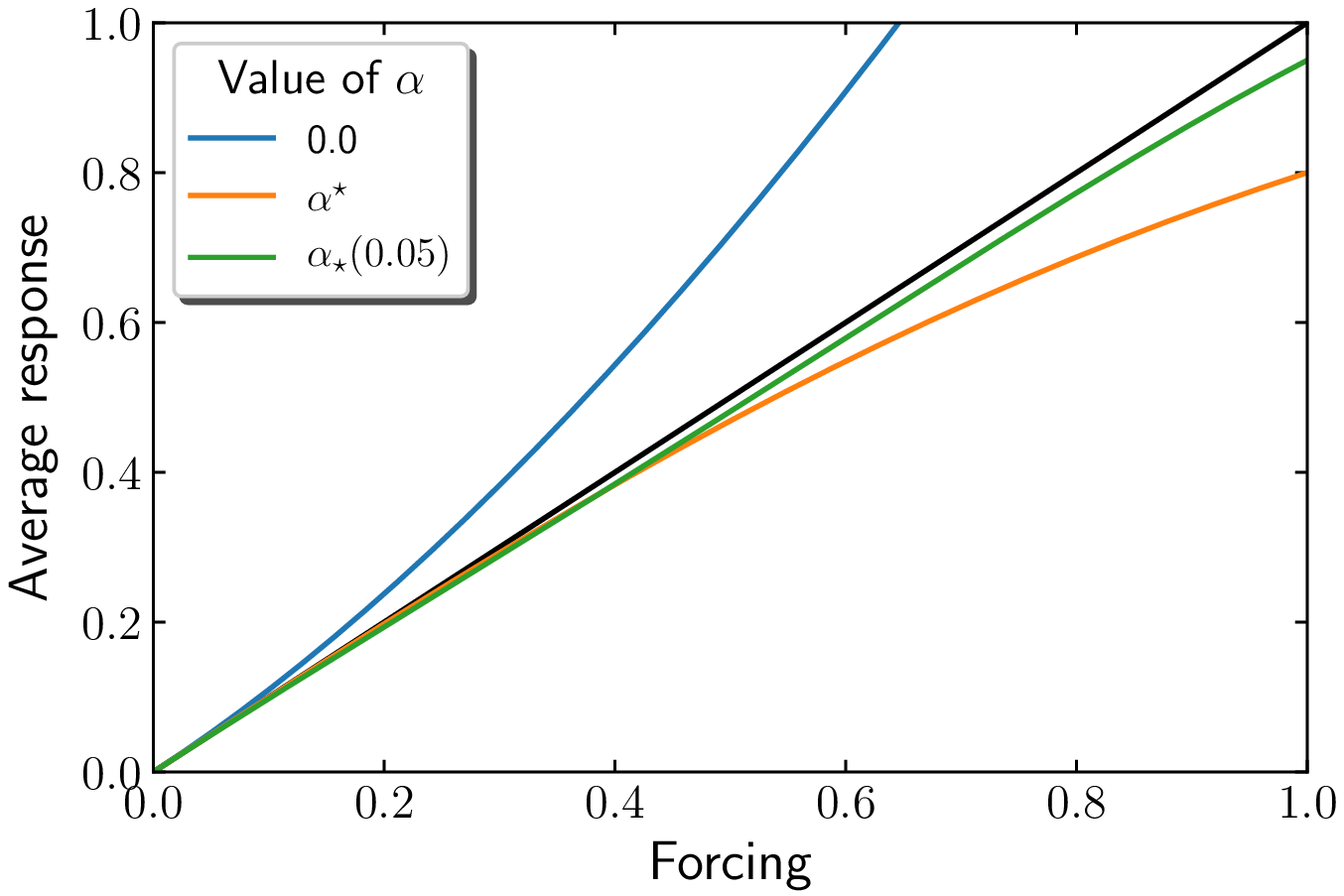}
    \caption{Full response curves $r_\alpha(\eta)$ for various $\alpha$. \\ $\,$}
    \label{subfig:mock1_res}
\end{subfigure}
\hfill
\begin{subfigure}{0.49\textwidth}
    \includegraphics[width=\textwidth]{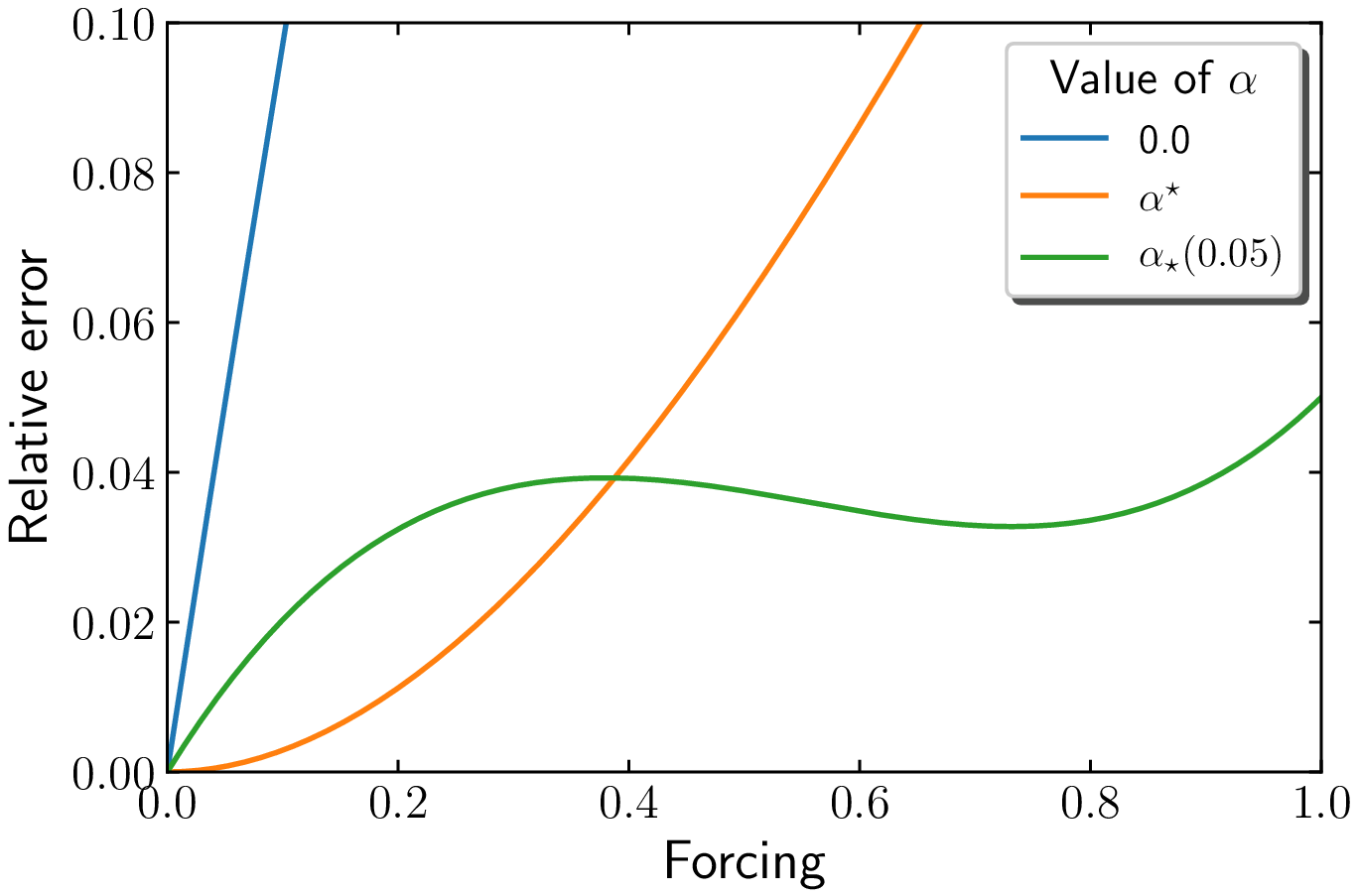}
    \caption{Corresponding relative error curves $\delta_\alpha(\eta)$ relative to the linear response.}
    \label{subfig:mock2_res}
\end{subfigure}
\caption{Illustration of the effect of different $\alpha$'s on the response}
\label{fig:optAlpha}
\end{figure}

In most practical applications, using $\alpha_\star(\eps)$ is the choice of interest, as the points on the response curve are anyway computed with some numerical error, such as timestep discretization and statistical error. There is, however, a tradeoff to be considered. As one tightens the acceptable relative error by decreasing $\eps$, the value of $\eta_{\alpha}(\eps)$ is also decreased. In practice, values of $\eps$ in the range 0.01 to 0.1 are small enough to lead to a bias of a few percent in relative magnitude, but large enough so that the benefit of the increased magnitude of the forcing is significant. We further illustrate the tradeoff in Section \ref{sec:numerical} with numerical results for several observables for overdamped and underdamped Langevin dynamics.

As a final remark in this section, note that computing $\alpha_\star(\eps)$ requires computing the full response curve. Although this is not practical, the aim here is to provide a proof of principle demonstrating the potential computational gains obtained by making use of synthetic forcings and decide, between various strategies, the most promising one in order to adapt the approach to actual systems of interest.

\subsection{Linearly combining extra forcings}
\label{subsec:lin_comb}

As discussed in Section \ref{subsec:choosing_alpha}, with the addition of some extra forcing $\Lx$, it is possible to find a value of $\alpha$ which cancels the second-order response $\rho_2$ due to the linearity of $\f_2$ in $\alpha$, as shown in~\eqref{f2_alpha}. We can extend this notion to linear combinations of extra forcings. In particular, one can combine $k$ forcings in order to cancel the first $k$ nonlinear orders of the response. That is, for $\alpha = (\alpha_1,\dotsc,\alpha_k)$, the synthetic perturbation $\Lp + \alpha_1\widetilde{\L}_\mathrm{extra,1} + \cdots + \alpha_k\widetilde{\L}_\mathrm{extra,k}$ can be used, with $\alpha$ chosen such that $\rho_{2,\alpha}= \cdots = \rho_{k+1,\alpha} = 0$. This is a nonlinear equation in $\alpha$, with $k$ unknowns scalar values and $k$ conditions to solve.

\section{Numerical results}
\label{sec:numerical}
The aim of the numerical illustrations presented in this section is to demonstrate the potential of the synthetic forcing approach, on the examples given in Section \ref{subsec:examples}. The numerical results are obtained by discretizing the PDEs determining the invariant probability measure of each system (namely the Fokker--Planck equation \eqref{FP_eta_alpha}) and the Poisson equation \eqref{poissonf1}. The use of this method, particularly in low-dimensional systems, allows us to extensively and thoroughly examine the quality of synthetic forcings, as we can easily compute full response curves and individual orders of the response, in contrast to Monte Carlo simulations, for which some statistical error and timestep discretization bias are present. This numerical method, however, cannot be used as such for higher dimensional systems, as solving the associated PDEs becomes too cumbersome a task. Thus, for higher dimensional systems, Monte Carlo simulations are usually preferred (see discussion in Section \ref{sec:perspectives}).

We first discuss in Section \ref{subsec:numerical_method} the numerical methods used, then present the numerical results for the one and two-dimensional overdamped Langevin, and one-dimensional Langevin dynamics in Sections~\ref{subsec:ovd_numerical_1d}, \ref{subsec:ovd_numerical_2d} and~\ref{subsec:lang_numerical}, respectively. We finally discuss in Section \ref{subsec:num_var} the impact of the synthetic forcings approach on variance reduction in the estimation of transport coefficients.

\subsection{Numerical method}
\label{subsec:numerical_method}
The full response curve and the linear response are computed by solving the associated Fokker--Planck equation for each dynamics, and the Poisson equation \eqref{poissonf1}, respectively. The main advantage of numerical methods based on solving partial differential equations is that the discretization error can be systematically reduced to a very small value by refining the mesh used to represent the functions at hand. There are two situations in which the solutions to PDEs are required to determine the response to external perturbations, which we outline below.

\paragraph{Approximation of the linear response $\rho_1$} The linear response is obtained by computing $\f_1$ and approximating the integral on the right-hand side of \eqref{linResf1}. To approximate $\f_1$, we solve the Poisson equation \eqref{poissonf1}, which we recall here for convenience:
\begin{equation}
    \Lr^*\f_1 = -\Lp^*\ind.
    \label{poissonf1_recall}
\end{equation}
For overdamped Langevin dynamics, we directly solve \eqref{poissonf1_recall} as we consider a bounded position space, namely the torus $\T^d$. For Langevin dynamics, instead of solving for $\f_1$ in the expansion \eqref{feta}, we solve for $\overline{\psi}_1$ in $\psi_\eta = \psi_0 + \eta\overline{\psi}_1 + \eta^2\overline{\psi}_2 + \cdots$, \emph{i.e.} $\overline{\psi}_1 = \f_1\psi_0$. This function satisfies the Poisson equation
\begin{equation}
	\Lr^\dagger\overline{\psi}_1 = -\Lp^\dagger \psi_0,
	\label{poissonf1_dagger}
\end{equation}
which is equivalent to reformulating \eqref{poissonf1_recall} in terms of the $L^2$-adjoint. Due to the unbounded momentum space for Langevin dynamics, solving \eqref{poissonf1_dagger} makes the problem easier to solve numerically than \eqref{poissonf1_recall}. Indeed, the unbounded momentum space needs to be truncated. A natural choice when solving for $\overline{\psi}_1$ is to set Dirichlet boundary conditions at the boundaries of the domain in $p$, which is consistent with the fact that~$\overline{\psi}_1(q,p)$ is expected to vanish as $|p|\to+\infty$. In contrast, there is no natural boundary condition for $\f_1$ in \eqref{linResf1} when the momentum space is truncated.

Once an approximation of $\f_1$ or $\overline{\psi}_1$ is obtained, we perform a quadrature on one of the following integrals to directly find the linear response
\begin{equation}
    \rho_1 = \int_\mathcal{X} R\f_1 \, \psi_0 = \int_\mathcal{X} R \, \overline{\psi}_1.
    \label{eq135}
\end{equation}
More generally, this procedure can be extended by using the recursive formula \eqref{eq6}, which allows for the computation of response terms of arbitrary orders; this is used in particular to compute $\rho_{2,\text{phys}}$ and $\rho_{2,\text{extra}}$ when computing the value of $\alpha^\star$ defined in \eqref{aOpt1}.

\paragraph{Approximation of the full response $r_\alpha(\eta)$} In this work, we compute the full response in order to quantify how much, and how quickly the response deviates from the linear regime. In actual applications, one typically does not compute the full response curve, in particular when using Monte Carlo simulations. The quantity of interest, namely the transport coefficient, comes from the linear response, which can be obtained from computing a single point (or typically two to ensure linearity), so computing the full response curve is not of interest.
 
The full response curve can be computed by solving the Fokker--Planck equation~\eqref{FP_eta_alpha}, which we recall for convenience
\begin{equation}
	\L_{\eta,\alpha}^\dagger\psi_{\eta,\alpha} = 0,
	\label{FP_eta_alpha_recall}
\end{equation}
and then performing a quadrature on the integral
\begin{equation}
    r_\alpha(\eta) = \int_\mathcal{X} R \, \psi_{\eta,\alpha}.
\end{equation}

\paragraph{Solving the PDE} For the discretization of the PDEs \eqref{poissonf1_recall}, \eqref{poissonf1_dagger} and \eqref{FP_eta_alpha_recall}, we use a finite-difference scheme. Periodic boundary conditions are used in the spatial variable since we always consider $q\in \T^d$ for both overdamped Lanvegin and Langevin dynamics, with $d = 1$ or $d=2$. For overdamped Langevin dynamics, we use a centered finite-difference scheme. For the one-dimensional Langevin dynamics, the momentum variable is first truncated to $[-p_\mathrm{max}, p_\mathrm{max}]$, then its domain is discretized with step size $h_p$. Dirichlet boundary conditions are imposed at $p = \pm p_\mathrm{max}$. We ensured that our truncated value of the momentum $p_\mathrm{max} = 6$ is large enough so that it does not affect the numerical results. A centered scheme is also used for Langevin dynamics, except for the transport term $-p^TM^{-1}\nabla_q$ in the Fokker--Planck equation \eqref{lang_FP_L2}, where an upwind scheme is used (refer to Appendix \ref{appendix:num_schemes} for the precise expressions of the aforementioned numerical schemes). All computations were performed with the \code{Julia} language.

\subsection{Overdamped Langevin dynamics - one-dimensional case}
\label{subsec:ovd_numerical_1d}
We present in this section the numerical results for the one-dimensional overdamped Langevin dynamics \eqref{neld_ovd} on $\mathcal{X} = \T$, with potential energy
\begin{equation}
	V(q) = \cos(2\pi q).
	\label{1D_cos_pot}
\end{equation}
We consider the observable
\begin{equation}
	R(q) = \left(a\cos(2\pi q) + b \sin(2\pi q)\right)\e^{\beta V(q)},
	\label{obs1D}
\end{equation}
where $a,b \in \R$. We choose this observable for two reasons. First, by construction, it has average zero with respect to the Gibbs probability measure, with density proportional to $\e^{-\beta V}$. Second, we can tune the coefficients $a$ and $b$ to control the magnitude of the individual orders of response. This allows us to choose $a$ and $b$ such that the first and second-order responses are normalized, \emph{i.e.} $\rho_1 = \rho_2 = 1$, which makes it easier to compare the quality of each synthetic forcing. 

Due to the symmetries of the potential energy function \eqref{1D_cos_pot}, we can directly compute the values of $a$ and $b$ such that $\rho_1 = \rho_2 = 1$. More precisely, it can be shown that $\f_1$ and $\f_2$ are respectively odd and even on $[-1/2, 1/2]$ so that $b$ controls the magnitude of $\rho_1$ and $a$ controls the magnitude of $\rho_2$. The value of $b$ such that $\rho_1 = 1$ is easily computed using \eqref{rhon_response} to be
\begin{equation}
	b = \left(Z^{-1}\int_\mathcal{X} \sin(2\pi q) \f_1(q) \, dq \right)^{-1}, \qquad Z = \int_\mathcal{X} \e^{-V(q)}dq.
	\label{b_obs_norm}
\end{equation}
Similarly, the value of $a$ such that $\rho_2 = 1$ is
\begin{equation}
	a = \left(Z^{-1}\int_\mathcal{X} \cos(2\pi q) \f_2(q) \, dq\right)^{-1}.
	\label{a_obs_norm}
\end{equation}
The spatial domain $\T = [0,1)$ is discretized into $m = 2000$ points, with uniform step size $h = 1/m$. The simulations were performed with inverse temperature $\beta = 1$ and mass $M = 1$, which is also the setting for the results in two dimensions obtained in Section \ref{subsec:ovd_numerical_2d}. Moreover, we consider the forcing $F=1$, as this is the only nongradient forcing in dimension one on the torus.

We present the full response curves and the associated linear response for each of the three synthetic forcings discussed in Section \ref{subsubsec:ovd_examples}. Figures \ref{ovd_1a}, \ref{ovd_1b} and \ref{ovd_1c} correspond to the Feynman--Kac forcing \eqref{fk_ovd}, modified fluctuation-dissipation \eqref{mfdr_ovd} and divergence-free vector field \eqref{divFree}, respectively. Note that in dimension one, \eqref{divFree} reduces to a single option, namely
\begin{equation}
	\Lx = \e^V\frac{d}{dq},
	\label{div_free_1d}
\end{equation}
which is the one used here. For each of the plots, we show the linear response, the full response curve $r_0(\eta)$ for $\alpha=0$, the response curve for $r_{\alpha^\star}(\eta)$ for $\alpha^\star$ computed using \eqref{aOpt1}, and response curves for some additional values of $\alpha\in\R$ for illustrative purposes. In some cases, in particular when the response curve associated with $\alpha^\star$ sees marginal improvement in extending the linear regime, we also compute the curve for $\alpha_\star(\eps)$, where we choose $\eps=0.05$. Lastly, Figure~\ref{ovd_1d} includes all synthetic forcings. Note that the maximal value of the forcing is much larger for Figure~\ref{ovd_1d}, where we compare the best choices for the magnitude of each synthetic forcing; this is also the setting for the illustrations presented in Sections \ref{subsec:ovd_numerical_2d} and \ref{subsec:lang_numerical}.

\begin{figure}[h]
\centering
\begin{subfigure}{0.49\textwidth}
    \includegraphics[width=\textwidth]{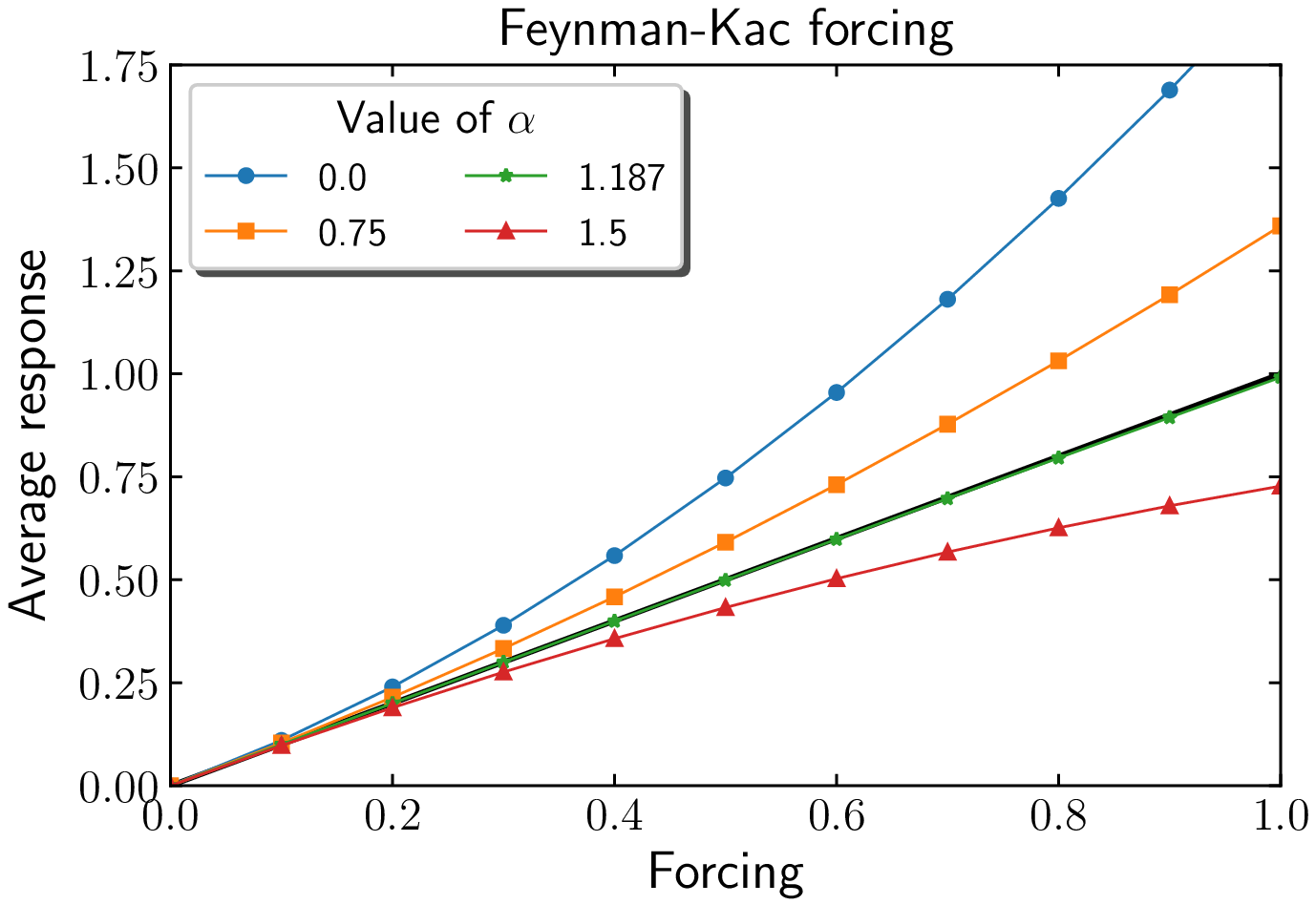}
    \caption{Feynman--Kac forcing, $\alpha^\star = 1.187$.}
    \label{ovd_1a}
\end{subfigure}
\hfill
\begin{subfigure}{0.49\textwidth}
    \includegraphics[width=\textwidth]{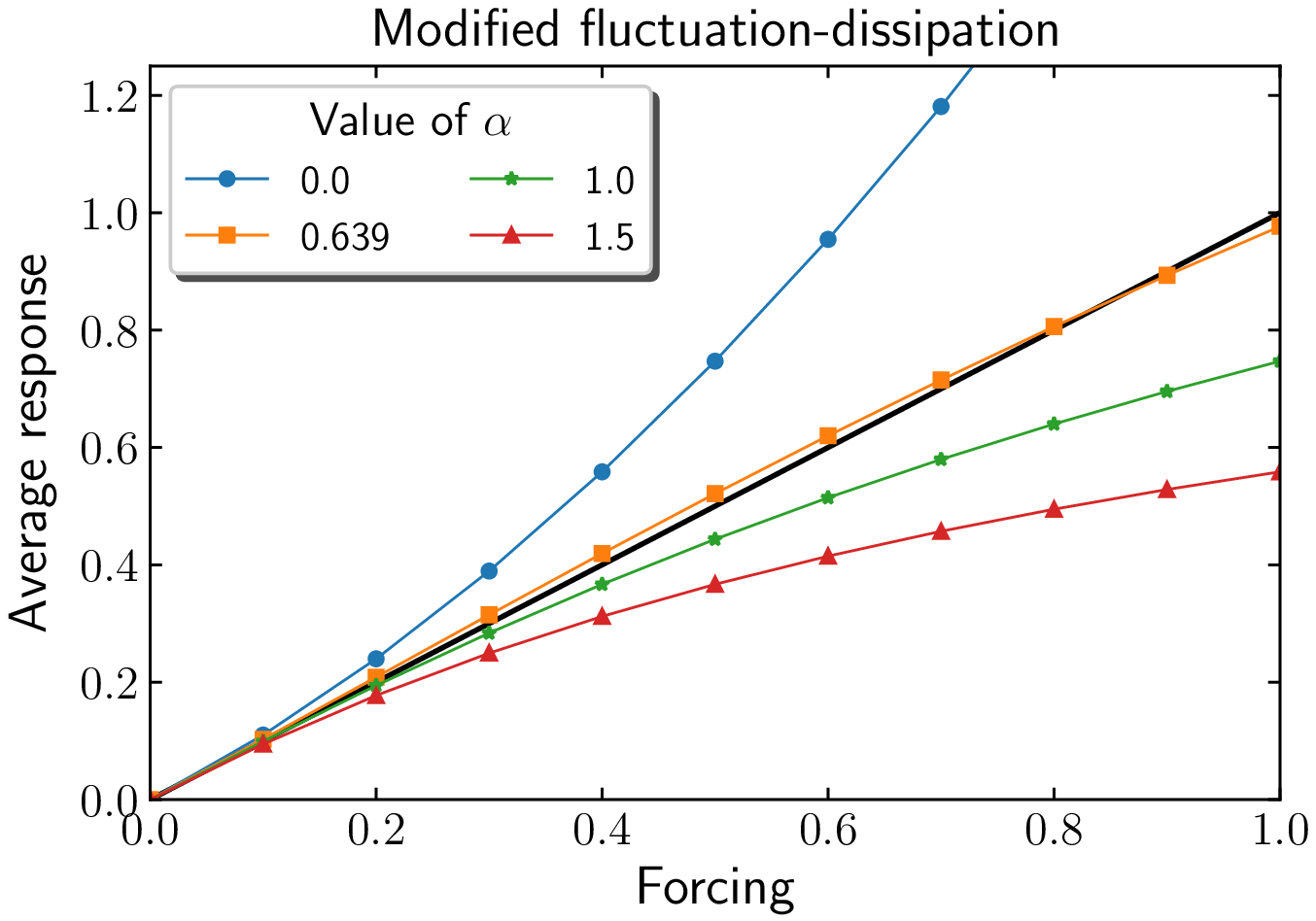}
    \caption{Modified FD, $\alpha^\star = 1.0$.}
    \label{ovd_1b}
\end{subfigure}
\hfill
\begin{subfigure}{0.49\textwidth}
    \includegraphics[width=\textwidth]{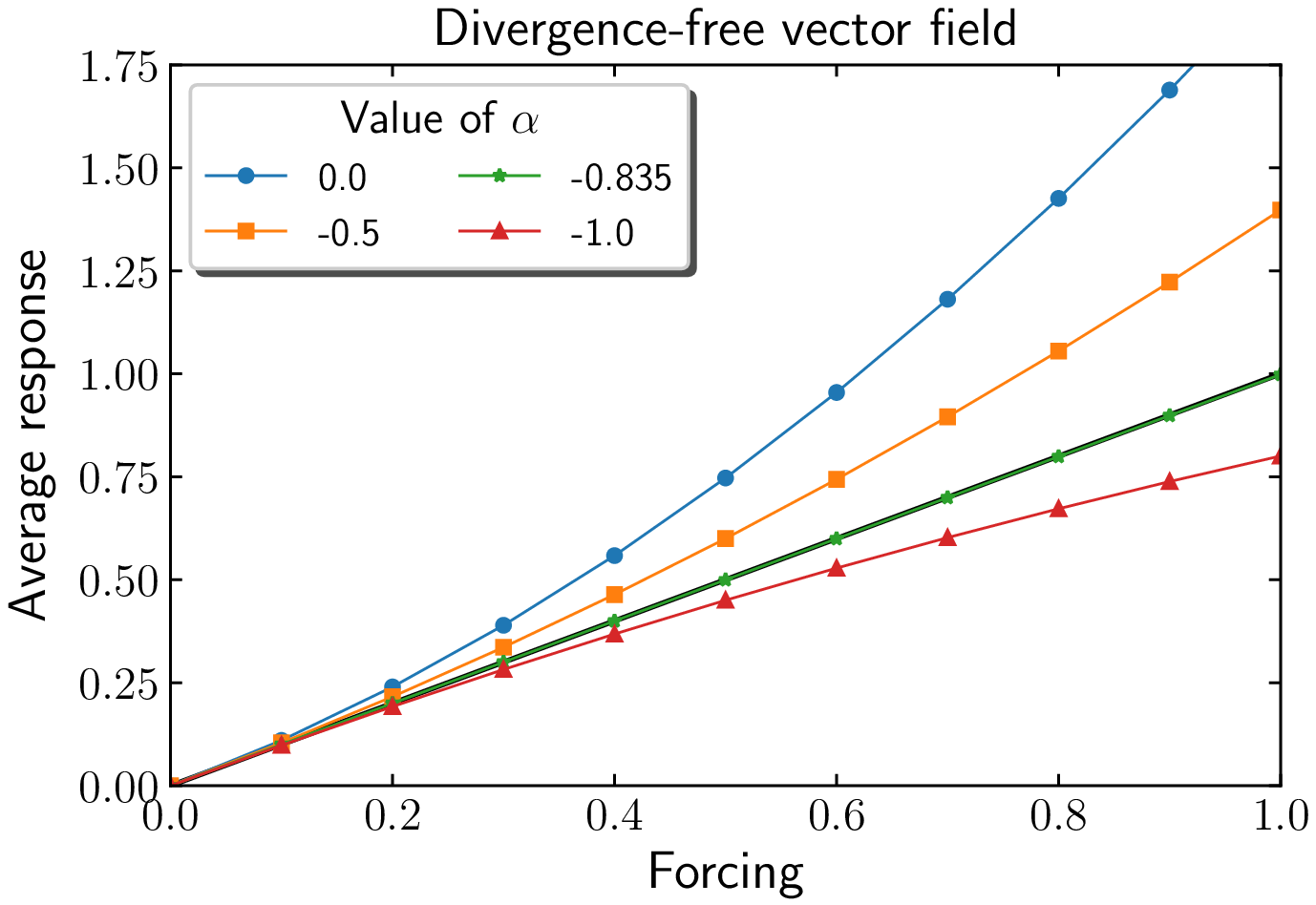}
    \caption{Div-free vector field, $\alpha^\star = -0.835$.}
    \label{ovd_1c}
\end{subfigure}
\hfill
\begin{subfigure}{0.49\textwidth}
    \includegraphics[width=\textwidth]{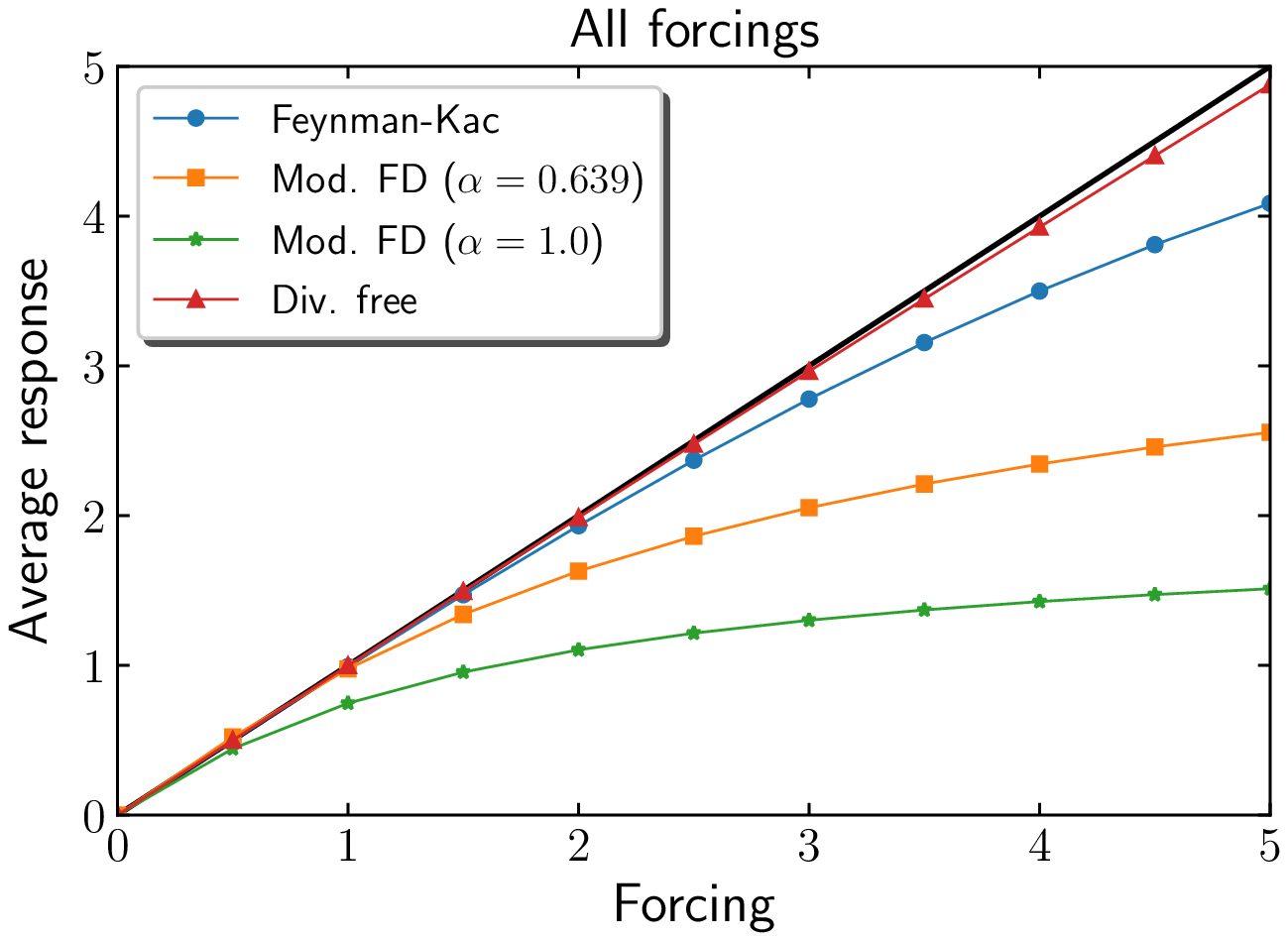}
    \caption{Comparing $\alpha^\star$ for all forcings.}
    \label{ovd_1d}
\end{subfigure}
	\caption{Response curves for various synthetic forcings for the one-dimensional overdamped Langevin dynamics.}
    \label{fig_ovd_1}
\end{figure}

The choice $\alpha^\star$, namely the one that cancels the second order response $\rho_2$, performs quite well for the Feynman--Kac and divergence-free forcings, as seen in Figures \ref{ovd_1a} and \ref{ovd_1c}. It allows to extend the range of linearity for $\eta > 1$, an increase of over tenfold when compared to the original response curve for $\alpha=0$. This allows for a variance reduction of a factor of order 1000 for the estimator \eqref{estimator}, as documented in Section~\ref{subsec:num_var}.

 Figure \ref{ovd_1b} illustrates the discussion of Section \ref{subsec:choosing_alpha} about choosing the values of $\alpha$ allowing to stay longer in an approximate linear response. Although the value of the parameter $\alpha$ such that $\rho_2(\alpha) = 0$ is $\alpha^\star = 1.0$ for the modified fluctuation-dissipation, we see that $\alpha_\star(0.05) = 0.639$ is a much more nicely behaved curve, staying within 5\% relative error for a large $\eta$ regime. This is a consequence of the fact that although $\alpha^\star$ cancels $\rho_2$, it might significantly increase $\rho_3$ and higher order terms.
 
\subsection{Overdamped Langevin dynamics - two-dimensional case}
\label{subsec:ovd_numerical_2d}
We now present the numerical results for overdamped Langevin dynamics in dimension two, defined on $\mathcal{X} = \T^2$, with the following periodic potential energy:
\begin{equation}
    V(q) = \frac{1}{2}\cos(2\pi q_1) + \cos(2\pi q_2) + \kappa\cos\left(2\pi(q_1 - q_2)\right),
\end{equation}
with $\kappa\in\R$. The observable considered here is of the same form as \eqref{obs1D}, and reads
\begin{equation}
	R(q) = \left(a\cos(2\pi q_1) + b \sin(2\pi q_1)\right)\e^{\beta V(q)}.
	\label{obs2D}
\end{equation}
For observables of this form, expressions \eqref{b_obs_norm} and \eqref{a_obs_norm} can be generalized to higher dimensions to compute the normalization constants. As in Section \ref{subsec:ovd_numerical_1d}, this ensures that $\rho_1=\rho_2=1$ when $\kappa=0$. When $\kappa$ is nonzero but not too large, $\rho_1$ and $\rho_2$ are of order 1, which is convenient to observe deviations from the linear regime.

For all numerical results presented here, we considered $\kappa = 0.3$ and a constant nongradient forcing $F = (1,0) \in \R^2$. The spatial domain $\T^2 = [0,1)^2$ was discretized using a regular product mesh of $m_q = 200$ points per dimension.

\begin{figure}[h!]
	\centering
	\begin{subfigure}{0.49\textwidth}
	    \includegraphics[width=\textwidth]{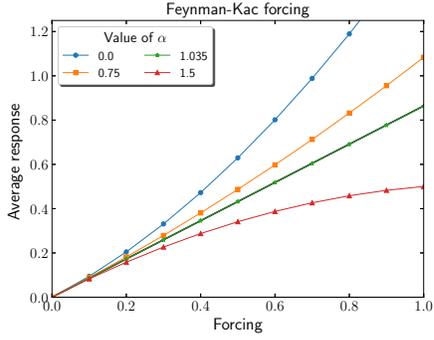}
	    \caption{Feynman--Kac forcing, $\alpha_\star = 1.035$.}
	    \label{ovd_2a}
	\end{subfigure}
	\hfill
	\begin{subfigure}{0.49\textwidth}
	    \includegraphics[width=\textwidth]{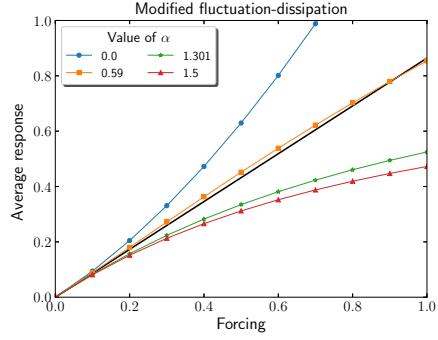}
	    \caption{Modified FD, $\alpha_\star = 1.301$.}
	    \label{ovd_2b}
	\end{subfigure}
	\hfill
	\begin{subfigure}{0.49\textwidth}
	    \includegraphics[width=\textwidth]{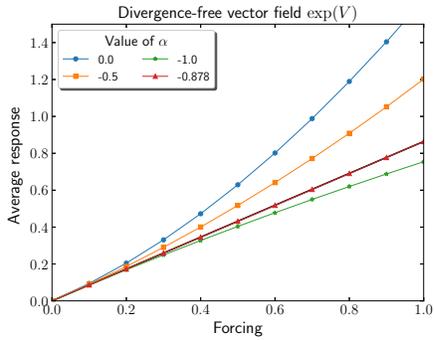}
	    \caption{Div-free vector field $\e^V$, $\alpha_\star = -0.878$.}
	    \label{ovd_2c}
	\end{subfigure}
	\hfill
	\begin{subfigure}{0.49\textwidth}
	    \includegraphics[width=\textwidth]{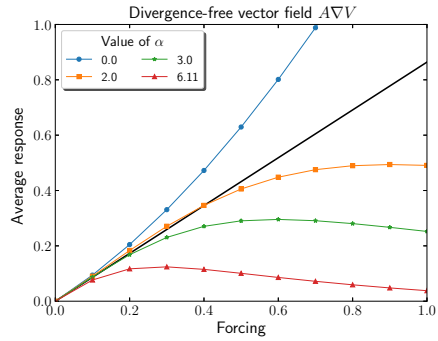}
	    \caption{Div-free vector field $A\nabla V$, $\alpha_\star = 6.11$.}
	    \label{ovd_2d}
	\end{subfigure}
	\hfill
	\begin{subfigure}{0.49\textwidth}
	    \includegraphics[width=\textwidth]{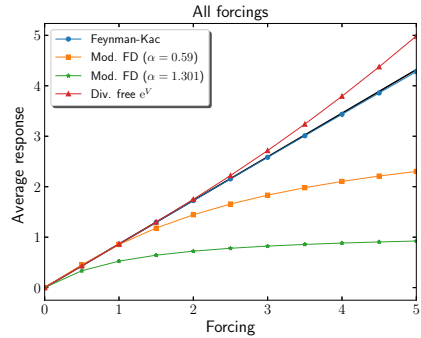}
	    \caption{All forcings.}
	    \label{ovd_2e}
	\end{subfigure}
		\caption{Response curves for various synthetic forcings for the two-dimensional overdamped Langevin dynamics.}
	    \label{fig_ovd_2}
\end{figure}

The results presented in Figure \ref{fig_ovd_2} correspond to response curves for each of the synthetic forcings discussed in Section \ref{subsubsec:ovd_examples}. Figure \ref{ovd_2a} corresponds to the Feynman--Kac forcing~\eqref{fk_ovd}, with~$\xi = (1,0) \in \R^2$. It is observed that this extra forcing is more effective in extending the linear regime when it is in the same direction as the physical perturbation~$F$, hence our choice of~$\xi$. We generalize the exponential divergence-free vector field \eqref{div_free_1d} presented in Section \ref{subsec:ovd_numerical_1d} to higher dimensions as $\Lx = \e^V(\nabla\cdot)$, presented in Figure \ref{ovd_2c}.

Overall, the results are qualitatively similar to those from Section \ref{subsec:ovd_numerical_1d}, apart from two main differences. First, we present a divergence-free vector field of the form~$\nabla V^T A\nabla$, with $A$ the symplectic matrix given by \eqref{symp_mat}, which has an underwhelming impact on increasing the linear regime as seen in Figure \ref{ovd_2d}, even for $\alpha_\star(0.05) = 2.0$. It seems therefore not to be a good option to consider. Second, Figure \ref{ovd_2e} shows that the Feynman--Kac forcing performs better than the exponential divergence-free field. On the other hand, the modified fluctuation-dissipation~\eqref{mfdr_ovd} is once again underwhelming, as seen in Figure~\ref{ovd_2b}, with response curve for $\alpha_\star(0.05) = 0.59$ once again performing better than $\alpha^\star = 1.301$. Note that the value $\alpha_\star(\eps)$ was not computed for the Feynman--Kac and the exponential divergence-free forcings, as~$\alpha^\star$ sufficiently extends the regime of linear response for those cases.

\subsection{Langevin dynamics - one-dimensional case}
\label{subsec:lang_numerical}
We present here the numerical results associated with the one-dimensional Langevin dynamics \eqref{lang_noneq}, where the potential energy function is the same as the one used for the one-dimensional overdamped Langevin dynamics case, namely \eqref{1D_cos_pot}, and the same observable \eqref{obs1D} as in Section \ref{subsec:ovd_numerical_1d}. The normalization constants $a$ and $b$ for $\rho_1$ and $\rho_2$ are chosen to be
\begin{equation}
	b = \left(\frac{1}{Z\sqrt{2\pi/\beta}}\int_{\mathcal{X}\times \R^d} \sin(2\pi q) \f_1(q,p) \exp(-\beta p^2/2) dp \, dq\right)^{-1},
\end{equation}
\begin{equation}
	a = \left(\frac{1}{Z\sqrt{2\pi/\beta}}\int_{\mathcal{X}\times \R^d} \cos(2\pi q) \f_2(q,p) \exp(-\beta p^2/2) dp \, dq\right)^{-1},
\end{equation}
with $Z$ the same normalization constant as in Section \ref{subsec:ovd_numerical_1d}, defined in \eqref{b_obs_norm}. The spatial domain $\T = [0,1)$ was discretized using $m_q = 200$ points, with uniform step size~$h_q = 1/m_q$. The unbounded momentum space was truncated to $[-p_\mathrm{max},p_\mathrm{max}]$ with $p_\mathrm{max}=6.0$, with Dirichlet boundary conditions at $p=\pm p_\mathrm{max}$, and then discretized into $m_p = 1000$ points with uniform step size~$h_p = 2p_\mathrm{max}/(m_p-1)$. The simulations were performed with $\beta=\gamma=M=1$. Note that when computing the response associated with the divergence-free vector field for large values of $\eta$, as presented in Figure \ref{fig:lang_1e}, the momentum space is truncated to~$p_\mathrm{max} = 10$.

The results presented in Figure \ref{fig:fig_lang} correspond to the response curves for each of the synthetic forcings discussed in Section \ref{subsubsec:lang_examples}. Figures \ref{fig:lang_1a} and \ref{fig:lang_1b} correspond to the position and momentum components of the modified fluctuation-dissipation forcing presented in Example \ref{ex:mod_fd_lang}, respectively, namely
\begin{equation}
	\Lx = -\partial_q^*\partial_q, \qquad \Lx = -\partial_p^*\partial_p.
\end{equation}
Figures \ref{fig:lang_1c} and \ref{fig:lang_1d} correspond to the position and momentum components of the Feynman--Kac forcing presented in Example \ref{ex:FK_lang}, respectively, namely
\begin{equation}
	\Lx = \partial_q^* = V' - \partial_q, \qquad \Lx = \partial_p^* = p - \partial_p.
\end{equation}
Lastly, Figure \ref{fig:lang_1e} corresponds to the divergence-free forcing \eqref{div_free_1d}.

\begin{figure}[h!]
\centering
\begin{subfigure}{0.49\textwidth}
    \includegraphics[width=\textwidth]{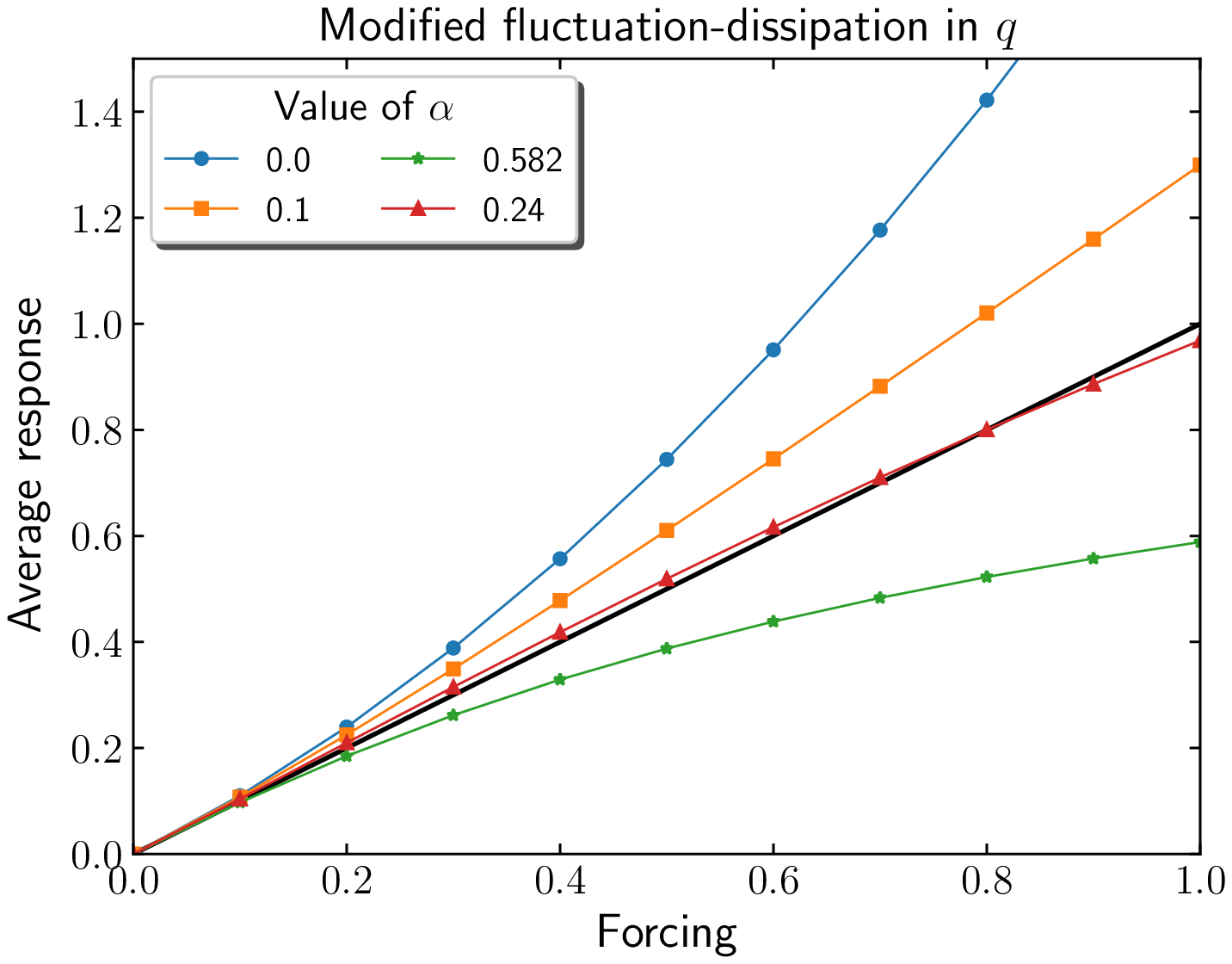}
    \caption{Modified FD in $q$, $\alpha_\star = 0.582$.}
    \label{fig:lang_1a}
\end{subfigure}
\hfill
\begin{subfigure}{0.49\textwidth}
    \includegraphics[width=\textwidth]{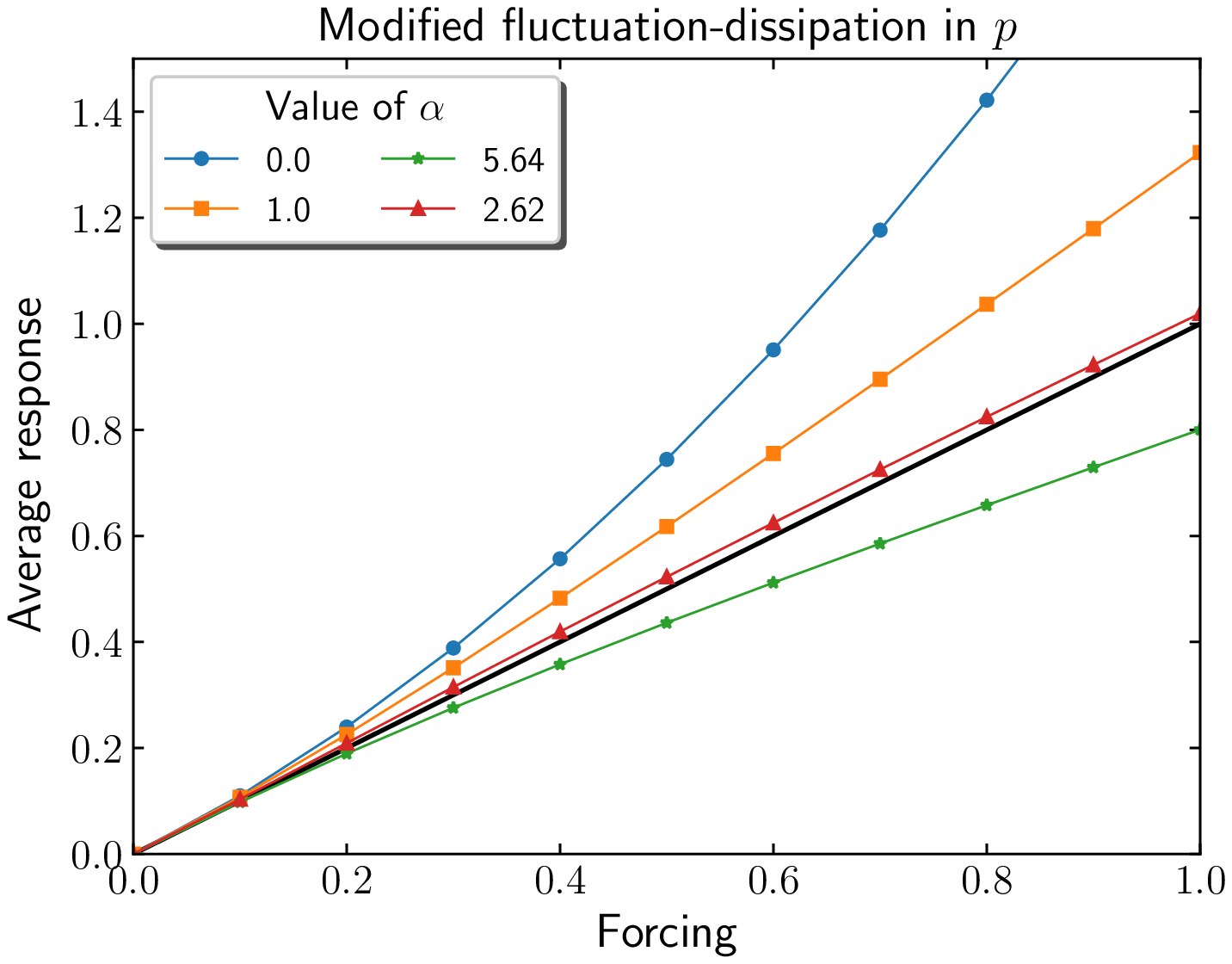}
    \caption{Modified FD in $p$, $\alpha_\star = 5.640$.}
    \label{fig:lang_1b}
\end{subfigure}
\hfill
\begin{subfigure}{0.49\textwidth}
    \includegraphics[width=\textwidth]{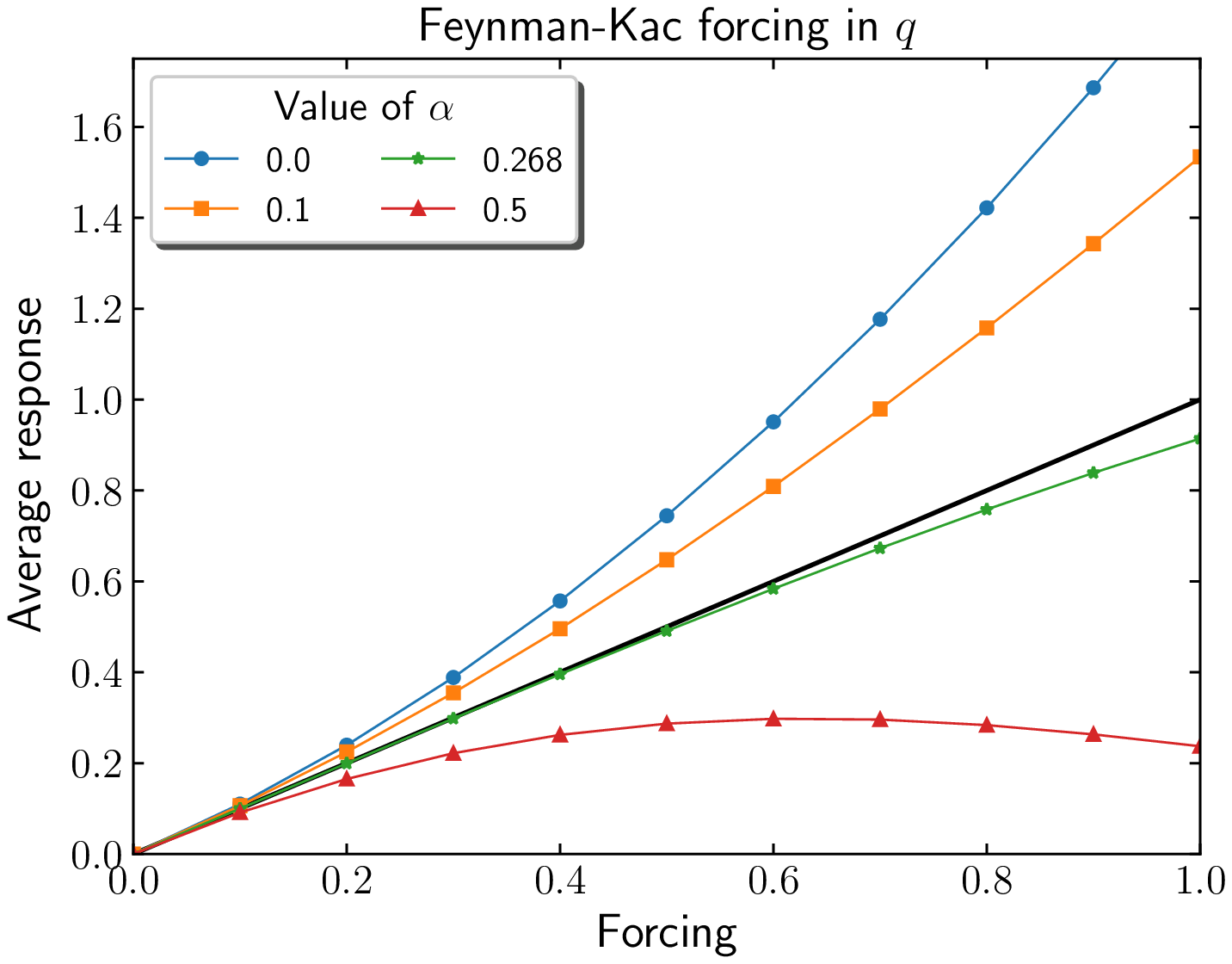}
    \caption{Feynman--Kac forcing in $q$, $\alpha_\star = 0.268$.}
    \label{fig:lang_1c}
\end{subfigure}
\hfill
\begin{subfigure}{0.49\textwidth}
    \includegraphics[width=\textwidth]{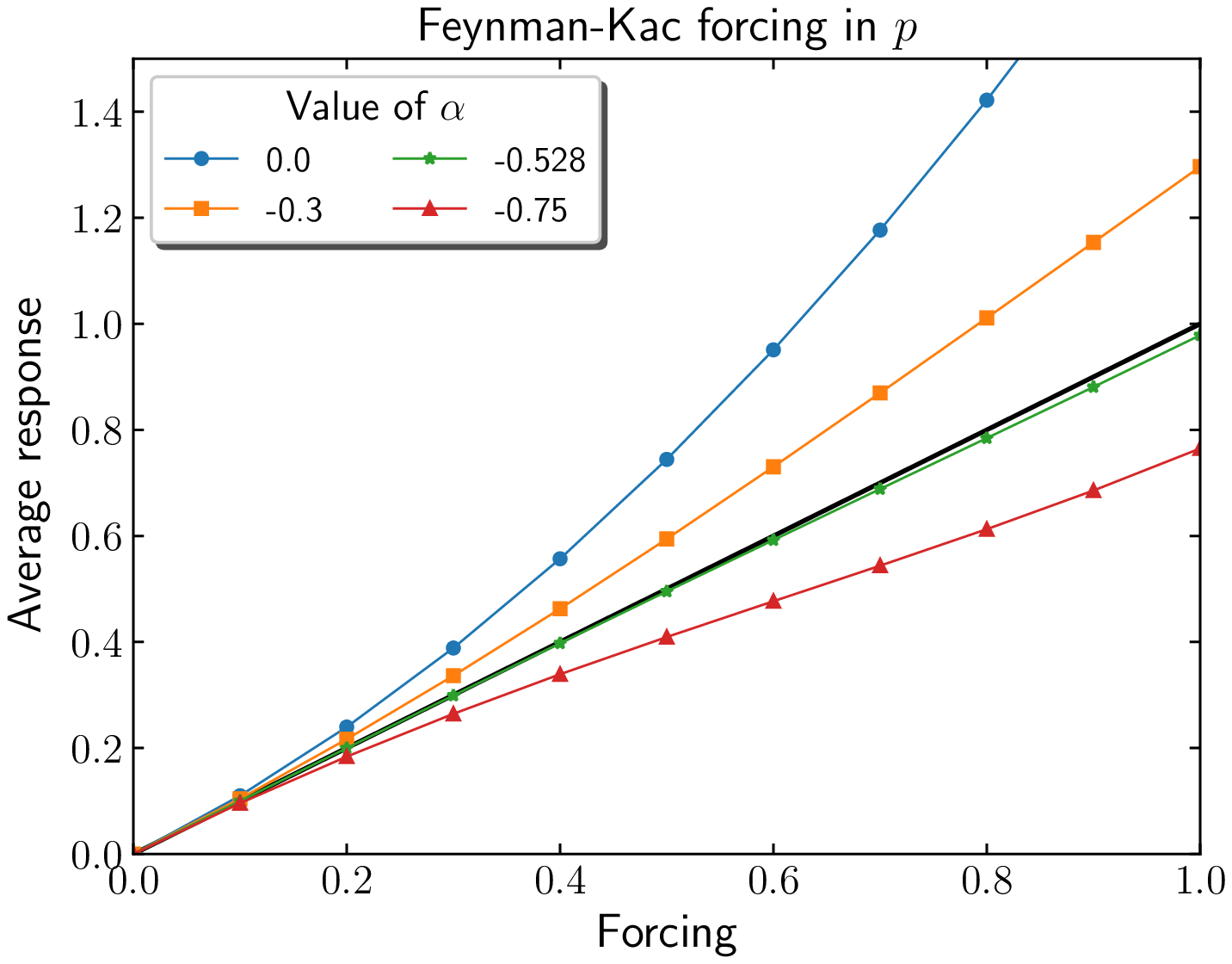}
    \caption{Feynman--Kac forcing in $p$, $\alpha_\star = -0.528$.}
    \label{fig:lang_1d}
\end{subfigure}
\hfill 
\begin{subfigure}{0.49\textwidth}
    \includegraphics[width=\textwidth]{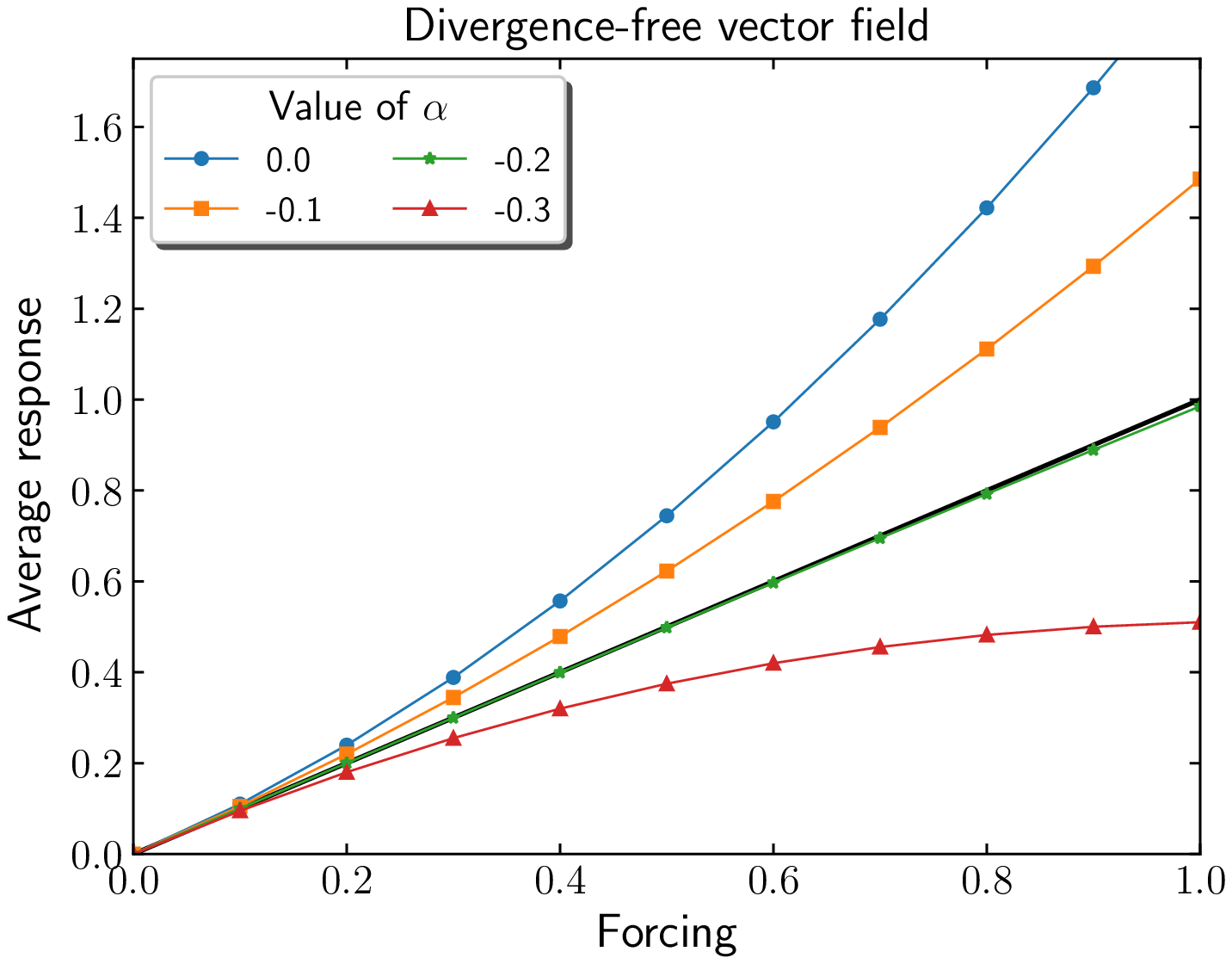}
    \caption{Div-free vector field, $\alpha_\star = -0.2$.}
    \label{fig:lang_1e}
\end{subfigure}
\hfill 
\begin{subfigure}{0.49\textwidth}
    \includegraphics[width=\textwidth]{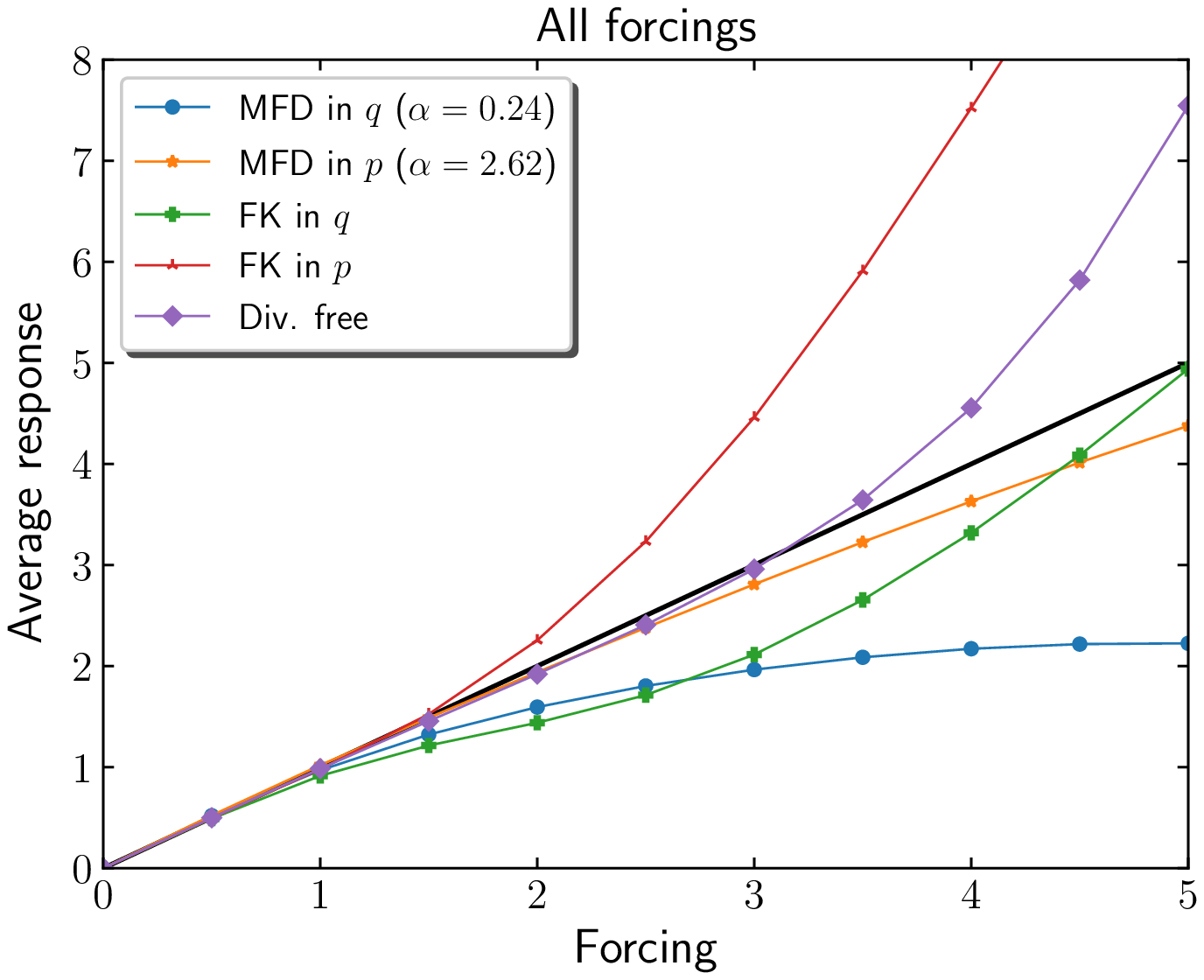}
    \caption{All forcings.}
    \label{fig:lang_1f}
\end{subfigure}
\hfill
\caption{Response curves for various synthetic forcings for the one-dimensional Langevin dynamics.}
\label{fig:fig_lang}
\end{figure}

For both modified fluctuation-dissipation forcings, Figures \ref{fig:lang_1a} and \ref{fig:lang_1b}, we see that~$\alpha^\star$ does not preserve linearity for a large $\eta$ regime, suggesting instead that $\alpha_\star(\eps)$ is the better choice for practical applications; the same conclusion was drawn for the modified fluctuation-dissipation forcing for overdamped Langevin in Section \ref{subsec:ovd_numerical_1d}. Among the two forcings, it seems better to modify the fluctuation-dissipation in $p$ as the response remains longer in the linear response regime for the optimal value $\alpha_\star(\eps)$.

Both Feynman--Kac forcings showcase great potential in extending the linear regime. Once again, Figures \ref{fig:lang_1d} and \ref{fig:lang_1f} suggest that opting for the extra forcing in the $p$ variable is the superior choice.

Lastly, the divergence-free forcing greatly increases the linear regime, as seen in Figures \ref{fig:lang_1e} and \ref{fig:lang_1f}. This behavior, as well as the ease to implement it in Monte Carlo simulations, make it is the most appealing choice of forcing, allowing to substantially reduce the variance of the estimator \eqref{estimator}, as discussed in Section \ref{subsec:num_var}. Although the Feynman--Kac extra forcing also demonstrates great potential in extending the linear regime, the challenges associated with its implementation render it an impractical choice.

\subsection{Scaling of the variance}
\label{subsec:num_var}
We discuss in this section the scaling of the variance $\sigma^2_{R,\eta}/\eta^2$, defined in \eqref{variance}, with the addition of synthetic forces, and in particular numerically illustrate the potential of synthetic forcings as a tool for variance reduction. See Appendix \ref{appendix:num_schemes} for details on the numerical computation of \eqref{variance}. We emphasize that the bias on the estimator \eqref{estimator}, made precise in \eqref{rhoHat_bias}, is negligible when compared to the variance, hence we concentrate on the variance. We anyway want to remain in the linear response regime so, by construction, the bias should be small.

Recall from Section \ref{subsec:analysis}, and in particular in Proposition \ref{proposition1}, that the estimator~$\widehat{\Phi}_{\eta, t}$ defined in \eqref{estimator} has asymptotic variance of order $\sigma^2_{R,\eta}/\eta^2$. Indeed, \eqref{eq9_sticky} suggests that this asymptotic variance is of the same order as the one associated with the equilibrium estimator, \emph{i.e.} $\sigma^2_{R,0}/\eta^2$, up to a small bias of order $\eta$:
\begin{equation}
	\frac{\sigma^2_{R,\eta}}{\eta^2} = \frac{\sigma^2_{R,0}}{\eta^2} + \bigO\left(\frac{1}{\eta}\right).
	\label{var_approx}
\end{equation}
This suggests that the asymptotic variance $\sigma^2_{R,\eta}$ has sufficiently small variations in $\eta$, which validates increasing $\eta$ as a way to substantially reduce the variance.

For each synthetic forcing and a fixed value of $\alpha$, we compute $\sigma^2_{R,\eta}/\eta^2$ with $\eta = \eta_{\alpha}(\eps)$ defined in \eqref{aOpt2}, \emph{i.e.} the first value of $\eta$ at which the response curve departs~$\eps$ in relative error relative to the linear response. For the results here presented, we use~$\eps = 0.05$. For each of the dynamics, the scaled asymptotic variance as a function of $\eta$ is illustrated in Figure \ref{fig:all_var}, where the values of $\eta_{\alpha}(\eps)$ are represented as dashed vertical lines to highlight the great reduction in variance potential. As expected, Figure \ref{fig:all_var} is consistent with \eqref{var_approx}, and it shows that increasing the regime of linear response has a dramatic effect on the variance.

In order to quantity the variance reduction, we define the gain as the ratio of the variance of the equilibrium system and its synthetic counterpart:
\begin{equation}
	\mathrm{gain} = \left(\frac{\sigma^2_{R,\eta_0(\eps)}}{\eta_0(\eps)^2}\right)\left(\frac{\sigma^2_{R,\eta_\alpha(\eps)}}{\eta_\alpha(\eps)^2}\right)^{-1}.
	\label{gain}
\end{equation}
For each of the cases presented in Table \ref{tab:var_gain}, namely the modified-fluctuation dissipation (MFD), Feynman--Kac forcing (FK) and divergence-free vector field (DF), $\alpha$ is chosen to be $\alpha^\star$ or $\alpha_\star(\eps)$, and the associated quantity $\eta_\alpha(\eps)$ is computed, where $\eta_0(\eps) = 0.05$ for all three dynamics. Note that for Langevin dynamics, the results here presented for the modified fluctuation-dissipation and the Feynman--Kac forcing both correspond to their $p$ counterparts.

\begin{figure}[tbhp]
\centering
\begin{subfigure}{0.49\textwidth}
    \includegraphics[width=\textwidth]{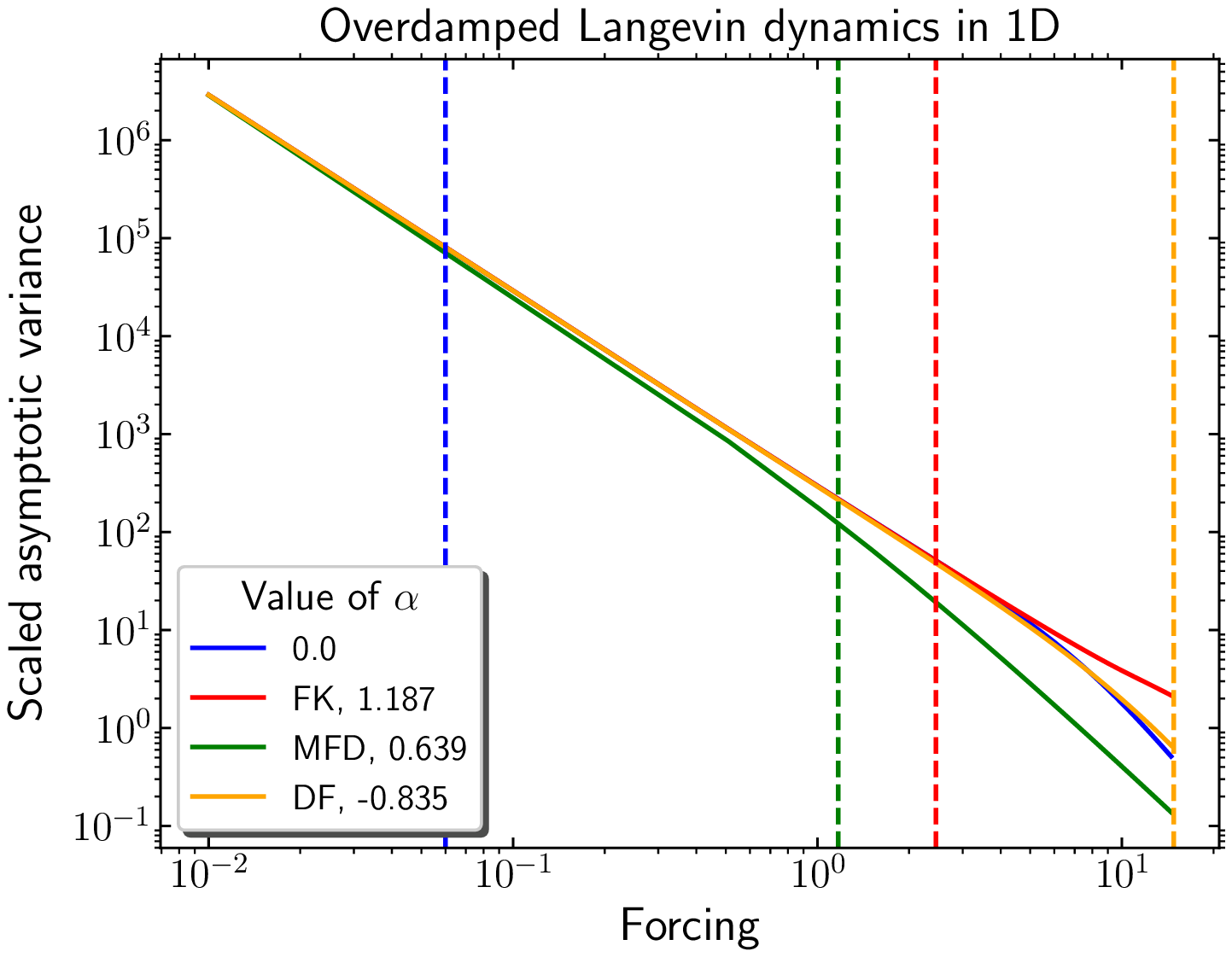}
    \caption{Overdamped Langevin dynamics in 1D.}
    \label{ovd_1d_var}
\end{subfigure}
\hfill
\begin{subfigure}{0.49\textwidth}
    \includegraphics[width=\textwidth]{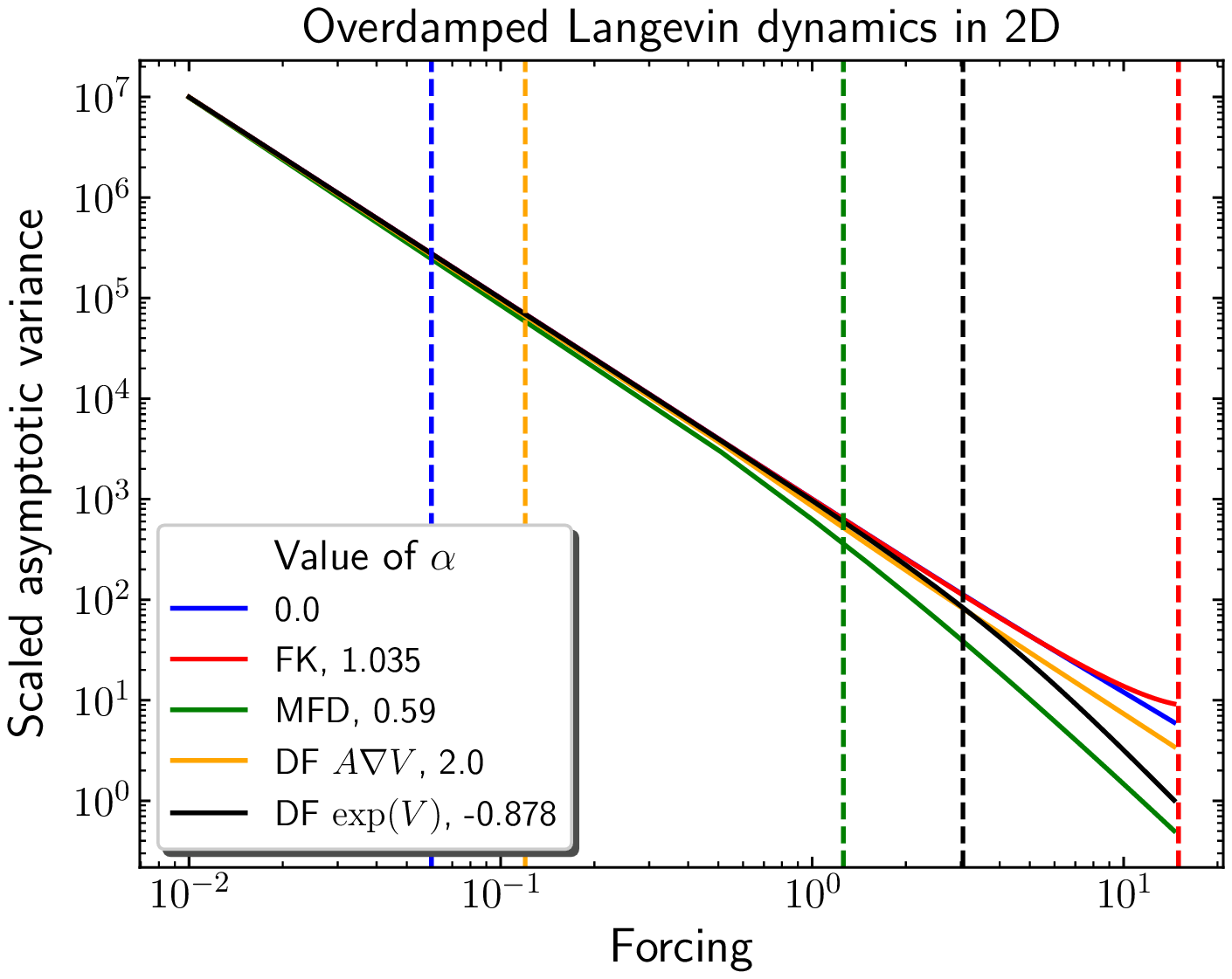}
    \caption{Overdamped Langevin dynamics in 2D.}
    \label{ovd_2d_var}
\end{subfigure}
\hfill
\begin{subfigure}{0.49\textwidth}
    \includegraphics[width=\textwidth]{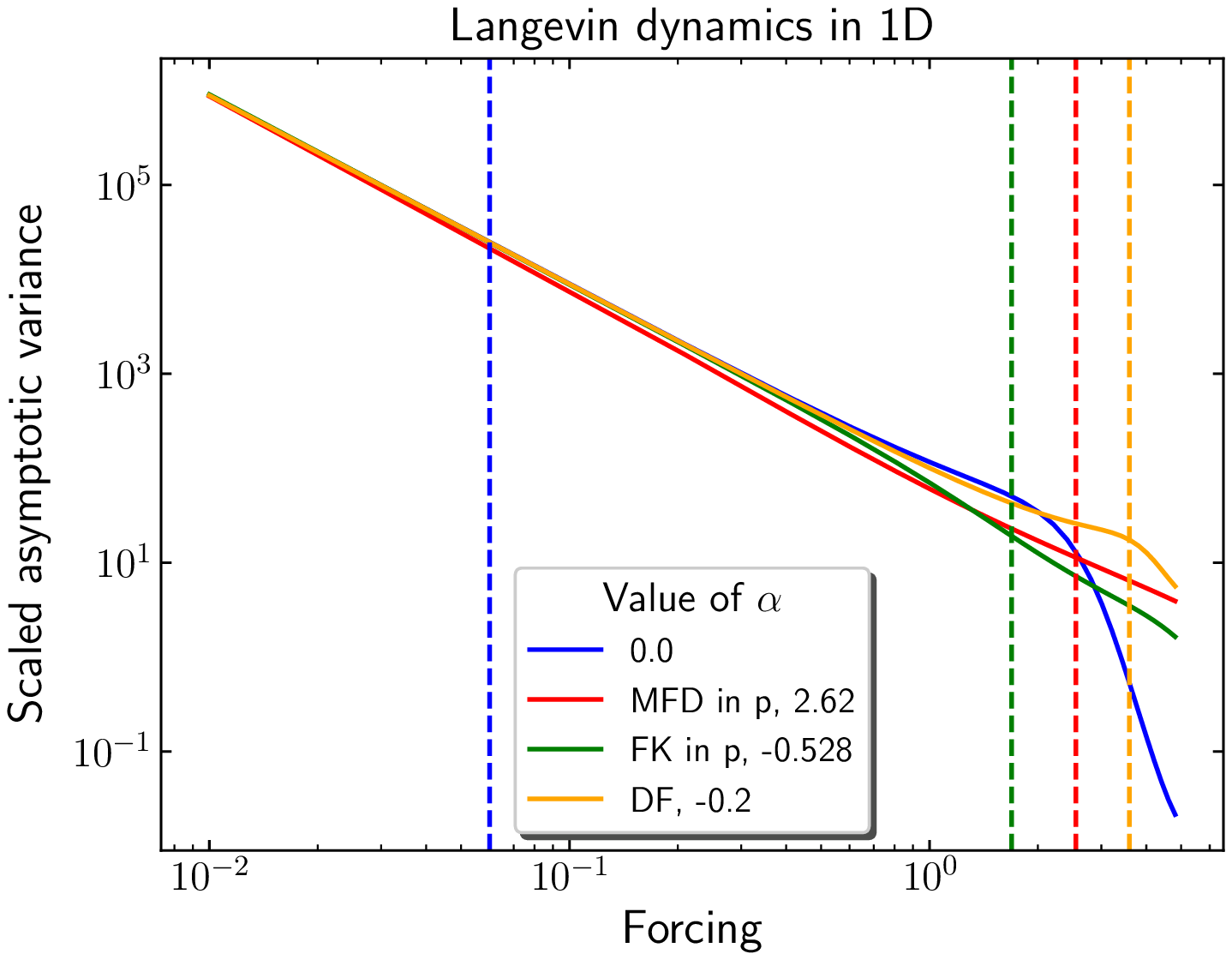}
    \caption{Langevin dynamics in 1D.}
    \label{lang_1d_var}
\end{subfigure}
\caption{Scaled asymptotic variance $\sigma^2_{R,\eta}/\eta^2$ as a function of $\eta$ for various synthetic forcings for all dynamics.}
\label{fig:all_var}
\end{figure}

\begin{table}[tbhp]
	\begin{center}
		\begin{tabular}{cccccc} 
		\hline\hline 
		\multirow{2}{*}{{\bf Dynamics}} & \multicolumn{5}{c}{{\bf Extra forcing}} \\
		\cline{2-6}
		 & none & MFD & FK & DF ($\e^V$) & DF ($A\nabla V$) \\
	    \hline 
		Ovd. 1D & 1 & $\sn{6.56}{2}$ & $\sn{1.58}{3}$ & $\sn{1.28}{5}$ & - \\
		Ovd. 2D & 1 & $\sn{7.65}{2}$ & $\sn{3.23}{4}$ & $\sn{3.33}{3}$ & $\sn{4.03}{0}$ \\ 
		Lang. 1D & 1 & $\sn{2.18}{3}$ & $\sn{1.29}{3}$ & $\sn{1.42}{3}$ & - \\
		\hline 
		\end{tabular}
	\end{center}
	\caption{Impact of synthetic forcings on variance reduction: gain \eqref{gain} for various synthetic forcings for each dynamics.}
	\label{tab:var_gain}
\end{table}

The results in Table \ref{tab:var_gain} show that the variance can be dramatically reduced with the use of synthetic forcings. Although no choice of extra forcing is universally better, we can achieve reduce the variance by a factor of over $1000$ in all cases.

\section{Extensions and perspectives}
\label{sec:perspectives}
Let us conclude this work by discussing potential extensions of our approach and practical applications to real molecular dynamics systems. The notion of synthetic forcings, as here presented, can be applied to a number of actual systems to assist in the computation of transport coefficients, \emph{e.g.} Lennard--Jones fluids for the computation of shear viscosity, or systems of atom chains for the computation of thermal transport (see \cite{lepri2003,dhar2008,lepri2016}).

The methodology here presented, however, must undergo some adaptation to be applied to actual systems. The numerical method we rely on, namely discretizing and solving the associated PDEs, does not scale well to higher dimensions and thus cannot be used as such. One typically relies on Monte Carlo simulations, which normally limits the number of values of $\eta$ used due to the high computational cost. Additionally, the computation of the optimal values of $\alpha$, presented in Section~\ref{subsec:choosing_alpha}, also cannot be done as such, as it relies on either the PDE approach, or the computation of full response curves.

One preliminary idea for the computation of $\alpha$ in real systems is the notion of \emph{prescreening}. Transport coefficients are intensive quantities, \emph{i.e.} they do not depend too drastically on the system size. This suggests, in practice, that one performs two simulations with small system sizes, with $\alpha_1 \ne \alpha_2$, from which the value of $\alpha^\star$ can be extrapolated, which can then be used in a large scale simulation. The results here presented suggest that such an adaptation, which is work in progress, is worth the effort.


\appendix
\section{Analysis of eigenvalues for Feynman--Kac forcings}
\label{app:FK_eig}
As discussed in Section \ref{subsubsec:gen_extra_forcings}, the operator $\L_{\eta,\alpha} = \Lr + \eta(\Lp + \alpha\Lx)$ admits a nonzero principal eigenvalue $\lambda_{\eta,\alpha}$ when $\Lx$ corresponds to the Feynman--Kac forcing, discussed in Examples \ref{ex:FK_ovd} and \ref{ex:FK_lang}. As a consequence, the Poisson equation \eqref{eq6} must be reformulated for such operators. In this section, we start by formally showing that $\lambda_{\eta,\alpha}$ is of order $\eta^2$, then we discuss the reformulation of \eqref{eq6}, in particular to compute $\alpha^\star$. This analysis could be made precise by adapting the approach of \cite{ferre_stoltz_2019}.

\paragraph{Quantifying the magnitude of $\lambda_{\eta,\alpha}$} For a general perturbed dynamics, we write the generator $\L_{\eta,\alpha}$ as $\L_{\eta,\alpha} = \Lr + \eta(\Lp + \alpha\Lx)$, with $\Lx = \xi^T\nabla^*$ for Feynman--Kac forcings. The associated Fokker--Planck equation is then
\begin{equation}
		\left(\Lr + \eta \Lp + \alpha\eta\Lx\right)^\dagger \psi_{\eta,\alpha} = \lambda_{\eta,\alpha}\psi_{\eta,\alpha}.
		\label{FP_FK}
\end{equation}
For a fixed value of $\alpha$, formally expanding $\psi_{\eta,\alpha}$ and $\lambda_{\eta,\alpha}$ in powers of $\eta$ yields
\begin{equation}
	\psi_\eta = \psi_0 + \eta\overline{\psi}_1 + \eta^2\overline{\psi}_{2,\alpha} + \cdots, \qquad \lambda_{\eta,\alpha} = \eta\overline{\lambda}_{1,\alpha} + \eta^2\overline{\lambda}_{2,\alpha} + \cdots.
	\label{psi_lambda_exps}
\end{equation}
Note that $\overline{\psi}_1$ has no $\alpha$ dependency, since the addition of $\alpha\Lx$ leaves $\f_1$ invariant due to \eqref{synthetic_condition}, as discussed in Section \ref{subsec:notion}. 

We substitute the expansions \eqref{psi_lambda_exps} in the Fokker--Planck equation \eqref{FP_FK}. Identifying terms with the same orders in $\eta$ leads to
\begin{align}
	\bigO(\eta):& \quad \Lr^\dagger\overline{\psi}_1 + \left(\Lp^\dagger + \alpha\Lx^\dagger\right) \psi_0 - \overline{\lambda}_{1,\alpha}\psi_0 = 0, \\
	\bigO(\eta^2):& \quad \Lr^\dagger\overline{\psi}_{2,\alpha} + \left(\Lp^\dagger + \alpha\Lx^\dagger\right) \overline{\psi}_1 - \overline{\lambda}_{1,\alpha}\overline{\psi}_1 - \overline{\lambda}_{2,\alpha}\psi_0 = 0. \label{Oeta2_exp}
\end{align}
From the $\bigO(\eta)$ expression, an integration leads to
\begin{equation}
	\overline{\lambda}_{1,\alpha}\int_\mathcal{X}	\psi_0 = \int_\mathcal{X}\Lr^\dagger \, \overline{\psi}_1 + \int_\mathcal{X} \Lp^\dagger\psi_0 + \alpha\int_\mathcal{X}\Lx^\dagger \, \psi_0.
	\label{lambda1_eq}
\end{equation}
From the definition of the $L^2$-adjoint, it follows that the first two terms on the right-hand side of \eqref{lambda1_eq} are 0, since $\Lr\ind = \Lp\ind = 0$. The remaining term can be written as
\begin{equation}
	\int_\mathcal{X}\Lx^\dagger \, \psi_0 = \int_\mathcal{X} \left(\Lx\ind\right) \psi_0 = \int_\mathcal{X}\left(\xi^T\nabla^*\ind\right) \psi_0 = \int_\mathcal{X}\xi^T\nabla\ind \, \psi_0 = 0.
\end{equation}
Thus, we conclude that $\overline{\lambda}_{1,\alpha} = 0$. Similarly, integrating \eqref{Oeta2_exp} on $\mathcal{X}$ leads to
\begin{equation}
	\overline{\lambda}_{2,\alpha} = \alpha\int_\mathcal{X}\Lx^\dagger \, \overline{\psi}_1 = \alpha\int_\mathcal{X}\left(\Lx\ind\right) \overline{\psi}_1,
	\label{lambda2}
\end{equation}
which is generally a nonzero quantity.
%
%
This suggests that
\begin{equation}
	\lambda_{\eta,\alpha} = \alpha\eta^2\int_\mathcal{X} \left(\Lx\ind\right) \overline{\psi}_1 + \bigO(\eta^3).
\end{equation}

\paragraph{Reformulating the Poisson equation} Computing the optimal value $\alpha^\star$ defined in~\eqref{aOpt1} requires computing the second-order response, as discussed in Section \ref{subsec:choosing_alpha}. In particular, this is done by using the recursive expression \eqref{eq6}, obtained from the Fokker--Planck equation via formal asymptotics. We reformulate \eqref{eq6} for Feynman--Kac systems to compute the second-order response, using that, in view of \eqref{Oeta2_exp},
\begin{equation}
	\overline{\psi}_{2,\alpha} = \left(-\Lr^{-1}\right)^\dagger\left[\left(\Lp^\dagger + \alpha\Lx^\dagger\right) \overline{\psi}_1 - \overline{\lambda}_{2,\alpha}\psi_0\right],
\end{equation}
where the value of $\overline{\lambda}_{2,\alpha}$ is determined by \eqref{lambda2}. This procedure could be extended to arbitrary orders by identifying terms with the same orders in $\eta$ in \eqref{FP_FK}.

\section{Numerical scheme for finite differences}
\label{appendix:num_schemes}
We describe here the numerical schemes used to discretize Fokker--Planck equations, used to solve for the density of the invariant probability measure for the systems at hand, namely overdamped Langevin dynamics in one and two dimensions, and Langevin dynamics in one dimension. This approach can be straightforwardly extended to solving Poisson equations with nontrivial right-hand sides such as \eqref{poissonf1} or the integrand in the asymptotic variance \eqref{variance}. We start each section by first discussing the discretizations associated with the equilibrium dynamics, after which we mention the modification needed when adding perturbations.

For all systems, the spatial domain, namely the torus $\T^d = [0,1)^d$, is discretized into $m_q^d$ points with uniform step size $h_q = 1/m_q$ in each direction. While $\psi_\eta$ denotes the invariant measure at the continuous level, we use $\Psi_\eta^h$ to denote its discretized counterpart, namely the approximations of the values of $\psi_\eta$ at the grid points.

\subsection{Overdamped Langevin dynamics}
\label{subapp:ovd_1d}
For a continuous function $u\colon\T\to\R$, we denote by $\dfxn{u}{i} = u(Q_i)$, where $Q_i = ih_q$ are the mesh points. Similarly, for a continuous function $u\colon\T^2\to\R$, we denote by $\dfxn{u}{i,j} = u(Q_{i,j})$, where $Q_{i,j} = (ih_q, jh_q)$ are the mesh points. The Fokker--Planck equation to discretize for the overdamped Langevin dynamics~\eqref{eq_ovd} reads
\begin{equation}
	\Lr^\dagger\psi_0 = \div(\nabla V\psi_0) + \frac{1}{\beta}\Delta\psi_0 = \nabla V^T \nabla\psi_0 + \Delta V\psi_0 + \frac{1}{\beta}\Delta\psi_0.
	\label{fp_ovd}
\end{equation}
Using centered finite differences, the discretization of \eqref{fp_ovd} in dimension one is given by
\begin{equation}
	\dfxn{V'}{i}\frac{\dpsi{0}{i+1} - \dpsi{0}{i-1}}{2h_q} + \dfxn{V''}{i}\dpsi{0}{i} + \frac{\dpsi{0}{i+1} - 2\dpsi{0}{i} + \dpsi{0}{i-1}}{\beta h_q^2} = 0.
	\label{disc_fp_ovd_1d}
\end{equation}
Periodic boundary conditions are imposed, \emph{i.e.} $\dpsi{0}{0} = \dpsi{0}{m_q}$. In two dimensions, applying centered finite differences to the Fokker--Planck \eqref{fp_ovd} yields
\begin{equation}
\begin{aligned}
	\dfxn{\partial_{q_1}V}{i,j}\frac{\dpsi{0}{i+1,j} - \dpsi{0}{i-1,j}}{2h_q} + \dfxn{\partial_{q_2}V}{i,j}\frac{\dpsi{0}{i,j+1} - \dpsi{0}{i,j-1}}{2h_q} + \dfxn{\Delta V}{i,j} \dpsi{0}{i,j}
	\\ + \frac{\dpsi{0}{i+1,j} + \dpsi{0}{i-1,j} - 4\dpsi{0}{i,j} + \dpsi{0}{i,j+1} + \dpsi{0}{i,j-1}}{\beta h_q^2} = 0.
\end{aligned}
\label{disc_fp_ovd_2d}
\end{equation}
Periodic boundary conditions are imposed, \emph{i.e.} $\dpsi{0}{0,:} = \dpsi{0}{m_q,:}$ and $\dpsi{0}{:,0} = \dpsi{0}{:,m_q}$.

\paragraph{Discretization of perturbations} We now consider several classes of perturbation operators. One typical form is $\Lt = F^T\nabla$, with $F = (F_1,F_2) \in \R^2$ for two-dimensional dynamics. Its $L^2$-adjoint acts as
\begin{equation}
	\Lt^\dagger\psi_0 = -\div(F)\psi_0 - F^T\nabla\psi_0.
	\label{perturbation1}
\end{equation}
Perturbations of this form include the physical perturbations $\Lp$ considered throughout this work, divergence-free vector fields \eqref{divFree}, and the differential term in the Feynman--Kac forcing \eqref{fk_ovd}. In dimension one, the-right hand side of \eqref{perturbation1} is discretized with a centered finite difference as
\begin{equation}
	-\dfxn{F'}{i}\dpsi{0}{i} - \dfxn{F}{i}\frac{\dpsi{0}{i+1} - \dpsi{0}{i-1}}{2h_q}.
\end{equation}
In dimension two, perturbations of the form \eqref{perturbation1} are discretized with centered finite differences as
\begin{equation}
\begin{aligned}
	-\left(\dfxn{F_1}{i,j}\frac{\dpsi{0}{i+1,j} - \dpsi{0}{i-1,j}}{2h_q} + \dfxn{F_2}{i,j}\frac{\dpsi{0}{i,j+1} - \dpsi{0}{i,j-1}}{2h_q}\right) \\
	- \dfxn{\partial_{q_1}F_1}{i,j}\dpsi{0}{i,j} - \dfxn{\partial_{q_2}F_2}{i,j}\dpsi{0}{i,j}.
\end{aligned}
\end{equation}

Another example is the modified fluctuation-dissipation perturbation \eqref{mfdr_ovd}, preceeded by a factor $\alpha\eta$ as presented in this work. Since it corresponds to the generator~\eqref{ovd_Lr} of the dynamics at hand, its discretization corresponds in any dimension to rescaling the discretized Fokker--Planck by $(1+\alpha\eta)$.

We also consider zero order operators, such as the source term for the Feynman--Kac forcing \eqref{fk_ovd}, namely $\xi^T\nabla V$. As terms of this form include no differential operators acting on $\psi_0$, their discretization is trivially done, in any dimension, by direct evaluation, \emph{e.g.} $\dfxn{V'}{i}\dpsi{0}{i}$ in dimension one.


\subsection{Langevin 1D}
For the one-dimensional Langevin dynamics, we discretize the unbounded momentum space as follows: we first truncate it to $[-p_\mathrm{max}, p_\mathrm{max}]$, then discretize it into $m_p$ interior points with uniform step size $h_p = 2p_\mathrm{max}/(m_p-1)$. For a continuous function $u\colon \T\times\R \to \R$, we denote by $\dfxn{u}{i,j} = u(Q_i, P_j)$, where $Q_i = ih_q$ and $P_j = jh_p$ are the mesh points.

The numerical scheme for the Fokker--Planck equation \eqref{lang_FP_L2} for Langevin dynamics is obtained with centered finite differences, except for the transport term $p^TM^{-1}\nabla_q\psi_0$, where an upwind scheme is used (see below for details). In dimension one, \eqref{lang_FP_L2} reads
\begin{equation}
    \Lr^\dagger\psi = -\L_\mathrm{ham}\psi + \gamma\left(M^{-1}\psi + M^{-1}p\partial_p\psi + \beta^{-1}\partial^2_p\psi\right).
    \label{lang_FP_1D}
\end{equation}
The discretization of \eqref{lang_FP_1D} then reads
\begin{equation}
\begin{aligned}
    \dfxn{V'}{i}\frac{\dpsi{0}{i,j+1} - \dpsi{0}{i,j-1}}{2h_p} - \frac{P^+_j\ddpsi{0}{i,j}^- + P^-_j\ddpsi{0}{i,j}^+}{M} + \frac{\gamma P_j}{M}\frac{\dpsi{0}{i,j+1} - \dpsi{0}{i,j-1}}{2h_p} \\ + \frac{\gamma}{M\beta}\frac{\dpsi{0}{i,j+1} - 2\dpsi{0}{i,j} + \dpsi{0}{i,j-1}}{h_p^2} + \frac{\gamma \dpsi{0}{i,j}}{M} = 0,
\end{aligned}
\end{equation}
where 
\begin{equation}
	p^+ = \max(p,0), \qquad p^- = \min(p,0),
	\label{p_uw_fd}
\end{equation}
and
\begin{equation}
	\ddpsi{0}{i,j}^+ = \frac{\dpsi{0}{i+1,j} - \dpsi{0}{i,j}}{h_q}, \qquad \ddpsi{0}{i,j}^- = \frac{\dpsi{0}{i,j} - \dpsi{0}{i-1,j}}{h_q}.
\end{equation}
Periodic boundary conditions are imposed on the positions, \emph{i.e.} $\dpsi{0}{0,j} = \dpsi{0}{m_q,j}$ for any $1\leq j\leq m_p$.

\paragraph{Discretization of perturbations} For Langevin dynamics, perturbations of the form~\eqref{perturbation1} can correspond to differential operators acting on the positions or momenta, \emph{i.e.} $F^T\nabla_q$ and $F^T\nabla_p$. In dimension one, these options reduce to $-F\partial_q\psi_0$ and $-F\partial_p\psi_0$, both of which are trivially discretized with centered finite differences.

For the modified fluctuation-dissipation in $p$, its discretization corresponds to scaling the $\L_\mathrm{FD}$ term in \eqref{lang_gen_split} by $(\gamma+\alpha\eta)$. For its position counterpart $-\beta^{-1}\nabla^*_q\nabla_q$, its discretization corresponds to adding \eqref{disc_fp_ovd_1d}, scaled by $\alpha\eta$.

\paragraph{Upwind scheme for Langevin} We now motivate the upwinding used to discretize the transport term $p^TM^{-1}\nabla_q$ in the Fokker--Planck equation \eqref{lang_FP_L2}. In particular, the use of a centered finite difference with an even number of points $m_q$ in the spatial domain would lead to independent submeshes, and hence might give incorrect results, an issue known as odd-even decoupling (see \cite{odd_even_decoupling} for a thorough discussion). To overcome this, one resorts to decentered finite differences, in particular upwinding. 

Recall that equation \eqref{general_FP} can be seen as the stationary solution to the evolution PDE $\partial_t\psi = \L^\dagger\psi$. To implement the upwinding scheme, we write the transport term in the Fokker--Planck in the form of the advection equation, that is
\begin{equation}
	\partial_t \psi + p^TM^{-1}\nabla_q\psi = 0.
\end{equation}
The scheme used to discretize the components of $\nabla_q\psi$ depends on the sign of the components of $p$. The (partial) derivative is discretized with a scheme decentered on the left when $p\geq 0$, and decentered on the right when $p\leq 0$. The discretization of $p\partial_q \psi$ in dimension one at a mesh point $(Q_i,P_j)$ is therefore done as follows:
\begin{equation}
	\begin{cases}
		p\dfrac{\dpsi{0}{i+1,j} - \dpsi{0}{i,j}}{h_q}, \qquad p<0,  \\
		p\dfrac{\dpsi{0}{i,j} - \dpsi{0}{i-1,j}}{h_q}, \qquad p>0.
	\end{cases}
\end{equation}

\section*{Acknowledgments}
This project has received funding from the European Union's Horizon 2020 research and innovation program under the Marie Sklodowska--Curie grant agreement No 945332, and from the European Research Council (ERC) under the European Union's Horizon 2020 research and innovation programme (project EMC2, grant agreement No 810367). We also acknowledge funding from the Agence Nationale de la Recherche, under grants ANR-19-CE40-0010-01 (QuAMProcs) and ANR-21-CE40-0006 (SINEQ). 

\bibliographystyle{siamplain}
\bibliography{references}
\end{document}